\documentclass[11pt,reqno]{amsart}
\usepackage[title]{appendix}
\usepackage{times}
\usepackage{amsmath,amsfonts, amstext,amssymb,amsbsy,amsopn,amsthm}
\usepackage[initials,alphabetic]{amsrefs}
\usepackage{mathrsfs}
\usepackage{bm}
\usepackage{dsfont}
\usepackage{esint}
\usepackage{graphicx}   
\usepackage{hyperref}
\usepackage[all]{xy}
\usepackage{relsize}
\usepackage{mathtools}
\usepackage{dsfont}
\usepackage{graphicx}   
\usepackage{hyperref}
\usepackage{tikz}

\newtheorem{theorem}{Theorem}[section]

\newtheorem{proposition}[theorem]{Proposition}
\newtheorem{lemma}[theorem]{Lemma}
\newtheorem{corollary}[theorem]{Corollary}
\theoremstyle{definition}
\newtheorem{definition}[theorem]{Definition}
\theoremstyle{remark}
\newtheorem{remark}{Remark}[section]
\theoremstyle{remark}

\theoremstyle{remark}

\newtheorem{claim}{Claim}[section]
\theoremstyle{remark}

\theoremstyle{remark}

\theoremstyle{remark}

\begin{document}

\title{Sharp Estimate of Global Coulomb Gauge}
\date{\today}
\author{Yu Wang}
\address{Department of Mathematics, Northwestern University, Evanston, Il 60208, USA}
\email{yuwang2018@u.northwestern.edu}

\maketitle
\begin{abstract}
Let $A$ be a $W^{1,2}$-connection on a principle $\text{SU}(2)$-bundle $P$ over a compact $4$-manifold $M$ whose curvature $F_A$ satisfies $\|F_A\|_{L^2(M)}\le \Lambda$. Our main result is the existence of a global section $\sigma: M\to P$ with finite singularities on $M$ such that the connection form $\sigma^*A$ satisfies the Coulomb equation $d^*(\sigma^*A)=0$ and admits a sharp estimate $\|\sigma^*A\|_{\mathcal{L}^{4,\infty}(M)}\le C(M,\Lambda)$. Here $\mathcal{L}^{4,\infty}$ is a new function space we introduce in this paper that satisfies $L^4(M)\subsetneq \mathcal{L}^{4,\infty}(M)\subsetneq L^{4-\epsilon}(M)$ for all $\epsilon>0$. More precisely, $\mathcal{L}^{4,\infty}(M)$ is the collection of measurable function $u$ such that $\|u\|_{\mathcal{L}^{4,\infty}(M)}\eqqcolon\|1/s_u\|_{L^{4,\infty}(M)}<\infty$, where $L^{4,\infty}$ is the classical Lorentz space and $s_u$ is the $L^4$ integrability radius function associated to $u$ defined by $s_{u}(x)=\sup\Big\{r:\sup_{y\in B_{r}(x)}\int_{B_r(y)}|u|^4 dV_g\le 1\Big\}.$

\end{abstract}
\tableofcontents

\section{Introduction}
Let $A$ be a $W^{1,2}$-connection on a principle bundle $P$ with compact Lie group $G$ fiber over a compact $4$-manifold $M$. In \cite{U82a}, K. Uhlenbeck proved if $\|F_A\|_{L^2(M)}<\epsilon_0$ for some constant $\epsilon_0$, then one could find a Coulomb gauge $\sigma$, that is $d^*(\sigma^*A)=0$, such that $\|\sigma^*A\|_{L^4(M)}$ is controlled by $\|F_A\|_{L^2(M)}$. However, without the smallness hypothesis of $\|F_A\|_{L^2(M)}$, such a gauge $\sigma$ does not need to exist that $\|\sigma^*A\|_{L^4(M)}$ is controlled in terms of $\|F_A\|_{L^2}$ due to topological obstructions. Indeed, one can prove the non-existence of such a gauge if the second Chern number $c_2$ of the bundle is non-trivial; see \cite{PR14}. In this paper, by focusing on $\text{SU}(2)$ bundles, we construct a global Coulomb gauge with sharp estimate. More precisely, we find a global section $s: M\to P$ such that $d^*(s^*A)=0$ and $s^*A$ is controlled in the following space (which we denote by $\mathcal{L}^{4,\infty}$ in order to have it distinguished from the classical $L^{4,\infty}$ space):
\begin{definition}\label{curlyl4infty}
Given any measurable function $u$ defined on a bounded domain $U$ in a Euclidean space, then for all $x\in U$ let us set 
\begin{eqnarray}
\begin{split}
\label{4rad}
s_{u}(x)=\sup\Big\{r:\sup_{y\in B_{r}(x)}\int_{B_r(y)}|u|^4 d\mu\le 1\Big\}.
\end{split}
\end{eqnarray}
We define
\begin{eqnarray}
\begin{split}
\|u\|_{\mathcal{L}^{4,\infty}(U)}\eqqcolon\Big\|\frac{1}{s_u}\Big\|_{L^{4,\infty}(U)}.
\end{split}
\end{eqnarray}
\end{definition}
\begin{remark}
The second ``$\sup$" in the big braces is clearly redundant, and dropping it does not create essential changes to our proofs. However, we include it in order to make the definition consistent with the more general definitions of regularity radius, which will be used later in the paper and where the second ``$\sup$" cannot be dropped (see Definition \ref{rads}).
\end{remark}
\begin{remark}
\label{clanew}
We will prove in Appendix that 
\begin{eqnarray}
\begin{split}
\label{clanew1}
&\|u\|_{\mathcal{L}^{4,\infty}(M)}\le C_0(M)\|u\|_{L^{4}(U)},\\
&\int_{M}|u|^{4-\epsilon}\le C(\epsilon,M)\|u\|^4_{\mathcal{L}^{4,\infty}(M)}
\end{split}
\end{eqnarray}
for all $\epsilon>0$ and $C_0(M),C(\epsilon,M)<\infty$. Especially this implies $L^4(M)\subsetneq \mathcal{L}^{4,\infty}(M)\subsetneq L^{4-\epsilon}(M)$. See Lemma \ref{interpo}. 
\end{remark}
By the discussions preceding Definition \eqref{curlyl4infty} combined with Remark \ref{clanew}, we see that the $\mathcal{L}^{4,\infty}(M)$-estimate of $s^*A$ is indeed sharp. 
\medskip

It is worth mentioning that in \cite{PR14}, M. Petrache and T. Riviere constructed a global gauge $\sigma$ (possibly with singularities) such that $\|\sigma(A)\|_{L^{4,\infty}(M)}\le C(M,\|F_A\|_{L^2(M)})$. Further, in Open Problem 1.3 of \cite{PR14} the authors asked whether or not there exists a $L^{4,\infty}$-controlled global Coulomb gauge. Despite of the inclusion relationships given in Remark \ref{clanew}, $\mathcal{L}^{4,\infty}$ is not stronger (nor weaker) than $L^{4,\infty}$. In other words, our main result does not yet give an affirmative answer to the Open Problem 1.3 of \cite{PR14}. Nevertheless, since $\mathcal{L}^{4,\infty}$ is defined by using the $L^4$-integrability radius, it carries an extra piece of information compared to $L^{4,\infty}$, and hence it is a more robust space from analysts' viewpoint. Furthermore, in this paper we obtain a variety of estimates as intermediate results, including an $\epsilon$-regularity theorem (for general group fiber principal bundles as opposed to $\text{SU}(2)$ bundles), which all seem likely to be usefully applied in the related contexts.
\medskip

\subsection{Main results}
Before stating the main theorem, let us explain some terminologies. Firstly, a (possibly singular) trivialization of $P$ is a map $\sigma: M\to P$ satisfying $\pi\circ \sigma=\text{Id}$, where $\pi$ is the projection map defined on $P$. Secondly, given a connection $A$ and a trivialization $\sigma$ on $P$, then the associated connection form is a Lie algebra valued $1$-form defined by $A_0=\sigma^*A$. Thirdly, if the connection form satisfies $d^*A_0=-g^{ij}\iota_{\partial_i}\nabla_{\partial_j}A_0=0$ weakly in $M$, we say that $A_0$ is weakly Coulomb. We now state our main result:
\begin{theorem}
\label{main}
Let $P\to M$ be an $\text{SU}(2)$-principal bundle over a compact $4$-manifold and let $A$ be a $W^{1,2}$ connection on $P$ satisfying $\|F_A\|_{L^2(M)}<\Lambda$ for some $\Lambda<\infty$. Then there exists a constant $C_0(M,\Lambda)$ and a global $P$-section $s_0$ such that its associated connection form $s_0^*A$ is weakly Coulomb and satisfies $\|s_0^*A\|_{\mathcal{L}^{4,\infty}(M)}\le C_0(M,\Lambda)$.        
\end{theorem}
\begin{remark}
To be more precise than the above statement, we will prove that $s_0^*A$ is $L^4$-integrable away from controllably many points on $M$. When $A$ is smooth, without much more work than this paper one can show that the global section $s_0$ is indeed smooth (but in an ineffective fashion) away from controllably many points. 
\end{remark}
\begin{remark}
We remark that with a little work the constant $C_0(M,\Lambda)$ could be made as precise as 
$$C_0(\text{sec}_M,\text{inj}_M,\Lambda),$$
where $\text{sec}_M=\sup_{M}|\text{sec}|, \text{inj}_M=\inf_{M}\text{inj}$ denote the sectional curvature bound of $M$ and the injectivity radius of $M$ respectively. 
\end{remark}
\begin{remark}
\label{singhajime}
Later, we say $x\in M$ is a singularity of $A_0$ if $s_{|A_0|}(x)=0$; see Definition \ref{4rad}.
\end{remark}
\medskip

\subsection{Preliminaries and terminologies}
Let $P$ be a principle bundle with simply-connected simple compact Lie group $G$ fiber, say $G\subseteq \text{SO}(k)$. The right action of a group element $g\in G$ on a bundle element $u\in P$ will be denoted by $ug$. We will narrow down to the $\text{SU}(2)$-bundles later on. Consider the left action of $\text{SO}(k)$ on $\mathbb{R}^k$ given by the matrix multiplication, and the $\mathbb{R}^k$-fiber bundle $E$ that is associated to $P$ via such an action. Then $E^k$ is isomorphic to a $\text{Mat}_{\mathbb{R}}(k\times k)$-fiber bundle associated to $P$ via left multiplication of matrices. More precisely, the action of a bundle element $u\in P$ on a matrix $\mathcal{M}\in \text{Mat}_{\mathbb{R}}(k\times k)$ denoted by $\lfloor u, \mathcal{M}\rfloor$ satisfies $\lfloor ug, \mathcal{M}\rfloor=\lfloor u, g\cdot \mathcal{M}\rfloor$ for all $g\in G$, where $g\cdot \mathcal{M}$ is the standard matrix multiplication. 
Next, recall that a section of $E^k$ (denoted by $\tilde{\sigma}$) is identified with a $\text{Mat}_{\mathbb{R}}(k\times k)$-valued tensorial function on $P$ (denoted by $\sigma$) in the following way (see \cite{KN} for details):
\begin{eqnarray}
\begin{split}
\label{tens}
\tilde{\sigma}(x)=\lfloor u, \sigma(u)\rfloor,\ \text{for all }u\in \pi^{-1}(x).
\end{split}
\end{eqnarray}
Now for arbitrarily fixed $x\in M$, consider a subset of the fiber $E^k_{x}$ given by $G_x=\lfloor u, G\rfloor$ for some $u\in P_x$. First let us notice that $G_x$ does not depend on the choice of $u$. Indeed, given any $u^{\prime}\in P_x$, there exists $g\in G$ such that $u^{\prime}=ug$. Using \eqref{tens} we have that $\lfloor u^{\prime},G\rfloor=\lfloor ug,G\rfloor=\lfloor u,g\cdot G\rfloor=\lfloor u,G\rfloor$. This leads to the following definition of frame.
\begin{definition}
Let $S$ be an open subset of $M$. We say $\Theta\in \Gamma(E^k)$ is a frame on $P$ over $S$, if $\Theta(x)\in G_x$ for almost all $x\in S$.
\end{definition}
Using \eqref{tens} again, the next proposition identifies the frames with the $P$-sections.
\begin{proposition}
\label{ggauge}
Let $\Theta$ be a section of $E^k$ with its associated tensorial $\text{Mat}_{\mathbb{R}}(k\times k)$-valued function $\tilde{\Theta}$, and $S$ be an open subset of $M$. Then $\Theta$ is a frame over $S$ if and only if there exists a section $\sigma:S\to P$ with $\Theta(\sigma(x))\equiv1_G$ (i.e. the identity matrix) for a.e. $x\in S$. 
\end{proposition}
\begin{proof}
Let us start with a frame $\Theta$ over $S$. Since $\tilde{\Theta}$ is tensorial and $G$-valued, we see that $s_x=p\tilde{\Theta}(p)$ is independent of $p\in P_x$. Obviously, we have $\tilde{\Theta}(s_x)=1_G$. By defining $\sigma$ by $\sigma(x)=s_x$, we then complete the proof of one direction. Conversely, if there exists a section $\sigma:S\to P$ such that $\Theta(\sigma(x))\equiv1_G$ for each $x\in S$, then by writing each $p\in P_x$ as $p=\sigma(x)g$ for some $g\in G$ we see that $\tilde{\Theta}(p)=g^{-1}\in G$. In other words, $\tilde{\Theta}$ is $G$-valued and therefore $\Theta$ is a frame over $S$.
\end{proof}

When $M$ is a compact $4$-dimensional manifold, it is well known that there exists a frame over $M$ that is smooth away from finitely many points, such that the sum of the degree at each singularity equals to the second Chern number $c_2$. See also discussions in Subsection \ref{prep}. Any frame $\Theta_0\in \Gamma(E^k)$ could be uniquely written as $(U_1,\cdots,U_k)$, where $U_i\in \Gamma(E)$ for each $i$. Now define 
\begin{eqnarray}
\begin{split}
A^{\alpha}_{\beta,i}=\langle(\nabla_AU_{\beta})(\partial_i),U_{\alpha}\rangle_E.
\end{split}
\end{eqnarray}
An elementary argument shows that when $\Theta_0$ is a frame, the $\text{Mat}_{\mathbb{R}}(k\times k)$-valued $1$-form given by $A_*=\bigg{(}A^{\alpha}_{\beta,i}dx^i\bigg{)}_{\alpha,\beta}$ is a $\mathfrak{g}$-valued $1$-form, where $\mathfrak{g}$ is the Lie algebra of $G$. 
\begin{definition}
\label{asso}
$A_0=\bigg{(}A^{\alpha}_{\beta,i}dx^i\bigg{)}_{\alpha,\beta}\in \Lambda^1(M,\mathfrak{g})$ is called the associated connection form associated to $\Theta$. 
\end{definition}
Furthermore, we say a frame $\Theta_0$ is a Coulomb frame, if its associated connection form $A_0$ satisfies
\begin{eqnarray}
\begin{split}
d^*A_0=-g^{ij}\iota_{\partial_i}\nabla_{\partial_j}A_0=0
\end{split}
\end{eqnarray}
weakly on $M$. Using the terminologies introduced in this subsection, we are ready to give an outline of the proofs and techniques of Theorem \ref{main}.

\subsection{Outline of the proofs and techniques}
In order to prove Theorem \ref{main}, we need to first find a Coulomb frame $\Theta_0\in \Gamma(E^k)$, and then prove that its associated connection form $A_0\in \Lambda^1(M,\mathfrak{g})$ satisfies the desired estimate
\begin{eqnarray}
\begin{split}
\|A_0\|_{\mathcal{L}^{4,\infty}(M)}\le C(M,\|F_A\|_{L^2(M)}).
\end{split}
\end{eqnarray}
To find a Coulomb frame is not difficult at all. Indeed we simply minimize the functional $\int_M|\nabla_A\Theta|^2dV_g$ over all $W^{1,2}$ frames $\Theta$, and obtain a minimizing sequence $\Theta_i$. It is an easy consequence of the lower semi continuity of the $W^{1,2}$ norm under the weak $W^{1,2}$-convergence that a minimizer exists. From now on let us fix a minimizer $\Theta_0$, and denote its associated connection form by $A_0$. By using the Euler-Lagrangian equation satisfied by $\Theta_0$ it is not hard to show that $d^*A_0=0$ weakly. In other word, $\Theta_0$ is a Coulomb frame. Hence, let us refer to $\Theta_0$ as a Coulomb minimizer. Nevertheless, no a-priori control over $A_0$ or $\Theta_0$ was known to exist. The goal of this paper is to show that $A_0$ admits the desired $\mathcal{L}^{4,\infty}$ (see Definition \ref{curlyl4infty}) estimate when $G=\text{SU}(2)$. For the sake of the outline, let us assume that $A$ is smooth, and the proof for general $W^{1,2}$-connections requires simply a standard approximation argument, which will be given in Section \ref{asc}.
\medskip

Let us begin by introducing a fundamental tool in tackling this problem, namely the following $\epsilon$-regularity theorem (see also Theorem \ref{eps}): 
\begin{theorem}
\label{epsintro}
There exists $\epsilon_0(G), \eta_0(G), r(M)$, such that if $\epsilon_0<\epsilon(G,M), \eta_0<\eta_0(G,M), r<r(G,M)$, and 
\begin{eqnarray}
\begin{split}
\label{2smallness}
\int_{B_r(x)}|F_A|^2<\epsilon_0, r^{-2}\int_{B_r(x)}|A_0|^2<\eta_0,
\end{split}
\end{eqnarray}
then the following estimate holds:
\begin{eqnarray}
\begin{split}
\int_{B_{r/2}(x)}|\nabla_A \Theta_0|^4+\int_{B_{r/2}(x)}|\nabla_A\nabla_A \Theta_0|^2\le C\bigg{(}\int_{B_r(x)}|F_A|^2+\big{(}r^{-2}\int_{B_r(x)}|A_0|^2\big{)}\bigg{)}.
\end{split}
\end{eqnarray}
\end{theorem}
Indeed, imagine for the moment that the curvature $F_A$ vanishes. Then $\Theta_0$ becomes a classical minimizing harmonic map into $G$. R. Schoen and K. Uhlenbeck (\cite{SU82}) proved an $\epsilon$-regularity theorem which claims that the smallness of the scaling invariant energy of a minimizing map implies the effective interior smoothness. Naturally, one expects $\Theta_0$ to behave like a minimizing map into $G$ when the curvature is small, and to possess an analogous $\epsilon$-regularity property as well, which is exactly what Theorem \ref{eps} concludes. Nevertheless, in the presence of a curvature that is small in $L^2$ but non-vanishing, the effective $L^4$ and $W^{1,2}$ estimates on $A_0$ is the best that one can achieve, instead of the effective smoothness as in Schoen and Uhlenbeck's $\epsilon$-regularity theorem. Roughly speaking, the reason is that one should not expect $A_0$ to possess better regularity than $W^{1,2}$ simply based on the knowledge of $L^2$ of $F_A$, which one should think of as the derivative of $A_0$ heuristically. For the sake of future convenience, we say a point $x_0$ is a singularity of $\Theta_0$ if $|\nabla\Theta_0|\equiv |A_0|$ is not locally $L^4$ integrable at $x_0$; in other word, $x_0$ is a singularity if $s_{|A_0|}(x)=0$. Now we are ready to present the outline of the proof of $\|A_0\|_{\mathcal{L}^{4,\infty}(M)}\le C_0(M,\Lambda)$, and from now on we restrict our attention to the $\text{SU}(2)$ bundles. As mentioned in Remark \ref{singhajime}, we define the singular set $S(\Theta_0)$ as $\{x\in M: s_u(x)=0\}$.
\medskip

To begin, let us focus on proving the following theorem; see also Theorem \ref{4localeps}:
\begin{theorem}
\label{4localepsintro}
There exists $\epsilon_0(G)$, $r(M)$, and $C_0$ such that if for all $B_{4r}(x_0)\subseteq M$ with $10r<r(M)$ and $\int_{B_{4r}(x_0)}|F_A|^2<\epsilon_0$, then 
\begin{eqnarray}
\begin{split}
\label{c1epsilon}
\|A_0\|_{\mathcal{L}^{4,\infty}(B_{r}(x_0))}\le C_0.
\end{split}
\end{eqnarray}
\end{theorem} 
The smallness of $\epsilon_0$ is always subject to further decrements along the outline. We also note that in this theorem we remove the smallness condition on $r^{-2}\int_{B_r(x)}|A_0|^2$ as in the $\epsilon$-regularity theorem \ref{epsintro}. Correspondingly, instead of $L^4$-estimate on $A_0$ we achieve the weaker estimate $\mathcal{L}^{4,\infty}$. The proof of Theorem \ref{4localepsintro} would require two key ingredients. The first is the following theorem; see also Theorem \ref{apr} for its full statement: 
\begin{theorem}
\label{radbelow}
There exist positive constants $\epsilon_0(G)$ and $c_0\le10^{-4}$, such that if $\int_{B_{2r}(y_0)}|F_A|^2<\epsilon_0$ and $\text{Sing}(\Theta_0)\cap B_{2r}(y_0)=\emptyset$, then for each $x\in B_{r}(y_0)$ it holds $s_{|A_0|}(x)>c_0r$.
\end{theorem}
Roughly speaking, this theorem says that the ineffective $L^4$-regularity implies the effective $L^4$-estimate. The second key ingredient is the singular structure analysis in Section \ref{tconestru}, summarized in the following theorems:
\begin{theorem}
$\text{Sing}(\Theta_0)$ consists of isolated singularities. Furthermore, at each singularity, a tangent cone is given by $U:\mathbb{R}^4\backslash\{0\}\to \mathbb{S}^3$, $x\mapsto T(\frac{x}{|x|})$ for some $T\in O(3)$.
\end{theorem}
\begin{theorem}
\label{number0}
There exists $\epsilon_0$, $r(M)$ and $N_0\ge 10$, such that if $10\rho_0<r_M$ and $\int_{B_{2r}(y_0)}|F_A|^2<\epsilon_0$, then $\#\{\text{Sing}(\Theta_0)\cap B_{3r/2}(y_0)\}\le N_0$.
\end{theorem}
Combining these two theorems we see that not only $\Theta_0$ has at most isolated singularities, but the number of singularities is controllable (provided the curvature being small in $L^2$). Now, by combining Theorem \ref{radbelow} and \ref{number0}, it is not hard to conclude Theorem \ref{4localepsintro}. 
\medskip

Next, we construct a decomposition of $M$ into controllably many annular regions and bubble regions. To be precise, a bubble region $\mathcal{B}$ looks like $B_{r}(x)\backslash \bigcup_{i=1}^N B_{r_i}(x_i)$, where $N\le N(\Lambda)$, $r_i\ge c_0(\Lambda)r$, and furthermore:
\begin{eqnarray}
\begin{split}
r_{F_A}(y)\equiv\sup\{s>0:\sup_{z\in B_{s}(y)}\int_{B_s(z)}|F_A|^2<\epsilon_0\}\ge c_0(\Lambda) r.
\end{split}
\end{eqnarray}
where $\epsilon_0$ is some universal small number to be specified later in Section \ref{annbub}. An annular region $\mathcal{A}$ looks like $A_{s,r}(x)$, where $s/r\ll 1$ and $\int_{A_{s,r}(x)}|F_A|^2<\epsilon_0$ for the same universal constant $\epsilon_0$. For a more precise description of a bubble region and an annular region, see Section \ref{annbub}. We then prove the following decomposition theorem:
\begin{theorem}
\label{decomp0}
Suppose $\int_M|F_A|^2\le \Lambda$. Then there exists a large number $N(M,\Lambda)$, such that one could find a cover of $M$ by annular and bubble regions given by $\{\mathcal{A}_{i}\}_{i=1}^{N_1}\cup\{\mathcal{B}_j\}_{j=1}^{N_2}$ satisfying $N_1+N_2\le N(M,\Lambda)$.
\end{theorem}

Firstly, consider a bubble region $\mathcal{B}=B_{r}(x)\backslash \bigcup_{i=1}^N B_{r_i}(x_i)$. Trivially, we can cover $\mathcal{B}$ by at most $Cc_0^{-4}$ many $B_{r_{F_A,\epsilon_0}(x_l)}(x_l)$, and by \eqref{c1epsilon}, there exists a universal constant $C_0$ such that the following holds for all $l$:
\begin{eqnarray}
\begin{split}
\|A_0\|_{\mathcal{L}^{4,\infty}(B_{r_{F_A}(x_i)}(x_i))}\le C_0.
\end{split}
\end{eqnarray}
Upon summing up the above inequality over all $l$ and using the covering property of $\{B_{r_{F_A}(x_l)}(x_l)\}_l$, we achieve
\begin{eqnarray}
\begin{split}
\label{bubble0}
\|A_0\|_{\mathcal{L}^{4,\infty}(\mathcal{B})}\le C_1,
\end{split}
\end{eqnarray}
for some universal constant $C_1$. 

Secondly, let us focus on an annular region $\mathcal{A}=A_{s,r}(x)$. Note that since $\frac{s}{r}\ll 1$ could be arbitrarily small, it is no longer true that $r_{F_A,\epsilon_0}(x)\ge c_0 r$ for all $x\in \mathcal{A}$ and some universal $c_0>0$. Therefore, we cannot directly apply \eqref{c1epsilon}. Nevertheless, by identifying $\text{SU}(2)$ with $\mathbb{S}^3$, and using the singular structure theory for stable-stationary harmonic maps into $\mathbb{S}^3$ developed in \cite{LW06} combining with a somewhat complicated limiting argument, we are able to prove the following lemma:
\begin{lemma}
\label{singintro}
There exists a constant $N_0(\Lambda)$, such that $\#\{\mathcal{A}\cap \text{Sing}(\Theta_0)\}\le N_0$.
\end{lemma}
Applying Lemma \ref{singintro} and Theorem \ref{radbelow}, we obtain
\begin{eqnarray}
\begin{split}
\label{neck0}
\|A_0\|_{\mathcal{L}^{4,\infty}(\mathcal{A})}\le C_2
\end{split}
\end{eqnarray}
for some universal constant $C_2$. \medskip

Finally, combining Theorem \ref{decomp0}, \eqref{bubble0}, and \eqref{neck0}, we achieve the desired estimate $\|A_0\|_{\mathcal{L}^{4,\infty}(M)}\le C_0(M,\Lambda)$. This finishes the outline.
\bigskip

\section*{Acknowledgement}
The author would like to thank his advisor, Prof. Aaron Naber, for his constant support, patient guidance, and those inspiring discussions with him.
\bigskip

\section{Preparation}
In this section, we will find a Coulomb frame $\Theta_0$ by minimizing the energy functional $\int_M |\nabla \Theta|^2$ among all frames over $M$ (see Subsection \ref{prep}), and then we provide some of its basic properties (see Subsection \ref{basicmin}), which will be used in later sections to obtain effective estimates on the associated connection form of $\Theta_0$.

\subsection{Coulomb minimizer}
\label{prep}
The goal of this subsection is to find a Coulomb frame for $A$. Let $P\to M$ a $G$ bundle over a compact $4$-manifold with a smooth connection $A$. Consider the following subclass of $W^{1,2}$ frames:
\begin{equation}
\mathcal{F}=\{\Theta:\ \text{$\Theta$ is a frame over $M$ s.t.}\ \int_M|\nabla_A\Theta|^2<\infty\}.
\end{equation} 
First note that the class $\mathcal{F}$ is none-empty. Indeed, when $G$ is a simply-connected simple Lie group, it is a classical topological result that there exists a global $P$-section $\Sigma_0$ that is continuous away from $k$ points, where the second Chern number $c_2=k$ or $-k$. Moreover the winding number at each singularity is $1$ or $-1$ depending on the sign of $c_2$. Without loss of generality let us assume $c_2=k$. Then upon standard mollification away from the singularities of $\Sigma_0$, and then a smooth perturbation at the neighborhood of each singularity, one achieves a new section $\Sigma_1$ smooth away from $k$ points. Now, by employing Proposition \ref{ggauge} we find the frame $\Theta_{\Sigma_1}$ associated to $\Sigma_1$. It is not hard to make $\int_{M}|\nabla_A \Theta_{\Sigma_1}|^2<\infty$ by choosing a nice enough smooth perturbation $\Sigma_1$ from before. Therefore $\Theta_{\Sigma_1}\in \mathcal{F}$.
\medskip

Given the non-emptiness, we could minimize the functional $\int_{M}|\nabla_A \Theta|^2$ in the class $\mathcal{F}$. Indeed, let $\{\Theta_i\}_i$ be a minimizing sequence in $\mathcal{F}$. Then due to the weakly compactness of $W^{1,2}(M,E^k)$ and the compact embedding of $L^2(M,E^k)$ into $W^{1,2}(M,E^k)$, we conclude that $\{\Theta_i\}$ converge to some $\Theta_0$ weakly in $W^{1,2}$ and strongly in $L^2$. Thus $\Theta_0$ is a frame. Further, by the lower-semicontinuity of the $W^{1,2}$-norm under the weak $W^{1,2}$-convergence, $\Theta_0$ is a minimizer of the functional $\int_{M}|\nabla_A \Theta|^2$. Especially, $\Theta_0$ satisfies $\int_M|\nabla_A \Theta_0|^2<\infty$, and therefore $\Theta_0\in \mathcal{F}$. 
\begin{definition}
\label{coumindef0}
A minimizer of $\int_{M}|\nabla_A \Theta|^2$ among all elements in $\mathcal{F}$ is called a Coulomb minimizer.
\end{definition}
By the previous discussions, a Coulomb minimizer always exists and belongs to the class $\mathcal{F}$. From now on, let us fix a Coulomb minimizer denoted by $\Theta_0$ with associated connection form denoted by $A_0$; see Definition \ref{asso}. 
\medskip

\subsection{Basic properties of Coulomb minimizers}
\label{basicmin}
In this subsection, we give a couple of very basic properties of the Coulomb minimizer $\Theta_0$, which includes the Euler-Lagrange equation, the stationarity, and the stability.
\medskip

Firstly, due to the minimizing property of $\Theta_0$, an Euler-Lagrange equation holds. More precisely, we have the following equivalences:
\begin{lemma}
\label{coulomb}
Let $P\to M$ be a $G$-bundle over a compact $4$-manifold with a smooth connection $A$. Let $\Theta_0=(U_1,\cdots,U_k)$ be a $W^{1,2}$-frame over $M$, with associated connection form $A_0$. Then the following conditions are equivalent:
\\
(1) $d^*A_0=0$ weakly. That is, $A_0$ is weakly Coulomb.\\
(2) $d_A^*d_A U_a\eqqcolon\Delta_{A}U_a=\sum_b\langle\nabla_{A}U_a,\nabla_{A}U_b\rangle_{\Lambda^1(E)} U_b$ weakly.\\
(3) For any smooth section $\xi$ of the adjoint $\mathfrak{g}$-bundle $\text{ad}P$, the following holds:
\begin{equation}
\label{lagr}
\frac{d}{dt}\bigg{\rvert}_{t=0}\sum_a\int_M|\nabla_A(\exp(t\xi)U_a)|^2\equiv\frac{d}{dt}\bigg{\rvert}_{t=0}\int_M|\nabla_A(\exp(t\xi)\Theta_0)|^2=0.
\end{equation}
\end{lemma}
\begin{remark}
\label{tohigh0}
As one could see from the proof presented below, the dimension assumption is unimportant. However, we state Lemma \ref{coulomb} only for $4$-manifolds since we will not be considering the higher dimensional cases in this paper.
\end{remark}
\begin{proof}
Choose any smooth section $\Phi\in \Gamma(E)$ as the test function in proving the weak identities. In the trivialization induced by $\Theta=(U_1,\cdots,U_k)$, both $\{U_a\}_a$ and $\Phi$ admit trivializations denoted by $\{E_a\}_a$ and $\phi$ respectively.
\medskip

(1) $\Longrightarrow$ (2): Given (1), we have
\begin{eqnarray}
\begin{split}
\label{force1}
\int_M \langle \nabla_A U_a,\nabla_A \Phi\rangle=\int_M\langle A_0(E_a),d\phi+A_0(\phi)\rangle=\int_M\langle A_0(E_a),d\phi\rangle +\int_M\langle A_0(E_a),A_0(\phi)\rangle.
\end{split}
\end{eqnarray}
The first term vanishes due to (1); i.e. $d^*A_0=0$ weakly. Hence above becomes
\begin{eqnarray}
\begin{split}
\int_M\langle A_0(E_a),A_0(\phi)\rangle=\int_M\langle A_0(E_a),A_0(\sum_b\langle E_b,\phi\rangle E_b)\rangle=\int_M\sum_b\langle\nabla_{A}U_a,\nabla_{A}U_b\rangle \langle U_b,\Phi\rangle.
\end{split}
\end{eqnarray}
Thus (2) is true.

(2) $\Longrightarrow$ (1): Given (2), we have
\begin{eqnarray}
\begin{split}
\label{force2}
\int_M \langle \nabla_A U_a,\nabla_A \Phi\rangle=\int_M \sum_b\langle\nabla_{A}U_a,\nabla_{A}U_b\rangle \langle U_b,\Phi\rangle.
\end{split}
\end{eqnarray}
By \eqref{force1}, we see that \eqref{force2} forces $\int_M\langle A_0(E_a),d\phi\rangle =0$. Namely, (1) is true.

(1) $\Longleftrightarrow$ (3): let $\xi$ be any smooth section of the adjoint $\mathfrak{g}$-bundle $\text{ad}P$ with trivialization $\xi_0$ in the gauge induced by $\Theta=(U_1,\cdots,U_k)$. We have
\begin{eqnarray}
\begin{split}
&\frac{d}{dt}\bigg{\rvert}_{t=0}\sum_a\int_M|\nabla_A(\exp(t\xi)U_a)|^2=\frac{d}{dt}\bigg{\rvert}_{t=0}\sum_a\int_M|\nabla_{\exp(t\xi_{0})^*A_{0}}E_a|^2\\
=&\frac{d}{dt}\bigg{\rvert}_{t=0}\sum_a\int_M|[\exp(-t\xi_{0})A_{0}\exp(t\xi_{0})](E_a)+\exp(-t\xi_{0})d(\exp(t\xi_{0}))(E_a)|^2\\
=&2\sum_a\int_M\langle\frac{d}{dt}\bigg{\rvert}_{t=0}\bigg{\{}[\exp(-t\xi_{0})A\exp(t\xi_{0})](E_a)+\exp(-t\xi_{0})d(\exp(t\xi_{0}))(E_a)\bigg{\}},A_{0}(E_a)\rangle\\
=&2\sum_a\int_M\langle d\xi_{0}(E_a)+[-\xi_{0},A_{0}](E_a),A_{0}(E_a)\rangle=2\int_M\sum_a\langle d\xi_{0}(E_a),A_{0}(E_a)\rangle\\
+&2\int_M\sum_a\langle [-\xi_{0},A_{0}](E_a),A_{0}(E_a)\rangle=2\int_M \langle d\xi_{0},A_{0}\rangle+2\int_M \langle [A_{0},\xi_{0}],A\rangle=2\int_M \langle d\xi_{0},A_{0}\rangle,
\end{split}
\end{eqnarray}
where in the penultimate equality we have used the fact that $\{U_a\}$ forms an orthonormal basis. Also, the second term in the last identity vanishes because $\langle [A_{0},\xi_{0}],A_{0}\rangle=\langle A_{0}, [\xi_{0},A_{0}]\rangle=-\langle  [A_{0},\xi_{0}],A_{0}\rangle$. Now, combining both sides of the above computation we have
\begin{eqnarray}
\begin{split}
\frac{d}{dt}\bigg{\rvert}_{t=0}\sum_a\int_M|\nabla_A(\exp(t\xi)U_a)|^2=2\int_M \langle d\xi_{0},A_{0}\rangle.
\end{split}
\end{eqnarray}
In view of the fact that $A_0\in \Lambda^1(M,\mathfrak{g})$, we see that (1) and (3) are equivalent.
\end{proof}
By the minimizing property of $\Theta_0$, we see that condition (3) in Lemma \ref{coulomb} holds for $\Theta_0$. Now we apply Lemma \ref{coulomb} to see that both condition (1) and condition (2) hold for $\Theta_0$. Especially, the connection form of $A$ under a Coulomb minimizer is weakly Coulomb.
\medskip

Another crucial property of $\Theta_0$ is the stability, which is also a consequence of the minimizing property. More precisely, for any smooth section $\xi$ of the adjoint $\mathfrak{g}$-bundle $\text{ad}P$, the following inequality holds:
\begin{equation}\label{stab1}
\frac{d^2}{dt^2}\bigg{\rvert}_{t=0}\sum_a\int_M|\nabla_A(\exp(t\xi)U_a)|^2\equiv \frac{d^2}{dt^2}\bigg{\rvert}_{t=0}\int_M|\nabla_A(\exp(t\xi)\Theta_0)|^2\ge0.
\end{equation}
The stability will be studied in further details later in Section \ref{sta}.
\begin{remark}
Using a limiting argument, one sees that both \eqref{lagr} and \eqref{stab1} remain true for a Coulomb minimizer $\Theta_0$ when the test function $\xi$ has only $W_0^{1,2}$ regularity.
\end{remark}
\medskip

Thirdly, being a minimizer implies that $\Theta_0$ is indeed a critical point of the functional $\int_M|\nabla_A\Theta|^2$, in the sense that not only does it satisfy the Euler-Lagrange equation (see Lemma \ref{coulomb}, condition (2)), but also the stationarity equation that comes from the domain variation. To be precise, let $\{\zeta_t(x)\}_{t\ge 0}$ be an arbitrary smooth $1$-parameter family of diffeomorphisms on $M$, such that $\zeta_0\equiv id$. Denote by $P_{\zeta_t(x),x}$ the unique parallel transport associated to $E^k$ of an element on the fiber at $\zeta_t(x)$ to an element on the fiber at $x$ via the curve $\zeta_t(x)$ from $t$ to $0$. By the minimizing property we have
\begin{equation}
\frac{d}{dt}\bigg{\rvert}_{t=0}\int_M|\nabla_A P_{\zeta_t(x),x}(\Theta(\zeta_t(x)))|^2=0.
\end{equation}
An elementary but somewhat tedious calculation shows that the above to be equivalent to the following:
\begin{eqnarray}
\begin{split}
\label{divfree}
\int_M \bigg{(}|\nabla_A \Theta_0|^2g_{ik}-2\langle \nabla_{A,i}\Theta_0,\nabla_{A,k}\Theta_0\rangle \bigg{)}g^{ij}{X^k}_{;j}+2\int_M\langle {F_i}^a(\Theta_0),\nabla_{A,a}\Theta_0\rangle X^i=0.
\end{split}
\end{eqnarray}
where $X$ is any smooth vector field on $M$ it holds. \eqref{divfree} is called the stationarity equation of $\Theta_0$. In this paper, the stationarity equation will only be applied to sufficiently small balls, whose metric is close to being Euclidean. For the sake of brevity, we shall assume the metric on the ball to be Euclidean whenever we apply \eqref{divfree}, since the modification necessary for the general metric case is purely technical and unessential to the matter. Now choose $X$ to be any compactly supported smooth vector field on a ball $B_r(x)$ with metric $g_{ij}=\delta_{ij}$. By using \eqref{divfree}, and following a standard procedure (see, for example, Section 2.2 of \cite{S96}), one could derive the following formula:
\begin{equation}
\label{mono}
\frac{d}{d\rho}\bigg{(}\rho^{-2}\int_{B_{\rho}(y)}|\nabla_A\Theta_0|^2\bigg{)}=2\frac{d}{d\rho}\bigg{(}\int_{B_{\rho}(y)}|x-y|^{-2}|\nabla_{A,\partial r} \Theta_0|^2\bigg{)}+2\rho^{-3}\int_{B_{\rho}(y)}\langle {F_i}^a,\nabla_{A,a}\Theta_0\rangle(x^i-y^i).
\end{equation}
Indeed, \eqref{mono} becomes exactly the monotonicity formula of a stationary harmonic map (see Section 2.4 of \cite{S96}) if the curvature $F_A$ vanishes.
\bigskip

\section{The $\epsilon$-regularity theorem}
\label{ert}
In this Section we will prove an $\epsilon$-regularity theorem which allows us to obtain the effective $L^4$-estimate of $A_0$ in $B_r(x)$ provided the smallness of both $\int_{B_{2r}(x)}|F_A|^2$ and $r^{-2}\int_{B_{2r}(x)}|A_0|^2$. By identifying $|A_0|$ as $|\nabla_A\Theta_0|$ and viewing $\Theta_0$ as a twisted minimizing harmonic map into the fiber group $G$, our result extends Schoen and Uhlenbeck's $\epsilon$-regularity theorem of classical minimizing harmonic maps (see Section 2.3 of \cite{S96}). Let us now present the main theorem in this section:

\begin{theorem}
\label{eps}
Let $P\to M$ be a smooth principal bundle with simply-connected simple compact Lie group $G$-fiber over a compact $4$-manifold and $A$ a smooth connection on $P$. Let $\Theta_0$ be a Coulomb minimizer obtained in Subsection \ref{prep}, with associated connection form $A_0$. Then, there exist $\epsilon_0(G)$, $\eta(G)$, $C(G)$ and $r(M)$, such that for $10\rho_0<r(M)$, if $B_{10\rho_0}(p)\subseteq M$ satisfies that $\int_{B_{10\rho_0}(p)}|F_A|^2<\epsilon_0$ and ${\rho_0}^{-2}\int_{B_{10\rho_0}(p)}|A_0|^2<\eta$, then it holds 
\begin{equation}
\label{hess}
\int_{B_{\rho_0/2}(y_0)}|A_0|^4+\int_{B_{\rho_0/2}(y_0)}|\nabla A_0|^2\le C\bigg{(}\big{(}{\rho_0}^{-2}\int_{B_{10\rho_0}(p)}|A_0|^2\big{)}^2+\int_{B_{10\rho_0}(p)}|F_A|^2\bigg{)}.
\end{equation}
\end{theorem}
\begin{remark}
In Schoen and Uhlenbeck's $\epsilon$-regularity theorem, the monotonicity of $\rho^{-2}\int_{B_{\rho}(y)}|\nabla u|^2$ is of crucial importance. Hence, the main difficulty of proving our $\epsilon$-regularity theorem lies in the fact that $\rho^{-2}\int_{B_{\rho}(y)}|\nabla_A\Theta_0|^2$ is not necessarily a monotone quantity due to the presence of the curvature term $2\rho^{-3}\int_{B_{\rho}(y)}\langle {F_i}^a,\nabla_{A,a}\Theta_0\rangle(x^i-y^i)$ in \eqref{mono} that is not even integrable in $\rho$. 
\end{remark}
\begin{proof}
It is a well known result of Uhlenbeck \cite{U82b} that under the smallness assumption that 
$$\int_{B_{\rho}(p)}|F_A|^2\le \epsilon_{Uh}$$ 
one could find a local Coulomb frame on $B_{\rho}(p)$ with associated connection form $A^*$ satisfying
\begin{eqnarray}
\begin{split} 
\label{uhgauge}
&d^*A^*=0\ \text{weakly in $B_\rho(p)$},\ \ast A^*\rvert_{\partial B_\rho(p)}=0,\\
&\|A^*\|_{L^{4}(B_{\rho}(p))}+\|A^*\|_{W^{1,2}(B_{\rho}(p))}\le C_{Uh}\|F_A\|_{L^2(B_{\rho}(p))}.
\end{split}
\end{eqnarray}
From now on, we fix such a frame on the ball $B_{10\rho_0}(p)$ denoted by $\Theta_{Uh}$ while referring to it as the Uhlenbeck-gauge, and denote the associated connection form by $A^*$. Moreover, we let $\Theta^*$ be a mapping from $B_{10\rho_0}(p)$ into $\text{Mat}_{\mathbb{R}}(k\times k)$ defined by the following identity: 
\begin{equation} 
\label{trivialuh}
\Theta_0\circ\sigma_{Uh}=\Theta^*
\end{equation}
where $\sigma_{Uh}$ is the $P$-section associated to $\Theta_{Uh}$ as in Proposition \ref{ggauge} satisfying $\Theta_{Uh}\circ\sigma_{Uh}\equiv 1_G$. We shall refer to $\Theta^*$ as the trivialization of the Coulomb minimizer $\Theta_0$ with respect to the Uhlenbeck-gauge $\Theta^*$. Using the equivalent condition (2) of Lemma \ref{coulomb} satisfied by $\Theta_0=(U_1,\cdots, U_k)$, and a straightforward computation, we see that the following identity holds in the weak sense:
\begin{equation}
\label{weaklap}
0=d^*d\Theta^*-\Theta^*\langle(d\Theta^*)^T\cdot d\Theta^*\rangle-\Theta^*\langle(d\Theta^*)^T\cdot A^*(\Theta^*)\rangle+\Theta^*\langle (A^*(\Theta^*))^T\cdot d\Theta^*\rangle.
\end{equation}
Next, define $\lambda_{y,\rho}=\rho^{-4}\int_{B_{\rho}(y)}\Theta^*(y)dy$, which is an element in $\text{Mat}_{\mathbb{R}}(k\times k)$. We begin by presenting several lemmas, which will be crucial to the proof of Theorem \ref{eps}. Note that by choosing $\epsilon$ sufficiently small in these lemmas, the curvature $L^2$ smallness hypothesis always guarantees the existence of the Uhlenbeck-gauge described as above. Hence we will be using the notations $A^*$ and $\Theta^*$ without explanations of the notations.

\begin{lemma}
\label{prerepo}
For fixed $\Lambda>0$, there exist $\epsilon_0<\epsilon_{Uh}$, $\eta_0$, $C_0$ and $r(M)$, such that for any $\delta\in (0,1)$ and $B_{10\rho_0}(y_0)\subseteq M$ with $\rho_0<r(M)$, if $\int_{B_{10\rho_0}(y_0)}|F_A|^2<\epsilon_0$, $\rho_{0}^{-2}\int_{B_{10\rho_0}(y_0)}|d\Theta^*|^2<\Lambda$, and $\rho_0^{-4}\int_{B_{10\rho_0}(y_0)}|\Theta^*-\lambda_{y_0,10\rho_0}|^2<\eta_0$, then for all $y\in B_{8\rho_0}(y_0)$ the following holds:
\begin{equation}
\label{seesaw}
\rho_0^{-2}\int_{B_{\frac{\rho_0}{2}}(y)}|d\Theta^*|^2\le C_0\bigg{(}\delta\rho_0^{-2}\int_{B_{\rho_0}(y)}|d\Theta^*|^2+\delta^{-1}\rho_0^{-4}\int_{B_{\rho_0}(y)}|\Theta^*-\lambda_{y,\rho_0}|^2+\rho_0^{-2}\int_{B_{10\rho_0}(y_0)}|A^*|^2\bigg{)}.
\end{equation}
\end{lemma}
\begin{remark}
Later $\Lambda$ will be specified and taken to be a universal constant.
\end{remark}
Let us now define:
\begin{equation}
\mathcal{Q}\eqqcolon \{B_{\sigma}(z): B_{2\sigma}(z)\subseteq B_{4\rho_0}(y_0),\ \sigma\in[\frac{\rho_0}{10},\rho_0 ].\}
\end{equation}
Set $Q=\sup_{B_{\sigma}(z)\in \mathcal{Q}}\sigma^2\int_{B_{\sigma}(z)}|d\Theta^*|^2$. 
\begin{lemma}
\label{repo}
For fixed $\Lambda>0$, there exist $\epsilon_0$, $\eta_0$ small enough (which will also fulfill Lemma \ref{prerepo}) and constant $C_1$ (depending on $C_0$), if $\int_{B_{10\rho_0}(y_0)}|F_A|^2<\epsilon_0$, $\rho_0^{-2}\int_{B_{10\rho_0}(y_0)}|d\Theta^*|^2<\Lambda$, $\rho_0^{-4}\int_{B_{10\rho_0}(y_0)}|\Theta^*-\lambda_{y_0,10\rho_0}|^2<\eta_0$, and
\begin{equation}
\label{largeness}
\sigma^2\int_{B_{\sigma}(z)}|d\Theta^*|^2>10^{-10}Q,\ \text{for some $B_{\sigma}(z)$ with $\sigma\in [\rho_0/2,\rho_0]$ and $z\in B_{\rho_0}(y_0)$},
\end{equation}
then for all $y\in B_{\rho_0}(y_0)$ we have
\begin{equation}
\label{almostrepo}
\rho_0^{-2}\int_{B_{\frac{\rho_0}{2}}(y)}|d\Theta^*|^2\le C_1\bigg{(}\rho_0^{-4}\int_{B_{10\rho_0}(y_0)}|\Theta^*-\lambda_{y_0,10\rho_0}|^2+\rho_0^{-2}\int_{B_{10\rho_0}(y_0)}|A^*|^2)\bigg{)}.
\end{equation}
\end{lemma}
\begin{lemma}
\label{inde}
Assume \eqref{almostrepo} holds. For fixed $\Lambda>0$ and any $C^*\gg 1$, there exists $\theta_0(\Lambda,C^*)\ll\frac{1}{10}$, $\eta_0$ and $\epsilon_0(\eta_0,\theta_0)$ that are sufficiently small, such that if $\int_{B_{10\rho_0}(y_0)}|F_A|^2<\epsilon_0$, $\rho_0^{-2}\int_{B_{10\rho_0}(y_0)}|d\Theta^*|^2<\Lambda$, and $\rho_0^{-4}\int_{B_{10\rho_0}(y_0)}|\Theta^*-\lambda_{y_0,10\rho_0}|^2<\eta_0$, then for all $y\in B_{\rho_0}(y_0)$, we have
\begin{equation}
\label{almostinde}
(\theta_0\rho_0)^{-4}\int_{B_{\theta_0\rho_0}(y)}|\Theta^*-\lambda_{y,\theta_0\rho_0}|^2<\frac{\eta_0}{C^*}.
\end{equation}
\end{lemma}
\begin{proof}[proof of Lemma \ref{prerepo}]
As in the Corollary of the classical Luckhaus Lemma (for its statement, see Section 2.7 of \cite{S96}), choose $\eta_0$ sufficiently small such that $\eta_0<\delta_1^2\delta^8$. Here $\delta_1$ is the same as in Section 2.7 of \cite{S96}, and $\delta$ will be determined in the proof of Lemma \ref{repo}. Then for each $y\in B_{8\rho_0}(y_0)$ there exists $W\in W^{1,2}(B_{\rho_0}(y),G)$ and some $\sigma\in (\frac{3\rho_0}{4},\rho_0)$ such that $W$ agrees with $\Theta^*$ in a neighborhood of $\partial B_{\sigma}(y)$, and moreover, the following holds:
\begin{equation}
\label{luch}
\sigma^{-2}\int_{B_{\sigma}(y)}|dW|^2\le \delta\rho_{0}^{-2}\int_{B_{\rho_0}(y)}|d\Theta^*|^2+C\delta^{-1}\rho_0^{-4}\int_{B_{\rho_0}(y)}|\Theta^*-\lambda_{y,\rho}|^2.
\end{equation}
Now apply \eqref{luch} and use the minimality of $\int_{B_{\sigma}(y)}|\nabla_A\Theta^*|^2$, we obtain
\begin{eqnarray}
\begin{split}
&\rho_0^{-2}\int_{B_{\frac{\rho_0}{2}}(y)}|d\Theta^*|^2\le C \sigma^{-2}\int_{B_{\sigma}(y)}|d\Theta^*|^2\le C\sigma^{-2}\int_{B_{\sigma}(y)}|\nabla_A\Theta^*|^2+C\sigma^{-2}\int_{B_{\sigma}(y)}|A^*|^2\\
\le& C\sigma^{-2}\int_{B_{\sigma}(y)}|\nabla_AW|^2+C\sigma^{-2}\int_{B_{\sigma}(y)}|A^*|^2\le C\sigma^{-2}\int_{B_{\sigma}(y)}|dW|^2+C^{\prime}\sigma^{-2}\int_{B_{\sigma}(y)}|A^*|^2\\
\le& C\delta\rho_0^{-2}\int_{B_{\rho_{0}(y)}}|d\Theta^*|^2+C^{\prime\prime}\delta^{-1}\rho_0^{-4}\int_{B_{\rho_0}(y)}|\Theta-\lambda_{y,\rho_0}|^2+C^{\prime}\sigma^{-2}\int_{B_{\sigma}(y)}|A^*|^2\\
\le& C_0\bigg{(}\delta\rho_0^{-2}\int_{B_{\rho_0}(y)}|d\Theta^*|^2+\delta^{-1}\rho_0^{-4}\int_{B_{\rho_0}(y)}|\Theta^*-\lambda_{y,\rho_0}|^2+\rho_0^{-2}\int_{B_{2\rho_0}(y_0)}|A^*|^2\bigg{)}
\end{split}
\end{eqnarray}
for some proper $C_0=C_0(G,\Lambda)$. This finishes the proof of Lemma \ref{prerepo}.
\end{proof}
\begin{proof}[proof of Lemma \ref{repo}]
Firstly, choose $\eta_0$ and $\epsilon_0$ small enough so that Lemma \ref{prerepo} holds. Now from \eqref{seesaw}, we have that:
\begin{equation}
\label{secsmall}
\rho_0^{-2}\int_{B_{\frac{\rho_0}{2}}(y)}|d\Theta^*|^2\le C_0\bigg{(}\delta\Lambda+\delta^{-1}\eta_0+C_{Uh}^2\epsilon_0^2\bigg{)}
\end{equation}
for all $y\in B_{8\rho_0}(y_0)$. Thus for all $\theta\in [\frac{1}{10},1]$, we have the following estimate trivially:
\begin{equation}
(\theta\rho_0)^{-2}\int_{B_{\theta\rho_0}(y)}|d\Theta^*|^2\le 10^4\cdot C_0\bigg{(}\delta\Lambda+\delta^{-1}\eta_0+C_{Uh}^2\epsilon_0^2\bigg{)}\eqqcolon s_0.
\end{equation}
After choosing a smaller $\delta$ (depending only on $\Lambda$, $C_0$ and $G$), then further decreasing $\epsilon_0$ and $\eta_0$ if necessary, we might assume that $s_0$ is small enough, then by Poincar\'e Inequality we have:
\begin{equation}
\label{replace}
\text{\eqref{seesaw} holds with $B_{\rho_0}(y)$ replaced by $B_{\sigma}(z)$},\ \text{for any $\sigma\in [\frac{\rho_0}{4},\frac{\rho_0}{2}]$ and $B_{\sigma}(z)\subseteq B_{4\rho_0}(y_0)$}.
\end{equation}
Now let us choose any $\sigma_0\in[\frac{\rho_0}{2},\rho_0]$, such that $B_{\sigma_0}(z)\in \mathcal{Q}$. Then we may select a Vitali cover of $B_{\sigma_0}(z)\in \mathcal{Q}$ by $\{B_{\frac{\sigma_0}{4}}(z_i)\}_{i=1}^N$ with $z_i\in B_{\sigma_0}(z)$ for each $i$, such that $N\le S$ for some universal constant $S$. By the facts that $B_{2\sigma_0}(z)\subseteq B_{4\rho_0}(y_0)$, $z_i\in B_{\sigma_0}(z)$, and $\rho_0/2\le\sigma_0\le\rho_0$, we have that $B_{\sigma_0}(z_i)\subseteq B_{4\rho_0}(y_0)$ and therefore $B_{\frac{\sigma_0}{2}}(z_i)\in \mathcal{Q}$, for each $i$. By the covering property of $\{B_{\frac{\sigma_0}{4}}(z_i)\}_i$, we have:
\begin{eqnarray}
\begin{split}
\label{subadd}
&\sigma_0^2\int_{B_{\sigma_0}(z)}|d\Theta^*|^2\le S\cdot\sigma_0^2\sup_{i}\int_{B_{\frac{\sigma_0}{4}}(z_i)}|d\Theta^*|^2=16S\cdot \sup_{i}(\frac{\sigma_0}{4})^2\int_{B_{\frac{\sigma_0}{4}}(z_i)}|d\Theta^*|^2\\
\le & 16SC_0\cdot \sup_i\big{(} \delta (\frac{\sigma_0}{2})^2\int_{B_{\frac{\sigma_0}{2}}(z_i)}|d\Theta^*|^2+\gamma_0 \big{)}\le 16SC_0\delta Q+16SC_0\gamma_0,
\end{split}
\end{eqnarray}
where we have used \ref{replace} in the second inequality, and the fact that $B_{\frac{\sigma_0}{2}}(z_i)\in \mathcal{Q}$ in the last inequality; moreover, we have set
\begin{equation}
\gamma_0\eqqcolon\delta^{-1}\int_{B_{4\rho_0}(y_0)}|\Theta^*-\lambda_{4\rho_0,y_0}|^2+\rho_0^2\int_{B_{4\rho_0}(y_0)}|A^*|^2.
\end{equation}
By \eqref{largeness}, we could assume from the beginning that $B_{\sigma_0}(z)$ is such a ball that $\sigma_0^2\int_{B_{\sigma_0}(z)}|d\Theta^*|^2>10^{-10}Q$. Plugging it back to \eqref{subadd} we obtain:
\begin{equation}
\big{(}10^{-10}-16SC_0\delta \big{)}Q \le 16SC_0\gamma_0.
\end{equation}
By choosing $\delta$ small enough (depending only on $C_0$) we have for all $\rho\in [\frac{\rho_0}{2},\rho_0]$ and $y\in B_{\rho_0}(y_0)$:
\begin{eqnarray}
\begin{split}
\rho^{2}\int_{B_{\rho}(y)}|d\Theta^*|^2\le Q&\le 16SC_0\cdot10^{12}\big{(}\delta^{-1}\int_{B_{4\rho_0}(y_0)}|\Theta^*-\lambda_{4\rho_0,y_0}|^2+\rho_0^2\int_{B_{4\rho_0}(y_0)}|A^*|^2\big{)}.
\end{split}
\end{eqnarray}
Multiplying above inequality through $\rho^{-4}$ on both sides, and then set $\rho=\rho_0/2$, we obtain \eqref{almostrepo} for some proper constant $C_1$ depending only on $\Lambda$, $C_0(\Lambda,G)$ and $G$. This finishes the proof of Lemma \ref{repo}.
\end{proof}
\begin{proof}[proof of Lemma \ref{inde}]
For any fixed $y\in B_{\rho_0}(y_0)$ and $\phi\in C^{\infty}_c(B_{\frac{\rho_0}{2}}(y),\text{Mat}_{\mathbb{R}}(k\times k))$, by the weak identity of \eqref{weaklap} as well as \eqref{almostrepo}, we have
\begin{eqnarray}
\begin{split}
\label{prescale}
&\bigg{|}(\frac{\rho_0}{2})^{-2}\int_{B_{\frac{\rho_0}{2}}(y)}\langle d\Theta^*,d\phi\rangle\bigg{|}\le \sup_{B_{\frac{\rho_0}{2}}(y)}|\phi|\cdot(\frac{\rho_0}{2})^{-2}\bigg{(}\int_{B_{\frac{\rho_0}{2}}(y)}|d\Theta^*|^2+2|A^*||d\Theta^*|\bigg{)}\\
\le&\sup_{B_{\frac{\rho_0}{2}}(y)}|\phi|\cdot CC_1  \bigg{(}\rho_0^{-4}\int_{B_{10\rho_0}(y_0)}|\Theta^*-\lambda_{y_0,10\rho_0}|^2+\rho_0^{-2}\int_{B_{10\rho_0}(y_0)}|A^*|^2\bigg{)}.
\end{split}
\end{eqnarray}
Now set $l=CC_1 \bigg{(}\rho_0^{-4} \int_{B_{10\rho_0}(y_0)}|\Theta^*-\lambda_{y_0,10\rho_0}|^2+\rho_0^{-2} \int_{B_{10\rho_0}(y_0)}|A^*|^2\bigg{)}^{1/2}$ and $\mathcal{V}=l^{-1}\Theta^*$, and then insert them into \eqref{prescale}, we have
\begin{eqnarray}
\begin{split}
\label{harweak}
&\bigg{|}(\frac{\rho_0}{2})^{-2}\int_{B_{\frac{\rho_0}{2}}(y)}\langle d\mathcal{V},d\phi\rangle\bigg{|}\\
\le& \bigg{(}\rho_0^{-4} \int_{B_{10\rho_0}(y_0)}|\Theta^*-\lambda_{y_0,10\rho_0}|^2+\rho_0^{-2} \int_{B_{10\rho_0}(y_0)}|A^*|^2\bigg{)}^{1/2}\rho_0\sup_{B_{\frac{\rho_0}{2}}(y)}|d\phi|;
\end{split}
\end{eqnarray}
also note that \eqref{almostrepo} implies
\begin{eqnarray}
\begin{split}
\label{harapple1}
(\frac{\rho_0}{2})^{-2}\int_{B_{\frac{\rho_0}{2}}(y)}|d\mathcal{V}|^2\le 1.
\end{split}
\end{eqnarray}
From now on, we shall follow the arguments in p.15 \cite{S96} verbatim to show that 
\begin{eqnarray}
\begin{split}
\label{simon}
&(\theta\rho_0)^{-4}\int_{B_{\theta\rho_0}(y)}|\Theta^*-\lambda_{y,\theta\rho_0}|^2\\
\le &CC_1(\theta^{-4}\epsilon^2+C\theta^2)\bigg{(}\rho_0^{-4}\int_{B_{10\rho_0}(y_0)}|\Theta^*-\lambda_{y_0,10\rho_0}|^2+\rho_0^{-2}\int_{B_{10\rho_0}(y_0)}|A^*|^2\bigg{)}.
\end{split}
\end{eqnarray}
For the readers convenience, we include the proof. To begin we present the following harmonic approximation lemma; for its proof we refer the reader to Section 1.6 \cite{S96}. 
\begin{lemma}
\label{harap}
For any given $\epsilon$, there is $\delta(n,\epsilon)>0$ such that if $f\in W^{1,2}(B_{\rho}(y))$, $\rho^{2-n}\int_{B_{\rho}(y)}|df|^2\le 1$, and $|\rho^{2-n}\int_{B_{\rho}(y)}df\cdot d\phi|\le \delta\rho \sup_{B_{\rho}(y)}|d\phi|$ for every $\phi\in C^{\infty}_c(B_{\rho}(y))$, then there is a harmonic function $u$ on $B_{\rho}(y)$ with $\rho^{2-n}\int_{B_{\rho}(y)}|du|^2\le 1$ and $\rho^{-n}\int_{B_{\rho}(y)}|u-\lambda_{y,\rho}|^2\le \epsilon^2$. 
\end{lemma}
Let $\epsilon>0$ be arbitrary for the moment. By \eqref{harapple1} and \eqref{harweak}, we could apply Lemma \ref{harap} in order to conclude that there is a $\text{Mat}_{\mathbb{R}}(k\times k)$-valued harmonic function $\mathcal{W}$ on $B_{\rho_0/2}(y)$ such that 
\begin{eqnarray}
\begin{split}
\label{1and2}
(\frac{\rho_0}{2})^{-2}\int_{B_{\rho_0/2}(y)}|d\mathcal{W}|^2\le 1,\ \text{and }(\rho_0/2)^{-4}\int_{B_{\rho_0/2}(y)}|\mathcal{V}-\mathcal{W}|^2\le \epsilon^2
\end{split}
\end{eqnarray}
assuming that 
\begin{eqnarray}
\begin{split}
\label{harapass}
\rho_0^{-4} \int_{B_{10\rho_0}(y_0)}|\Theta^*-\lambda_{y_0,10\rho_0}|^2+\rho_0^{-2} \int_{B_{10\rho_0}(y_0)}|A^*|^2\le \delta^2\end{split}
\end{eqnarray}
where $\delta\equiv \delta(\epsilon)$ is as in Lemma \ref{harap}. Now take $\theta\in (0,1/4]$ and note that we have the following trivially:
\begin{eqnarray}
\begin{split}
\label{square}
(\theta \rho_0)^{-4}\int_{B_{\theta \rho_0}(y)}|\mathcal{V}-\mathcal{W}(y)|^2\le 2(\theta\rho_0)^{-4}\int_{B_{\theta\rho_0}(y)}(|\mathcal{V}-\mathcal{W}|^2+|\mathcal{V}-\mathcal{W}(y)|^2).
\end{split}
\end{eqnarray}
Now using $1$-dimensional calculus along line segment with end-point at $y$ together with a standard elliptic estimate for harmonic function, we have
\begin{eqnarray}
\begin{split}
\sup_{B_{\theta\rho_0}(y)}|\mathcal{W}-\mathcal{W}(y)|^2\le (\theta\rho_0\sup_{B_{\theta\rho_0}(y)}|d\mathcal{W}|)^2\le C\theta^2\rho_0^{2-n}\int_{B_{\rho_0/2}(y)}|d\mathcal{W}|^2.
\end{split}
\end{eqnarray}
Using this together with \eqref{1and2} in \eqref{square}, we conclude that
\begin{eqnarray}
\begin{split}
\label{vw}
(\theta\rho_0)^{-4}\int_{B_{\theta\rho_0}(y)}|\mathcal{V}-\mathcal{W}(y)|^2\le \theta^{-n}\epsilon^2+C\theta^2,
\end{split}
\end{eqnarray}
Where $C$ is universal. Writing $\mathcal{V}=l^{-1}\Theta^*$ in \eqref{vw}, we obtain \eqref{simon}. Now we determine the parameters $\theta$, $\epsilon$, $\eta_0$, and $\epsilon_0$. Indeed, we determine $\theta$ and then $\epsilon$ such that both $\theta$ and $\epsilon$ are sufficiently small and independent of $\eta_0$, and lastly choose $\epsilon_0$ small with respect to $\eta_0$ in order to conclude from \eqref{vw} that \eqref{almostinde} holds for the fixed $C^*\gg 1$ (which will be determined in the proof of Proposition \ref{smallnessprop}, to depend only on $C_1$ from Lemma \ref{repo} and the Poincar\'e Inequality constant $C_p$). Next, for the chosen $\epsilon$, we further decrease $\eta_0$ and then $\epsilon_0$ such that \eqref{harapass} holds. Upon these choices of parameters, we finish the proof of Lemma \ref{inde}.
\end{proof}
The main difficulty of Theorem \ref{eps} would be to achieve the smallness of $\sigma^{-2}\int_{B_{\sigma}(y)}|d\Theta^*|^2$ for all $y\in B_{\rho_0}(y_0)$ and $\sigma\le\rho_0$ given that on the top scale $\rho_0$. Indeed, from this we easily obtain the smallness of the BMO-seminorm on $\Theta^*$, which helps to improve the regularity of $\Theta^*$ and eventually achieve the desired $L^4$-estimates, by following a quite standard elliptic PDE regularity method (similar to that in \cite{B93}). Let us remark that in the case of the classical stationary harmonic map, such smallness is a trivial consequence of the monotonicity formula for $r^{2-n}\int_{B_r(q)}|\nabla u|^2$. In contrast, the quantity $\zeta_y(\rho)\equiv\rho^{-2}\int_{B_{\rho}(y)}|\nabla_A\Theta_0|^2$ is not monotone due to the error term in \eqref{mono} involving the curvature $F_A$. Furthermore, the integral of this error term from the top scale, say $\rho_0$, to an arbitrary small scale $\rho$ grows like $|\log\frac{\rho}{\rho_0}|$ as $\rho\to 0$. Hence, the natural idea of directly integrating both sides of \eqref{mono} and then using fundamental theorem of calculus is not very useful in estimating $\zeta_y(\rho)$ for $\rho$ small. To conquer this issue, we use induction on scales instead, which one should think of as a method of continuity. Heuristically, we shall be applying Lemma \ref{repo} and Lemma \ref{inde} in an alternating fashion, starting from the top scale $\rho_0$, then each time the scale goes down by a bounded factor in between $10^{-1}$ and $C^{-1}$ for some universal large constant $C\gg 10$ which will be determined later, while we shall keep track of the smallness of the energy $\zeta_y(\cdot)$ through all the stages. If such spiral mechanism succeeds, we will then finish our induction. Nevertheless, one subtlety lies in the fact that condition \eqref{largeness} is not always verified (which is essentially caused by the error term of \eqref{mono}) for Lemma \ref{repo} to be applicable, and therefore this heuristic spiral mechanism is likely to break down. Fortunately, this needs not to stop our induction on the smallness of energy. A key observation is that the failure of \eqref{largeness} exactly implies that an energy decay by a definite factor from the last scale must occur, which fills in the gap of the induction anyway. To rigidify this observation we will need to employ \eqref{mono} to two consecutive scales, in which case scenario \eqref{mono} behaves like an ``almost'' monotonicity formula. Now by our inductive assumption, we can complete the induction when the spiral mechanism breaks down (i.e. the condition \ref{largeness} fails). This gives the outline of how we will proceed next. To start with, we present the following Proposition:
\begin{proposition}
\label{smallnessprop}
For any $\delta_0$, there exist $\epsilon_0$ and $\eta$ sufficiently small such that given $\int_{B_{10\rho_0}(y_0)}|F_A|^2<\epsilon_0$ and $\rho_0^{-2}\int_{B_{10\rho_0}(y_0)}|d\Theta^*|^2<\eta$, then 
\begin{equation}
\rho^{-2}\int_{B_{\rho}(y)}|d\Theta^*|^2<\delta_0.
\end{equation}
for all $y\in B_{\rho_0}(y_0)$ and $\rho\le\rho_0$.
\end{proposition}
\begin{proof}
Set $\eta_0=10C_p\eta$ (here $C_p$ denotes the Poincar\'e Inequality constant). Firstly, let us choose $\eta$ and $\epsilon_0$ so small that the conclusions of Lemma \ref{prerepo} through Lemma \ref{inde} hold, provided all other conditions are satisfied. The parameter $C^*$ (as appear in Lemma \ref{inde}) depends only on $C_1$ and the Poincar\'e $C_p$, and will be determined by the end of the proof. Thus fixes $\theta_0$ as in Lemma \ref{inde} as well. For each fixed $y\in B_{\rho_0}(y_0)$, we will prove by induction that:
\begin{equation}
\label{indl}
\rho_{l}^{-2}\int_{B_{\rho_l}(y)}|d\Theta^*|^2<\eta,
\end{equation}
where $\rho_l$ will be determined in the course of the induction process, satisfying that
\begin{equation}
\label{scales}
10\le\frac{\rho_l}{\rho_{l+1}}\le10000\theta_0^{-1}.
\end{equation}
For the starting scale $l=0$, the conclusion \eqref{indl} follows exactly from the condition of Proposition \ref{smallnessprop}. Next, we shall assume for some $l_0\ge 0$, all previous scales $\rho_{l}$ with $l\le l_0$ have been defined on which \eqref{scales} holds, such that $\rho_{l}^{-2}\int_{B_{\rho_l}(y)}|d\Theta^*|^2<\eta$. In order to find the next inductive scale $\rho_{l_0+1}$, we consider two alternative cases on the current scale $\rho_{l_0}$.\\

Case (a): By treating $B_{\rho_{l_0}}(y)$ as $B_{10\rho_0}(y_0)$ in Lemma \ref{repo}, the condition \eqref{largeness} does not hold; however this means that:
\begin{equation}
\sigma^2\int_{B_{\sigma}(z)}|d\Theta^*|^2<10^{-10}Q,\ \forall B_{\sigma}(z)\ \text{with}\ \sigma\in[\frac{\rho_{l_0}}{20},\frac{\rho_{l_0}}{10}], z\in B_{\frac{\rho_{l_0}}{10}}(y).
\end{equation}
It will then be an easy exercise to see that this means that:
\begin{equation}
\sigma^{-2}\int_{B_{\sigma}(z)}|d\Theta^*|^2\le \frac{1}{20000}(\frac{2\rho_{l_0}}{5})^{-2}\int_{B_{\frac{2\rho_{l_0}}{5}}(y)}|d\Theta^*|^2<\frac{\eta}{2}<\eta.
\end{equation}
By taking $B_{\sigma}(z)$ to be $B_{\frac{\rho_{l_0}}{10}}(y)$ and defining $\rho_{l_0+1}\eqqcolon \frac{\rho_{l_0}}{10}$, we are done.\\

Case (b): By treating $B_{\rho_{l_0}}(y)$ as $B_{10\rho_0}(y_0)$ in Lemma \ref{repo}, the condition \eqref{largeness} holds. In this case, we could then apply Lemma \ref{repo} and then Lemma \ref{inde} (note that the $\Lambda$-bound condition is satisfied by $\eta$, and hence can be taken to be, for instance, $100$. Here we simply want to point out the fact that the smallness of $\Lambda$ is unecessary) to obtain from \eqref{almostinde} that:
\begin{equation}
\label{de1}
(\frac{\theta_0\rho_{l_0}}{10})^{-4}\int_{B_{\frac{\theta_0\rho_{l_0}}{10}}(y)}|\Theta^*-\lambda_{y,\frac{\theta_0\rho_{l_0}}{10}}|^2<\frac{\eta_0}{C^*}=\frac{10C_p\eta}{C^*}.
\end{equation}
Now, we are facing two alternative sub-cases of Case (b), which we refer to as Case (c) and Case (d) respectively.\\

Case (c): By treating $B_{\theta_0\rho_{l_0}}(y)$ as $B_{10\rho_0}(y_0)$ in Lemma \ref{repo}, the condition \eqref{largeness} holds. Then Lemma \ref{repo} is valid with  $B_{\theta_0\rho_{l_0}}(y)$ in the place of $B_{10\rho_0}(y_0)$, and let us apply \eqref{de1} (note that the $\Lambda$-bound condition is satisfied by constant $10^4\theta_0^{-4}\eta$; hence, by choosing $\eta$ small,we could assume $10^8\theta_0^{-8}\eta<100$ and simply take $\Lambda=100$. We note such choice of $\eta$ and $\Lambda$ works for all cases discussed below and therefore we no longer mention the $\Lambda$-boundedness condition in the following arguments) and obtain: 
\begin{eqnarray}
\begin{split}
(\frac{\theta_0\rho_{l_0}}{200})^{-2}\int_{B_{\frac{\theta_0\rho_{l_0}}{200}}(y)}|d\Theta^*|^2\le &C_1\bigg{(}(\frac{\theta_0\rho_{l_0}}{100})^{-4}\int_{B_{\frac{\theta_0\rho_{l_0}}{10}}(y)}|\Theta^*-\lambda_{y,\frac{\theta_0\rho_{l_0}}{10}}|^2+\int_{B_{\frac{\theta_0\rho_{l_0}}{10}}(y_0)}|F_A|^2\bigg{)}\\
\le& \frac{10^5C_1C_p\eta}{C^*}+C_1\epsilon_0^2.
\end{split}
\end{eqnarray}
By choosing $C^*$ large in Lemma \ref{inde} (to depend only on $C_1$ and $C_p$), and then $\epsilon_0$ small with respect to $\eta$, above might be made less than $\eta/2$. Upon setting $\rho_{l_0+1}\eqqcolon \frac{\theta_0\rho_{l_0}}{200}$, we are done.\\

Case (d): By treating $B_{\theta_0\rho_{l_0}}(y)$ as $B_{10\rho_0}(y_0)$ in Lemma \ref{repo}, the condition \eqref{largeness} does not hold. In this case scenario, we further study two alternative sub-cases of Case (d), which we call Case (e) and Case (f) respectively.\\

Case (e): 
\begin{equation}
(\theta_0\rho_{l_0})^{-2}\int_{B_{\theta_0\rho_{l_0}}(y)}|d\Theta^*|^{2}<2\eta.
\end{equation}
Now by the same reasoning as in Case (a), the fact that \eqref{largeness} does not hold implies
\begin{equation}
(\frac{\theta_0\rho_{l_0}}{10})^{-2}\int_{B_{\frac{\theta_0\rho_{l_0}}{10}}(y)}|d\Theta^*|^{2}<\frac{1}{20000}(\frac{2\theta_0\rho_{l_0}}{5})^{-2}\int_{B_{\frac{2\theta_0\rho_{l_0}}{5}}(y)}|d\Theta^*|^{2}<\eta.
\end{equation}
Hence, by setting $\rho_{l_0+1}\eqqcolon \frac{\theta_0\rho_{l_0}}{10}$ we are done.\\

Case (f): 
\begin{equation}
\label{contra}
(\theta_0\rho_{l_0})^{-2}\int_{B_{\theta_0\rho_{l_0}}(y)}|d\Theta^*|^{2}\ge2\eta.
\end{equation}
This is the last possible case to be considered, and we will prove that this case cannot happen at all, by choosing $\epsilon_0$ sufficiently small. The proof goes by contradiction. Suppose this is the case. Indeed, according to inductive assumption we have
\begin{equation}
\label{contra1}
\rho_{l_0}^{-2}\int_{B_{\rho_{l_0}}(y)}|d\Theta^*|^{2}<\eta.
\end{equation}
Let us simplify the notation by setting $\rho_1\eqqcolon\theta_0\rho_{l_0}$ and $\rho_2\eqqcolon \rho_{l_0}$; we then have
\begin{equation}
\label{contra2}
\rho_1^{-2}\int_{B_{\rho_1}(y)}|d\Theta^*|^{2}-\rho_{2}^{-2}\int_{B_{\rho_2}(y)}|d\Theta^*|^{2}>\eta.
\end{equation}
Further set $\zeta_y(r)\coloneqq r^{-2}\int_{B_{r}(y)} |\nabla_A \Theta|^2$. Using this notation we see that \eqref{contra2} gives
\begin{eqnarray}
\begin{split}
\label{contra3}
&\zeta_y(\rho_1)-\zeta_y(\rho_2)=\big{(}\rho_1^{-2}\int_{B_{\rho_1}(y)}|d\Theta^*|^{2}-\rho_{2}^{-2}\int_{B_{\rho_2}(y)}|d\Theta^*|^{2}\big{)}
-\rho_1^{-2}\int_{B_{\rho_1}(y)}|A^*|^2\\
-&\rho_2^{-2}\int_{B_{\rho_2}(y)}|A^*|^2-2\rho_1^{-2}|\int_{B_{\rho_1}(y)}\langle A^*,d\Theta^*\rangle|-2\rho_2^{-2}|\int_{B_{\rho_2}(y)}\langle A^*,d\Theta^*\rangle|\\
\ge& \eta-2C_{Uh}^2\epsilon_0^2-2\sqrt{\Lambda}C_{Uh}\epsilon_0.
\end{split}
\end{eqnarray}
On the other hand, it is a simple computation to see that \eqref{mono} gives $\big{(}\sqrt{\zeta_y(r)}\big{)}^{'}\ge -\frac{1}{r}\sqrt{\epsilon_0}$. Thus for all $\rho_1<\rho_2<\rho_0$, the following holds:
\begin{equation}
\sqrt{\zeta_y(\rho_2)}-\sqrt{\zeta_y(\rho_1)}\ge  -|\log(\frac{\rho_1}{\rho_2})|\cdot \sqrt{\epsilon_0}.
\end{equation}
Therefore,
\begin{equation}
\zeta_y(\rho_2)-\zeta_y(\rho_1)=(\sqrt{\zeta_y(\rho_2)}+\sqrt{\zeta_y(\rho_1)})(\sqrt{\zeta_y(\rho_2)}-\sqrt{\zeta_y(\rho_1)})\ge -2\sqrt{\Lambda}|\log\theta_0|\sqrt{\epsilon_0}.
\end{equation}
Combining this with \eqref{contra3}, we have
\begin{equation}
\label{givecontra}
2\sqrt{\Lambda}|\log\theta_0|\sqrt{\epsilon_0}\ge\eta-2C_{Uh}^2\epsilon_0^2-2\sqrt{\Lambda}C_{Uh}\epsilon_0.
\end{equation}
where $\Lambda$, as we mentioned earlier, might be taken as $100$. Now \eqref{givecontra} gives a contradiction if we choose $\epsilon_0$ sufficiently small. In other words, we have ruled out Case (f). \\

Thus we finish the induction. To sum up, for each point $y\in B_{\rho_0}(y_0)$ we have obtained that 
$$\rho_{l}^{-2}\int_{B_{\rho_0}(y)}|d\Theta^*|^2<\eta$$ 
for a sequence of $\{\rho_l\}_{l\ge0}$ satisfying $10\le\frac{\rho_l}{\rho_{l+1}}\le10000\theta_0^{-1}$. But from this trivially we have
\begin{equation}
\rho^{-2}\int_{B_{\rho}(y)}|d\Theta^*|^2<10^{16}\theta_0^{-4}\eta
\end{equation}
for all $\rho<\rho_0$. By choosing even smaller $\eta$ from the very beginning of the proof with respect to $\delta_0$ (and hence correspondingly a smaller $\epsilon_0$) we have $\rho^{-2}\int_{B_{\rho}(y)}|d\Theta^*|^2<\delta_0$ for all $\rho<\rho_0$ and $y\in B_{\rho_0}(y_0)$. This is the end of the proof of Proposition \ref{smallnessprop}.
\end{proof}
Combining Proposition \ref{smallnessprop} and \eqref{uhgauge}, we have proved that given the $L^2$ smallness of the curvature, the smallness of the scaling invariant energy of $A_0$ on the first scale (i.e. $B_{10\rho_0}(p)$) implies its smallness on all balls contained in $B_{4\rho_0}(p)$. More precisely, we have
\begin{corollary}
\label{a0bound}
For all $\delta_0$, there exists $\epsilon_0$ and $\eta_0$, given that $\int_{B_{10\rho_0}(y_0)}|F_A|^2<\epsilon_0$, $\rho_0^{-2}\int_{B_{10\rho_0}(y_0)}|A_0|^2<\eta_0$, then $\rho^{-2}\int_{B_{\rho}(x)}|A_0|^2<\delta_0$ for all $B_{\rho}(x)\subseteq B_{4\rho_0}(y_0)$.
\end{corollary}
To continue to finish the proof of the $\epsilon$-regularity Theorem \ref{eps}, and especially, to achieve the desired $L^4$ estimate on $A_0$, an intermediate step is to show that the $L^4$ norm of $A_0$ on $B_{\rho_0}(y_0)$ is finite. Namely,
\begin{proposition}
\label{l4finite}
There exists $\epsilon_0$ and $\eta_0$, given that $\int_{B_{10\rho_0}(y_0)}|F_A|^2<\epsilon_0$, $\rho_0^{-2}\int_{B_{10\rho_0}(y_0)}|A_0|^2<\eta_0$, then $\int_{B_{2\rho_0}(y_0)}|\nabla A_0|^2+|A_0|^4<\infty$.
\end{proposition} 
\begin{proof}
Fix some $p_0>2$, such that $W^{1,p_0}\hookrightarrow L^{4+\alpha_0}$ is the Sobolev embedding in dimension $4$ for some $\alpha_0$ to be determined later in the proof. Define the following regularity radius of $A$:
\begin{eqnarray}
\begin{split}
\hat\rho=\sup\{r\le \rho_0: \sup_{x\in B_{3\rho_0}(y_0)}\int_{B_{r}(x)}|F_A|^{p_0}\le \hat\epsilon\}.
\end{split}
\end{eqnarray}
for some small number $\hat\epsilon>0$ to be specified. By the smoothness of $A$, we have $\hat{\rho}>0$. To prove Proposition \ref{l4finite}, it suffices to prove $\int_{B_{\hat\rho/2}(x)}|A_0|^4<\infty$ for each $x\in B_{2\rho_0}(y_0)$. For this purpose, let us fix some $B_{\hat\rho}(x_0)$ for arbitrary $x_0\in B_{2\rho_0}(y_0)$. By choosing $\hat\epsilon<\epsilon_{Uh}$, we could then apply Uhlenbeck Theorem (see, for instance, Theorem 6.1, \cite{W04}) to find a Coulomb frame (which we shall again refer to as the Uhlenbeck gauge) on $B_{\hat\rho}(x_0)$ such that the connection form of $A$ under this frame (which we again denote by $A^*$) satisfies
\begin{eqnarray}
\begin{split}
\label{p0bound}
\|A^*\|_{W^{1,p_0}(B_{\hat\rho}(x_0))}\le C_{Uh}\|F_A\|_{L^{p_0}(B_{\hat\rho}(x_0))}.
\end{split}
\end{eqnarray}
Again define $\Theta^*: B_{\hat\rho}(x_0)\to \text{Mat}_{\mathbb{R}}(k\times k)$ the trivialization of $\Theta_0$ under this Uhlenbeck gauge. Clearly, $\Theta^*$ satisfies \eqref{weaklap}. We need point out that  By Corollary \ref{a0bound} and \eqref{p0bound}, for any $\delta$ we may choose sufficiently small $\epsilon_0$, $\eta_0$, and $\epsilon_{Uh}$, such that $\rho^{-2}\int_{B_\rho(x)}|d\Theta^*|^2<\delta_0$ holds for all $B_\rho(x)$ with $\rho\le \hat{\rho}/2$ and $x\in B_{\hat{\rho}/2}(x_0)$. By Sobolev and Poincar\'e inequality, we have
\begin{eqnarray}
\begin{split}
\label{allsmall}
\rho^{-4}\int_{B_\rho(x)}|\Theta^*-\lambda_{x,\rho}|^4<\eta_0^2,
\end{split}
\end{eqnarray}
where $\lambda_{x,\rho}=\rho^{-4}\int_{B_\rho(x)}\Theta^*(y)dV_g(y)$ and $\eta_0\le C\delta_0$. Hence, for some $\delta$ sufficiently small to be specified, upon choosing $\delta_0$ and $\epsilon_0$ sufficiently small, and then replacing $y_0$ by any $x\in B_{2\hat\rho/3}(x_0)$ and $\rho_0$ by any $\rho\le \hat\rho/30$, the crucial smallness hypothesis in Lemma \ref{prerepo} is verified by \eqref{allsmall}; then we could follow the proof of Lemma \ref{prerepo} verbatim to conclude that
\begin{eqnarray}
\begin{split}
\label{simonepsilon}
\rho^{-2}\int_{B_{\frac{\rho}{2}}(x)}|d\Theta^*|^2\le C_0\bigg{(}\delta\rho^{-2}\int_{B_{\rho}(x)}|d\Theta^*|^2+\delta^{-1}\rho^{-4}\int_{B_{\rho}(x)}|\Theta^*-\lambda_{x,\rho}|^2+\rho^{-2}\int_{B_{\rho}(x)}|A^*|^2\bigg{)}.
\end{split}
\end{eqnarray}
Now we need the following lemma (we refer the readers to Lemma 2, Section 2.8 of \cite{S96} for its proof):
\begin{lemma}
\label{simonbook}
Let $B_\rho(x)$ be any ball in $\mathbb{R}^n$, $k\in \mathbb{R}$, $\gamma>0$, and let $\phi$ be any $[0,\infty)$-valued convex sub-additive function on the collection of convex subsets of $B_\rho(x)$. There exist $\epsilon_0\equiv \epsilon_0(n,k),\ C\equiv C(n,k)$ such that if 
\begin{eqnarray}
\begin{split}
\sigma^k\phi(B_{\sigma/2}(z))\le\epsilon_0\sigma^k\phi(B_\sigma(z))+\gamma
\end{split}
\end{eqnarray}
whenever $B_{2\sigma}(z)\subseteq B_\rho(x)$, then
\begin{eqnarray}
\begin{split}
\rho^k\phi(B_{\rho/2}(x))\le C\gamma.
\end{split}
\end{eqnarray}
\end{lemma}
If we take $B_\rho(x)$ to be arbitrary ball with $x\in B_{\hat\rho/2}(x_0)$ and $\rho<\hat\rho/30$, and set $k=2$, $\phi(B)=\int_B|d\Theta^*|^2$, and then define
\begin{eqnarray}
\begin{split}
\gamma=C_0(\delta^{-1}\int_{B_{\rho}(x)}|\Theta^*-\lambda_{x,\rho}|^2+\rho^{2}\int_{B_{\rho}(x)}|A^*|^2),
\end{split}
\end{eqnarray}
then by \eqref{simonepsilon}, for any $B_\sigma(z)$ with $B_{2\sigma}(z)\subseteq B_\rho(x)$ we have
\begin{eqnarray}
\begin{split}
&\sigma^2\phi(B_{\sigma/2}(z))=\sigma^2\int_{B_{\sigma/2}(z)}|d\Theta^*|^2\\
\le& C_0\delta\int_{B_\sigma(z)}|d\Theta^*|^2+C_0\bigg{(}\delta^{-1}\int_{B_\sigma(z)}|\Theta^*-\lambda_{z,\sigma}|^2+\sigma^2\int_{B_\sigma(z)}|A^*|^2\bigg{)}\\
\le&C_0\delta\int_{B_\sigma(z)}|d\Theta^*|^2+C_0\bigg{(}\delta^{-1}\int_{B_\sigma(z)}|\Theta^*-\lambda_{x,\rho}|^2+\rho^2\int_{B_\rho(x)}|A^*|^2\bigg{)}\\
\le& C_0\delta\int_{B_\sigma(z)}|d\Theta^*|^2+\gamma=C_0\delta\sigma^2 \phi(B_{\sigma}(z))+\gamma.
\end{split}
\end{eqnarray}
Hence upon choosing $\delta$ sufficiently small in \eqref{simonepsilon}, we could apply Lemma \ref{simonbook} to see that
\begin{eqnarray}
\begin{split}
\label{simonepsilon1}
\rho^{-2}\int_{B_{\frac{\rho}{2}}(x)}|d\Theta^*|^2\le C_1\bigg{(}\rho^{-4}\int_{B_{\rho}(x)}|\Theta^*-\lambda_{x,\rho}|^2+\rho^{-2}\int_{B_{\rho}(x)}|A^*|^2)\bigg{)}.
\end{split}
\end{eqnarray}
for all $B_\rho(x)$ with $\rho<\hat\rho/2$ and $x\in B_{\hat\rho}(x_0)$, provided $\delta$ is chosen sufficiently small. Now we could follow the proof of Lemma \ref{inde} verbatim to obtain \eqref{simon} upon replacing both $y$ and $y_0$ by $x$, and $\rho_0$ by $\rho$. Given any $\alpha<1$, by first choosing $\theta$ and then $\epsilon$ to be small on the right hand side of \eqref{simon} such that both depend on $\alpha$, we arrive at the following estimate:
\begin{eqnarray}
\begin{split}
\label{itnow}
(\theta\rho)^{-4}\int_{B_{\theta\rho}(x)}|\Theta^*-\lambda_{x,\theta\rho}|^2\le \theta^{2\alpha}\bigg{(}\rho^{-4}\int_{B_{\rho}(x)}|\Theta^*-\lambda_{x,\rho}|^2+\rho^{-2}\int_{B_{\rho}(x)}|A^*|^2\bigg{)}.
\end{split}
\end{eqnarray}
Recall $\alpha_0$ is a that $W^{1,p_0}\hookrightarrow L^{4+\alpha_0}$ is a Sobolev embedding in dimension $4$. The choices of both $\alpha_0$ and $p_0$ will be specified later to depend on $\alpha$. Using H\"older inequality we obtain
\begin{eqnarray}
\begin{split}
\label{4pluseps}
\rho^{-2}\int_{B_\rho(x)}|A^*|^2\le \rho^{-2}\big{(}\int_{B_\rho(x)}|A^*|^{4+\alpha_0}\big{)}^{\frac{2}{4+\alpha_0}}\cdot\rho^{\frac{4(2+\alpha_0)}{4+\alpha_0}}\le \epsilon_1\rho^{\frac{2\alpha_0}{4+\alpha_0}}\equiv \epsilon_1\rho^{2\beta_0}
\end{split}
\end{eqnarray}
for $\epsilon_1\eqqcolon C_{Uh}^2\epsilon_{Uh}^{2/p_0}$ and $\beta_0\eqqcolon\frac{\alpha_0}{4+\alpha_0}<1$. Note in the second inequality we have used that $\int_{B_{\hat{\rho}}(x_0)}|F_A|^{p_0}<\epsilon_{Uh}$ and the Uhlenbeck gauge property that $\|A^*\|_{W^{1,p_0}}\le C_{Uh}\|F_A\|_{L^{p_0}}$. For the sake of brevity, in the following estimates we introduce the following notations:
\begin{eqnarray}
\begin{split}
\phi_{x}(r)=r^{-4}\int_{B_r(x)}|\Theta^*-\lambda_{x,r}|^2,\ \rho_i\equiv \theta^i\hat\rho,\ \text{for all integer\ }i\ge 0.
\end{split}
\end{eqnarray}
For arbitrary $K\ge 1$ we iteratively apply \eqref{itnow} as follows:
\begin{eqnarray}
\begin{split}
\label{itera1}
\phi_{x_0}(\rho_K)&\le \theta^{2\alpha}(\phi_{x_0}(\rho_{K-1})+\epsilon_1\rho_{K-1}^{2\beta_0})\le \theta^{2\alpha}((\theta^{2\alpha}\phi_{x_0}(\rho_{K-2})+\epsilon_1\rho_{K-2}^{2\beta_0})+\epsilon_1\rho_{K-1}^{2\beta_0})\\
&\le \cdots\le \theta^{2\alpha K}\phi_{x_0}(\hat\rho)+\epsilon_1\sum_{i=1}^K\rho_{K-i}^{2\beta_0}\cdot \theta^{2\alpha(i-1)}\\
&\le (\frac{\rho_K}{\hat\rho})^{2\alpha}\phi_{x_0}(\hat\rho)+\epsilon_1\theta^{-2\alpha}\rho_K^{2\beta_0}\sum_{i=1}^K\theta^{2i(\alpha-\beta_0)}. 
\end{split}
\end{eqnarray}
Now we could determine the choices of $\alpha_0$ and $p_0$ by setting $\beta_0=\max\{2\alpha-1,\alpha/2\}$. Plugging this into the right hand side of the above inequality we have
\begin{eqnarray}
\begin{split}
\phi_{x_0}(\rho_K)\le C(\theta,\hat\rho,\alpha)  \rho_K^{2\beta_0},
\end{split}
\end{eqnarray}
which implies that
\begin{eqnarray}
\begin{split}
\rho^{-4}\int_{B_{\rho}(x_0)}|\Theta^*-\lambda_{x,\rho}|^2=\phi_{x_0}(\rho)\le C(\theta,\hat\rho,\alpha)^2  \rho^{2\beta_0}.
\end{split}
\end{eqnarray}
for all $\rho\le \hat{\rho}/2$. By Campanato's Theorem, 
\begin{eqnarray}
\begin{split}
\label{holderbeta}
[\Theta^*]_{C^{\beta_0}(B_{\hat{\rho}/2}(x_0))}\le C(\theta,\hat\rho,\alpha).
\end{split}
\end{eqnarray}
Now we need the following proposition:
\begin{proposition}
\label{localprop0}
For all $y\in B_{\hat\rho/2}(x_0)$, there exists some $0<\rho_y<\hat\rho$, such that the following holds for some universal constant $C>0$ and all $\zeta\in C^{\infty}_c(B_{\rho_y}(y))$:
\begin{equation}
\label{nabla2}
\int_{B_{\rho_y}(y)}|\nabla^2\Theta^*|^2\zeta^2\le C\bigg{(}\int_{B_{\rho_y}(y)}|d\Theta^*|^2|d\zeta|^2+\int_{B_{\rho_y}(y)}|\nabla A^*|^2\zeta^2\bigg{)}<\infty.
\end{equation}
\end{proposition}
In view of the arbitrariness of $x_0\in B_{2\rho_0}(y_0)$, Proposition \ref{l4finite} immediately follows from Proposition \ref{localprop0} together with a standard covering argument and the Sobolev embedding $W^{1,2}\hookrightarrow L^4$.
\end{proof}
Thus it remains to prove Proposition \ref{localprop0} in order to conclude Proposition \ref{l4finite}.
\begin{proof}[Proof of Proposition \ref{localprop0}]
We will use the standard integration by parts and difference quotience method. Choose any $x\in B_{\hat\rho/2}(x_0)$ and some $\rho\ll\hat\rho/2$ to be specified. For any $\xi\in W^{1,2}_0(B_{\rho}(y),\text{Mat}_{\mathbb{R}}(k\times k))\cap L^{\infty}(B_{\rho}(y),\text{Mat}_{\mathbb{R}}(k\times k))$, let $\eta(z)=|\xi|^2(\Theta^*(z)-\Theta^*(y))$ play the role of the test section in the weak identity of \eqref{weaklap}, we have
\begin{eqnarray}
\begin{split}
\label{wkid1}
0=&\int_{B_{\rho}(y)}|d\Theta^*|^2|\xi|^2+\int_{B_{\rho}(y)}\langle\langle d\Theta^*,\langle 2\xi, d\xi\rangle_{E}\rangle_{\Lambda^1},\Theta^*-\Theta^*(y)\rangle_E\\
&-\int_{B_{\rho}(y)}\langle\Theta^*\langle(d\Theta^*)^T\cdot d\Theta^*\rangle,|\xi|^2(\Theta^*-\Theta^*(y))\rangle_E\\
&-\int_{B_{\rho}(y)}\langle\Theta^*\langle(d\Theta^*)^T\cdot A^*(\Theta^*)\rangle,|\xi|^2(\Theta^*-\Theta^*(y))\rangle_{E}\\
&+\int_{B_{\rho}(y)}\langle\Theta^*\langle (A^*(\Theta^*))^T\cdot d\Theta^*\rangle,|\xi|^2(\Theta^*-\Theta^*(y))\rangle_E\\
=&Q_1+Q_2+Q_3+Q_4+Q_5.
\end{split}
\end{eqnarray}
By Cauchy's inequality and \eqref{holderbeta}, we obtain
\begin{eqnarray}
\begin{split}
& |Q_2|+|Q_3|+|Q_4|+|Q_5|\le C\rho^{\beta_0}\big{(}\int_{B_{\rho}(y)}|d\Theta^*|^2|\xi|^2\big{)}^{1/2}\big{(}\int_{B_{\rho}(y)}|d\xi|^2\big{)}^{1/2}\\
&+C\rho^{\beta_0}\int_{B_{\rho}(y)}|d\Theta^*|^2|\xi|^2+C\rho^{\beta_0}\big{(}\int_{B_{\rho}(y)}|d\Theta^*|^2|\xi|^2\big{)}^{1/2}\big{(}\int_{B_{\rho}(y)}|A^*|^2|\xi|^2\big{)}^{1/2}\\
&=P_1+P_2+P_3.
\end{split}
\end{eqnarray}
Insert these estimates back to \eqref{wkid1} and let $P_2$ be absorbed into $Q_1$ by choosing $\rho$ small, we obtain
\begin{eqnarray}
\begin{split}
\label{firstorder}
\int_{B_{\rho}(y)}|d\Theta^*|^2|\xi|^2\le C\rho^{\beta_0}\big{(}\int_{B_{\rho}(y)}|d\xi|^2+\int_{B_{\rho}(y)}|A^*|^2|\xi|^2\big{)}.
\end{split}
\end{eqnarray}
Now choose arbitrary $\zeta\in C_c^{\infty}(B_{\rho}(y))$. For any $l=1,\cdots,4$, define $\phi_l=\big{(}\zeta^2(\Theta^*)_{(l,\tau)}\big{)}_{(l,\tau)}$, where $(\cdot)_{(l,\tau)}$ is the $\tau$-difference quotient along $l^{\text{th}}$ direction. Apply \eqref{weaklap} with test function $\phi_l$, we obtain
\begin{eqnarray}
\begin{split}
\label{absorbr1}
0=&\int_{B_{\rho}(y)}|(d\Theta^*)_{(l,\tau)}|^2\zeta^2+\int_{B_{\rho}(y)}\langle\langle (d\Theta^*)_{(l,\tau)},d\zeta\rangle_{\Lambda^1},2\zeta(\Theta^*)_{(l,\tau)}\rangle+\int_{B_{\rho}(y)}\langle(E)_{(l,\tau)},(\Theta^*)_{(l,\tau)}\rangle\zeta^2\\
=&R_1+R_2+R_3.
\end{split}
\end{eqnarray}
where $E=\Theta^*\langle(d\Theta^*)^T\cdot d\Theta^*\rangle+\Theta^*\langle(d\Theta^*)^T\cdot A^*(\Theta^*)\rangle-\Theta^*\langle (A^*(\Theta^*))^T\cdot d\Theta^*\rangle$. By applying Young's Inequality with $\epsilon(=1/10)$ we estimate
\begin{eqnarray}
\begin{split}
|R_2|\le \frac{1}{10} \int_{B_{\rho}(y)}|(d\Theta^*)_{(l,\tau)}|^2\zeta^2+C\int_{B_{\rho}(y)}|d\zeta|^2|(\Theta^*)_{(l,\tau)}|^2.
\end{split}
\end{eqnarray}
Next, we shall estimate $R_3$:
\begin{eqnarray}
\begin{split}
&|R_3|\le \int_{B_{\rho}(y)}|(\Theta^*)_{l,\tau}|^2|d\Theta^*|^2\zeta^2+2\int_{B_{\rho}(y)}|(d\Theta^*)_{l,\tau}||d\Theta^*||(\Theta^*)_{(l,\tau)}|\zeta^2\\
+&4\int_{B_{\rho}(y)}|(\Theta^*)_{(l,\tau)}|^2|d\Theta^*||A^*|\zeta^2+2\int_{B_{\rho}(y)}|(\Theta^*)_{(l,\tau)}||(d\Theta^*)_{l,\tau}||A^*|\zeta^2\\
+&2\int_{B_{\rho}(y)}|(A^*)_{(l,\tau)}||d\Theta^*||(\Theta^*)_{(l,\tau)}|\zeta^2=T_1+T_2+T_3+T_4+T_5.
\end{split}
\end{eqnarray}
Firstly let us estimate $T_1$. Indeed, by treating $\zeta(\Theta^*)_{(l,\tau)}$ as $\xi$ we could apply \eqref{firstorder}. This gives
\begin{eqnarray}
\begin{split}
\label{t1}
|T_1|\le C\rho^{2\beta_0}\big{(}\int_{B_{\rho}(y)}|(d\Theta^*)_{(l,\tau)}|\zeta^2+\int_{B_{\rho}(y)}|(\Theta^*)_{(l,\tau)}|^2|d\zeta|^2+\int_{B_{\rho}(y)}|A^*|^2|(\Theta^*)_{(l,\tau)}|\zeta^2\big{)}.
\end{split}
\end{eqnarray}
To the last term above we use the following combination of Cauchy inequality and Poincar\'e-Sobolev inequality:
\begin{eqnarray}
\begin{split}
&\int_{B_{\rho}(y)}|A^*|^2|(\Theta^*)_{(l,\tau)}\zeta|^2\le (\int_{B_{\rho}(y)}|A^*|^4)^{1/2}\cdot (\int_{B_{\rho}(y)}|(\Theta^*)_{(l,\tau)}\zeta|^4)^{1/2}\\
\le &C(\int_{B_{\rho}(y)}|A^*|^4)^{1/2}\cdot\int_{B_{\rho}(y)}|d((\Theta^*)_{(l,\tau)}\zeta)|^2\le C(\int_{B_{\rho}(y)}|A^*|^4)^{1/2}\cdot\int_{B_{\rho}(y)}|(d\Theta^*)_{(l,\tau)}\zeta|^2\\
+&C(\int_{B_{\rho}(y)}|A^*|^4)^{1/2}\cdot\int_{B_{\rho}(y)}|(\Theta^*)_{(l,\tau)}|^2|d\zeta|^2.
\end{split}
\end{eqnarray}
Now insert this back to \eqref{t1}, we have
\begin{eqnarray}
\begin{split}
T_1\le \tilde{\delta}\big{(}\int_{B_{\rho}(y)}|(d\Theta^*)_{(l,\tau)}\zeta|^2+ \int_{B_{\rho}(y)}|(\Theta^*)_{(l,\tau)}|^2|d\zeta|^2\big{)}.
\end{split}
\end{eqnarray}
for some small $\tilde{\delta}$ (to be determined later) provided both $\rho$ and $\epsilon_0$ are small. Secondly, we estimate $T_2$. Using Young's inequality with $\epsilon$ we have
\begin{eqnarray}
\begin{split}
T_2\le \epsilon \int_{B_{\rho}(y)}|(d\Theta^*)_{(l,\tau)}\zeta|^2+C(\epsilon)T_1.
\end{split}
\end{eqnarray}
Thirdly we estimate $T_3$. Using standard Young's inequality we obtain
\begin{eqnarray}
\begin{split}
T_3\le& CT_1+ \int_{B_{\rho}(y)}|A^*|^2|(\Theta^*)_{(l,\tau)}\zeta|^2\le  CT_1+C(\int_{B_{\rho}(y)}|A^*|^4)^{1/2}\cdot\int_{B_{\rho}(y)}|(d\Theta^*)_{(l,\tau)}\zeta|^2\\
&+C(\int_{B_{\rho}(y)}|A^*|^4)^{1/2}\cdot\int_{B_{\rho}(y)}|(\Theta^*)_{(l,\tau)}|^2|d\zeta|^2.
\end{split}
\end{eqnarray}
where we have once again used Poincar\'e-Sobolev Inequality to the second term in the second inequality. Next, we estimate term $T_4$. Again using Young's Inequality, we achieve the following:
\begin{eqnarray}
\begin{split}
T_4&\le \epsilon \int_{B_{\rho}(y)}|(d\Theta^*)_{(l,\tau)}\zeta|^2+C(\epsilon)\int_{B_{\rho}(y)}|A^*|^2|(\Theta^*)_{(l,\tau)}\zeta|^2\le  \epsilon \int_{B_{\rho}(y)}|(d\Theta^*)_{(l,\tau)}\zeta|^2\\
+&C(\epsilon)(\int_{B_{\rho}(y)}|A^*|^4)^{1/2}\cdot\int_{B_{\rho}(y)}|(d\Theta^*)_{(l,\tau)}\zeta|^2+C(\epsilon)(\int_{B_{\rho}(y)}|A^*|^4)^{1/2}\cdot\int_{B_{\rho}(y)}|(\Theta^*)_{(l,\tau)}|^2|d\zeta|^2.
\end{split}
\end{eqnarray}
Lastly, let us estimate term $T_5$. Using standard Young's Inequality, we have
\begin{eqnarray}
\begin{split}
T_5\le \int_{B_{\rho}(y)}|(A^*)_{(l,\tau)}|^2\zeta^2 +CT_1.
\end{split}
\end{eqnarray}
Now by combining the above estimates of $T_1$ through $T_5$, inserting them into \eqref{absorbr1}, choosing $\epsilon$ and $\tilde{\delta}$ in both Young Inequalities to be small enough, and then taking $\epsilon_0$ small, we could let all terms containing $\int_{B_{\rho}(y)}|d((\Theta^*)_{(l,\tau)}\zeta)|^2$ be absorbed into $R_1$. Upon further arrangements we obtain
\begin{eqnarray}
\begin{split}
\int_{B_{\rho}(y)}|(d\Theta^*)_{(l,\tau)}|^2\zeta^2\le C\bigg{(}\int_{B_{\rho}(y)}|(\Theta^*)_{(l,\tau)}|^2|d\zeta|^2+\int_{B_{\rho}(y)}|(A^*)_{(l,\tau)}|^2\zeta^2\bigg{)}.
\end{split}
\end{eqnarray}
Now sending $\tau$ to $0$, summing up all $l$, we achieve \eqref{nabla2} for some small but positive $\rho\eqqcolon \rho_y$. This proves Proposition \ref{localprop0}. 
\end{proof}
Next, we shall refine the ineffective $L^4$ estimate in Proposition \ref{l4finite} to the desired effective $L^4$ estimate in Theorem \ref{eps} by using an argument somewhat similar to that given in \cite{Evans}, where the duality between Hardy space and BMO space also plays a key role. Once again we go back to the Uhlenbeck gauge we find in the beginning of the proof of Theorem \ref{eps}, and let $\Theta^*$ be the same as in \eqref{trivialuh}. Proposition \ref{l4finite} guarantees that
\begin{eqnarray}
\begin{split}
\label{ineff}
\int_{B_{2\rho_0}(y_0)}|d\Theta^*|^4+|\nabla^2\Theta^*|^2<\infty.
\end{split}
\end{eqnarray}
The following proof frequently uses \eqref{ineff} without mentioning it. Our goal is to improve \eqref{ineff} to effective $L^4$-estimates. The key step is to achieve the following estimate:
\begin{proposition}
\label{simon3}
Given any $\delta>0$ and $0<\alpha<1$, there exists $\epsilon_0$ and $\eta_0$ as in Theorem \ref{eps}, and a large constant $C>0$ depending on $\delta$ and $\alpha$, such that for all $B_s(y)$ satisfying $B_{2s}(y)\subseteq B_{2\rho_0}(y_0)$, the following estimate holds:
\begin{eqnarray}
\begin{split}
\label{simon4}
\int_{B_{s/2}(y)}|d\Theta^*|^4\le \delta\int_{B_{s}(y)}|d\Theta^*|^4+C\bigg{(}(\frac{s}{\rho_0})^{4\alpha}\cdot \rho_0^{-4}\int_{B_{10\rho_0}(y_0)}|\Theta^*-\lambda_{y_0,10\rho_0}|^4+\int_{B_{10\rho_0}(y_0)}|F_A|^2\bigg{)}.
\end{split}
\end{eqnarray}
\end{proposition}
\begin{proof}
Fix arbitrary $B_s(y)$ satisfying $B_{2s}(y)\subseteq B_{2\rho_0}(y_0)$. Choose $\zeta$ to be a positive smooth cutoff function with support on $B_s(y)$ such that $\zeta\equiv 1$ on $B_{s/2}(y)$, $\zeta\equiv 0$ outside $B_{2s/3}(y)$, and $\sup|\nabla \zeta|\le 5s^{-1}$. Set $\xi=|d\Theta^*|\zeta$. Next, write $\Theta^*=(U_1,\cdots,U_k)$ and define $T_{ij}=\langle \nabla_A U_i,U_j\rangle_E$. The fact that $\Theta_0$ is Coulomb implies $d^*T_{ij}=0$ weakly, for every $i,j$. Moreover, it is a straightforward consequence from the equivalent condition (2) of Lemma \ref{coulomb} that the following identity holds in $B_s(y)$ in the weak sense:
\begin{equation}
\label{hardy1}
\Delta_A U_i=-\sum_{b=1}^k\langle T_{ib},\nabla_AU_b\rangle_{\Lambda^1(E)}=-\sum_{b=1}^k\langle T_{ib},dU_b\rangle_{\Lambda^1(E)}-\sum_{b=1}^k\langle T_{ib},A^*(U_b)\rangle_{\Lambda^1(E)}.
\end{equation}
Using these notations and employing integration by parts, we have
\begin{eqnarray}
\begin{split}
&\int_{B_s(y)}|\nabla_A \Theta^*|_{\Lambda^1(E^k)}^2\cdot \xi^2=\sum_i\int_{B_s(y)}|\nabla_A U_i|_{\Lambda^1(E)}^2\cdot\xi^2\\
=&\sum_i\int_{B_s(y)}\langle \Delta_AU_i,U_i\rangle_{E}\cdot \xi^2 +2\sum_i\int_{B_s(y)}\langle \nabla_AU_i,U_i\otimes \nabla\xi\rangle_{\Lambda^1(E^k)} \cdot\xi\\
=&T_1+T_2.
\end{split}
\end{eqnarray}
Due to the fact that $|U_i|_E^2\equiv 1$, trivially one has $T_2\equiv 0$. Then, we expand $T_1$ as follows:
\begin{eqnarray}
\begin{split}
\label{p1p2p3}
&T_1=-\sum_{i} \int_{B_s(y)}\big{\langle}\sum_b\langle T_{ib},dU_b\rangle_{\Lambda^1(E)}+\sum_b\langle T_{ib},A^*(U_b) \rangle_{\Lambda^1},U_i\big{\rangle}_E\cdot \xi^2\\
&=-\sum_{i} \bigg{\{}\sum_b\int_{B_s(y)}\big{\langle}\langle T_{ib},dU_b \rangle_{\Lambda^1},U_i\big{\rangle}_E\cdot \xi^2+\sum_b\int_{B_s(y)}\big{\langle}\langle \langle dU_i,U_b\rangle_{E},A^*(U_b)\rangle_{\Lambda^1}U_i\big{\rangle}_E\cdot \xi^2\\
&+\sum_b\int_{B_s(y)}\big{\langle}\langle \langle A^*(U_i),U_b\rangle_{E},A^*(U_b)\rangle_{\Lambda^1}U_i\big{\rangle}_E\cdot\xi^2\bigg{\}}\coloneqq -\{P_1+P_2+P_3\}.
\end{split}
\end{eqnarray}
We shall estimate $P_1$ through $P_3$ respectively. Since all norms are taken over $B_s(y)$, for convenience we omit the domain when expressing the norms. Let us start with estimating $P_1$. Thanks to the trivial identity $\big{\langle} T_{ib}\otimes U_b, U_i\big{\rangle}_E=0$, the following holds:
\begin{eqnarray}
\begin{split}
\label{hardy02}
\big{\langle}\langle T_{ib},dU_b \rangle_{\Lambda^1},U_i\big{\rangle}_E\cdot \xi^2=\big{\langle}\langle T_{ib},d(\xi^2U_b) \rangle_{\Lambda^1},U_i\big{\rangle}_E.
\end{split}
\end{eqnarray}
In view of the facts that $d^*T_{ib}=0$ weakly, $T_{ib}\in L^4$, and $d(\xi^2U_b) \in L^{4/3}$, the term $\langle T_{ib},d(\xi^2U_b) \rangle_{\Lambda^1}\in \mathcal{H}^1(B_s(y))$ (Hardy space) by \cite{CLMS93}. Plugging \eqref{hardy02} into $P_1$, and using the duality between Hardy space and BMO space, we estimate as follows:
\begin{eqnarray}
\begin{split}
\label{hardbmo}
|P_1|&\le \sum_{i,b}\|\langle T_{ib},d(\xi^2 U_b)\rangle\|_{\mathcal{H}^1}\cdot [\Theta^*]_{\text{BMO}}\le \sum_{i,b}\|T_{ib}\|_{L^4}\cdot\|d(\xi^2 U_b)\|_{L^{\frac{4}{3}}}\cdot [\Theta^*]_{\text{BMO}}.
\end{split}
\end{eqnarray}
Note we have the following trivial point-wise bound:
\begin{eqnarray}
\begin{split}
\label{pthard}
|T_{ib}|&\le |d\Theta^*|+|A^*|.
\end{split}
\end{eqnarray}
as well as the following estimate from H\"older inequality:
\begin{eqnarray}
\begin{split}
\label{normhard}
\|d(\xi^2 U_b)\|_{L^{\frac{4}{3}}}\le \|d \Theta^*\|_{L^4}^3+2\|\zeta\nabla^2\Theta^*\|_{L^2}\cdot \|d\Theta^*\|_{L^4}+2\|d\Theta^*\|_{L^4}^2.
\end{split}
\end{eqnarray}
Moreover, we note that exact the same arguments we used to prove \eqref{itnow} gives the following estimates on every ball $B_\rho(z)\subseteq B_{s}(y)$: given any $\alpha<1$, by first choosing $\theta$ and then $\epsilon$ to be small on the right hand side of \eqref{simon} such that both depend on $\alpha$, we arrive at the following estimate:
\begin{eqnarray}
\begin{split}
(\theta\rho)^{-4}\int_{B_{\theta\rho}(z)}|\Theta^*-\lambda_{x,\theta\rho}|^2\le  \theta^{2\alpha}\bigg{(}\rho^{-4}\int_{B_{\rho}(z)}|\Theta^*-\lambda_{z,\rho}|^2+\big{(}\int_{B_{\rho}(z)}|A^*|^4\big{)}^{1/2}\bigg{)}.
\end{split}
\end{eqnarray}
Now iteratively apply the above estimate similarly as in \eqref{itera1}, we obtain for arbitrary integer $K>0$ the following estimate:
\begin{eqnarray}
\begin{split}
\label{itera2}
\phi_{z}(\rho_K)&\le (\frac{\rho_K}{\rho_0})^{2\alpha}\phi_{z}(\rho_0)+\theta^{-2\alpha}\cdot\sum_{i=1}^K\theta^{2i\alpha}\cdot \big{(}\int_{B_{\rho_{K-i}}(z)}|A^*|^4\big{)}^{1/2}\\
&=(\frac{\rho_K}{\rho_0})^{2\alpha}\phi_{z}(\rho_0)+C_{\theta,\alpha}\cdot\big{(}\int_{B_{\rho_0}(z)}|A^*|^4\big{)}^{1/2}\\
&\le C(\frac{\rho_K}{\rho_0})^{2\alpha}\phi_{y_0}(10\rho_0)+C_{\theta,\alpha}\cdot\big{(}\int_{B_{\rho_0}(z)}|A^*|^4\big{)}^{1/2}.
\end{split}
\end{eqnarray}
where $\rho_i=\theta^i\rho_0$. This implies
\begin{eqnarray}
\begin{split}
\label{bmobd}
[\Theta^*]_{\text{BMO}(B_s(y))}\le C \bigg{(}(\frac{s}{\rho_0})^{\alpha}\phi_{y_0}(10\rho_0)^{1/2}+\big{(}\int_{B_{10\rho_0}(y_0)}|A^*|^4\big{)}^{1/4}\bigg{)}.
\end{split}
\end{eqnarray}
Combining \eqref{hardbmo}, \eqref{pthard}, \eqref{normhard}, and \eqref{bmobd}, we achieve
\begin{eqnarray}
\begin{split}
|P_1|\le& C\big{(}\|d\Theta^*\|_{L^4}+\|A^*\|_{L^4}\big{)}\cdot \big{(}\|d\Theta^*\|^3_{L^4}+\|\nabla^2\Theta^*\|_{L^2}\cdot \|d\Theta^*\|_{L^4}+\|d\Theta^*\|_{L^4}^2\big{)}\cdot [\Theta^*]_{\text{BMO}}\\
\le& C\bigg{(}\|d\Theta^*\|_{L^4}[\Theta^*]_{\text{BMO}}+\|\zeta\nabla^2\Theta^*\|_{L^2}\|d\Theta^*\|_{L^4}^2[\Theta^*]_{\text{BMO}}+\|d\Theta^*\|_{L^4}^3[\Theta^*]_{\text{BMO}}\\
&+\|A^*\|_{L^4}\|\Theta^*\|_{L^4}^3[\Theta^*]_{\text{BMO}}+\|A^*\|_{L^4}\|\zeta\nabla^2\Theta^*\|_{L^2}\|d\Theta^*\|_{L^4}[\Theta^*]_{\text{BMO}}\\
&+\|A^*\|_{L^4}\|d\Theta^*\|_{L^4}^2[\Theta^*]_{\text{BMO}}\bigg{)}=Q_1+Q_2+Q_3+Q_4+Q_5+Q_6.
\end{split}
\end{eqnarray}
In estimating $Q_1, Q_2, Q_4, Q_5$, we shall use the fact that $[\Theta^*]_{\text{BMO}}<\delta_0^{1/2}$, which is a consequence of applying Corollary \ref{a0bound} with some sufficiently small $\delta_0$ to be determined. Especially, in treating the term containing $\nabla^2\Theta^*$, we use the following inequality, which is a consequence of the first order elliptic interior estimate and the fact that $\Theta_0$ is Coulomb:
\begin{eqnarray}
\begin{split}
\label{cou21}
\|\zeta\nabla^2\Theta^*\|_{L^2}\le C\bigg{(}\|F_A\|_{L^2}+\|A^*\|_{L^4}^2+\|d\Theta^*\|_{L^4}^2\bigg{)}.
\end{split}
\end{eqnarray}
Then we estimate as follows at the following:
\begin{eqnarray}
\begin{split}
\label{1245}
Q_1+Q_2+Q_4+Q_5\le C(\delta_0^{1/2}+\delta_0+\delta_0^{4/3}+\delta_0^2)\|d\Theta^*\|_{L^4}^4+C\|A^*\|_{L^4}^4+C\|F_A\|_{L^2}^2.
\end{split}
\end{eqnarray}
To estimate $Q_3$ and $Q_6$, we use \eqref{bmobd} as well as Young's inequality with $\epsilon$, for some sufficiently small $\epsilon$ to be specified:
\begin{eqnarray}
\begin{split}
\label{36}
Q_3+Q_6\le& \epsilon\|d\Theta^*\|_{L^4}^4+ C\|A^*\|_{L^4}^4\\
&+C(\epsilon)\bigg{(} (\frac{s}{\rho_0})^{4\alpha}\rho_0^{-4}\int_{B_{10\rho_0}(y_0)}|\Theta^*-\lambda_{10\rho_0,y_0}|^4+\int_{B_{10\rho_0}(y_0)}|A^*|^4\bigg{)}.
\end{split}
\end{eqnarray}
Choosing $\epsilon,\delta_0$ sufficiently small, and combining \eqref{1245} with \eqref{36}, we arrive at the following estimate of $P_1$ in \eqref{p1p2p3}:
\begin{eqnarray}
\begin{split}
\label{p1only}
|P_1|\le \delta/3\int_{B_{s}(y)}|d\Theta^*|^4+C(\delta)\bigg{(} (\frac{s}{\rho_0})^{4\alpha}\rho_0^{-4}\int_{B_{10\rho_0}(y_0)}|\Theta^*-\lambda_{10\rho_0,y_0}|^4+\int_{B_{10\rho_0}(y_0)}|F_A|^2\bigg{)}.
\end{split}
\end{eqnarray}
In addition, one easily arrives at the following estimates for $P_2$ and $P_3$ in \eqref{p1p2p3} simply by using Young's inequality with $\epsilon$:
\begin{eqnarray}
\begin{split}
\label{p2p3}
|P_2|+|P_3|\le \delta/3\int_{B_s(y)}|d\Theta^*|^4+C(\delta)\int_{B_s(y)}|A^*|^4.
\end{split}
\end{eqnarray}
On the other hand, we have the following again by only using Young's inequality with $\epsilon$:
\begin{eqnarray}
\begin{split}
\label{rhscut}
\int_{B_{s/2}(y)}|d\Theta^*|^4\le \int_{B_{s}(y)}|d\Theta^*|^4\zeta^2\le \int_{B_{s}(y)}|\nabla_A\Theta^*|^2\xi^2+\delta/3\int_{B_s(y)}|d\Theta^*|^4+C(\delta)\int_{B_s(y)}|A^*|^4.
\end{split}
\end{eqnarray}
Finally, combining \eqref{p1p2p3}, \eqref{p1only}, \eqref{p2p3}, \eqref{rhscut}, as well as \eqref{uhgauge}, we obtain \eqref{simon4}, and thus finished the proof of Proposition \ref{simon3}.
\end{proof}
Now, by choosing $\delta$ to be small enough, setting $B_{s}(y)\eqqcolon B_\rho(x)$ (such that $B_{2s}(y)\subseteq B_{2\rho_0}(y_0)$), $k=0$, $\phi(B)=\int_B|d\Theta^*|^4$, and then defining
\begin{eqnarray}
\begin{split}
\gamma_s=C\bigg{(}(\frac{s}{\rho_0})^{4\alpha}\cdot \rho_0^{-4}\int_{B_{10\rho_0}(y_0)}|\Theta^*-\lambda_{y_0,10\rho_0}|^4+\int_{B_{10\rho_0}(y_0)}|F_A|^2\bigg{)},
\end{split}
\end{eqnarray}
we conclude from Lemma \ref{simonbook} and \eqref{simon4} that
\begin{eqnarray}
\begin{split}
&\int_{B_{s/2}(y)}|d\Theta^*|^4=\phi(B_{s/2}(y))\le C\gamma\\
=&C\bigg{(}(\frac{s}{\rho_0})^{4\alpha}\cdot \rho_0^{-4}\int_{B_{10\rho_0}(y_0)}|\Theta^*-\lambda_{y_0,10\rho_0}|^4+\int_{B_{10\rho_0}(y_0)}|F_A|^2\bigg{)}.
\end{split}
\end{eqnarray}
Especially, this implies
\begin{eqnarray}
\begin{split}
\label{rho0}
\int_{B_{\rho_0}(y_0)}|d\Theta^*|^4\le C\bigg{(}\rho_0^{-4}\int_{B_{10\rho_0}(y_0)}|\Theta^*-\lambda_{y_0,10\rho_0}|^4+\int_{B_{10\rho_0}(y_0)}|F_A|^2\bigg{)}.
\end{split}
\end{eqnarray}
From \eqref{rho0}, \eqref{cou21}, and \eqref{uhgauge}, one sees that \eqref{hess} follows immediately. Thus we complete the proof of Theorem \ref{eps}.
\end{proof}
\bigskip

\section{Stability and compactness}
\label{sta}
In this section, we will prove a $W^{1,2}$ sequential compactness result for Coulomb minimizers by employing both the stability (see \eqref{stab1}) and the stationarity (see \eqref{divfree}) of the minimizers. This extends a result in \cite{HW99}, which establishes the $W^{1,2}$ sequential compactness of the stable stationary harmonic maps into spheres. Our main theorem in this section is stated as follows:
\begin{theorem}
\label{cpt}
Let $G$ be a simply-connected simple compact Lie group. Given a sequence of smooth connections $\{A_i\}_{i=1}^{\infty}$ defined over trivial bundle $P_0=B_{10\rho}(p)\times G$, such that
\begin{eqnarray}
\begin{split}
\label{meconn}
A_i\to A_{\infty}\ \text{in}\ W^{1,2}(B_{10\rho}(p))
\end{split}
\end{eqnarray} 
for some $W^{1,2}$ connection $A_{\infty}$. For each $i$ let $\Theta_i$ be a stationary stable critical point of the functional 
$$\int_{B_{10\rho}(p)}|\nabla_{A_i} \Theta|^2d\text{V}_g$$ 
with associated connection form $A_{0,i}$. Then there exists a $W^{1,2}$-frame $\Theta_\infty$ of $P_0$ and a subsequence $\{A_{0,i_l}\}_{l}$ such that
\begin{eqnarray}
\begin{split}
\label{connconv}
\lim_{l\to \infty}\|A_{0,i_l}-A_{0,\infty}\|_{L^2(B_{\rho}(p))}=0,
\end{split}
\end{eqnarray}
where $A_{0,\infty}$ is a connection form associated to $A_{\infty}$ under the frame $\Theta_\infty$. Moreover, $d^*A_{0,\infty}=0$ weakly. 
\end{theorem}
If instead of \eqref{meconn}, we assume that
\begin{eqnarray}
\begin{split}
\label{addas}
\|F_{A_i}\|_{L^2(B_{10\rho}(p,g_i))}\to 0,\ g_i\rightarrow g_{\text{Euc}}\ \text{in }C^{1,\alpha}(B_{10\rho}(p))
\end{split}
\end{eqnarray}
Then we can find for each $i$ an Uhlenbeck gauge $\sigma_i$ of for $A_i$ with connection form $A_i^*$. In other words, $\sigma_i: B_{10\rho}(p,g_i)\to P$ such that $\sigma_i^*A=A_i^*$ and $A^*_i$ satisfies \eqref{uhgauge}. Then $A_i^*$ subconverges to some $A_\infty^*$ in $L^2$ (see \cite{W04}, Chapter 6). Moreover, $A^*_\infty$ is the connection form of $A_\infty$ under some Uhlenbeck gauge $\sigma_\infty$. Now denote by $\Theta^*_i$ and $\Theta^*_\infty$ the trivialization of $\Theta_i$ and $\Theta_\infty$ under $\sigma_i$ and $\sigma_\infty$ respectively. The following corollary will be useful in later sections:
\begin{corollary}
\label{keycoro}
In Theorem \ref{cpt} if we replace \eqref{meconn} by \eqref{addas}, then we have 
\begin{equation}
\label{keco}
\lim_{l\to \infty}\|\Theta^*_i-\Theta^*_\infty\|_{W^{1,2}(B_{\rho}(p,g_{\text{Euc}}))}=0.
\end{equation}
Moreover, $\Theta_\infty$ is a stable stationary harmonic map defined on $B_{\rho}(p,g_{\text{Euc}})$.
\end{corollary}
The proof Theorem \ref{cpt} of relies on the following local estimate, where the simpleness of $G$ plays a crucial role:
\begin{theorem}
\label{stathm}
Let $P\to M$ be a smooth principal bundle with simply-connected simple compact Lie group $G$-fiber over a compact $4$-manifold and $A$ a smooth connection on $P$. Let $\Theta_0$ be a Coulomb minimizer obtained in Subsection \ref{prep}, with associated connection form $A_0$. Then there exists $r(M)$, $\epsilon_0(G)$, and $C(G)$, such that for any ball $B_{\rho}(p)\subseteq M$ satisfying $10\rho<r(M)$ and $\int_{B_{\rho}(p)}|F_A|^2<\epsilon_0$ (with $\epsilon_0$ so small that we could fix an Uhlenbeck-gauge $A^*$ on $B_{\rho}(p)$ satisfying \eqref{uhgauge}), and any $\zeta\in C^{\infty}_c(B_{\rho}(p))$, the following holds:
\begin{equation}
\label{sta1}
\int_{B_{\rho}(p)}|A_0|^2\zeta^2\le C(G)\big{(}\int_{B_{\rho}(p)}|A^*|^2\zeta^2+\int_{B_{\rho}(p)}|d\zeta|^2\big{)}.
\end{equation}
\end{theorem}
\begin{proof}
Let $\mathfrak{g}$ denote the Lie algebra associated to $G$. Choose any smooth mapping $\xi: [0,t_0]\to C_c^{\infty}(B_{\rho}(y),\mathfrak{g})$ such that $\xi(0)\equiv 0$. Plugging $\xi$ into the stability inequality \eqref{stab1}, we have
\begin{eqnarray}
\begin{split}
\label{nonneg}
\frac{d^2}{dt^2}\bigg{\rvert}_{t=0}\int_{B_{\rho}(p)}|\nabla_A\big{(}\exp(\xi(t))\cdot\Theta^*\big{)}|^2\ge 0.
\end{split}
\end{eqnarray}
On the other hand, by expanding the integrand of the left hand side above, we obtain
\begin{eqnarray}
\begin{split}
\label{long}
&\frac{d^2}{dt^2}\bigg{\rvert}_{t=0}|\nabla_A\big{(}\exp(\xi(t))\cdot\Theta^*\big{)}|^2\\
=&\bigg{(}2 \langle \big{(}[\nabla_A\xi^{\prime}(0),\xi^{\prime}(0) ]+\nabla_A\xi^{\prime\prime}(0)\big{)}(\Theta^*),A^*(\Theta^*) \rangle+|[A^*,\xi](\Theta^*)|^2+2\langle [A^*,\xi^{\prime}(0)](\Theta^*),d\xi^{\prime}(0)\rangle\\
&+2\langle [[A^*,\xi^{\prime}(0)](\Theta^*),\xi^{\prime}(0)](\Theta^*), d\Theta^* \rangle+2\langle [A^*,\xi^{\prime\prime}(0)](\Theta^*),d\Theta^* \rangle\bigg{)}\\
&+\bigg{(}2|d\xi^{\prime}(0)|^2+2\langle \big{(}[d\xi^{\prime}(0),\xi^{\prime}(0)]+d\xi^{\prime\prime}(0)\big{)}(\Theta^*),d\Theta^* \rangle\bigg{)}\coloneqq e_{\xi}+m_{\xi},
\end{split}
\end{eqnarray}
where $e_{\xi}$ denotes the first big parenthesis while $m_{\xi}$ denotes the second one. Then set
\begin{eqnarray}
\begin{split}
\mathcal{E}_\xi\eqqcolon\int_{B_\rho(p)}e_\xi,\ \mathcal{M}_\xi\eqqcolon \int_{B_\rho(p)}m_\xi.
\end{split}
\end{eqnarray}
It is easily seen by an elementary computation that
\begin{eqnarray}
\begin{split}
\label{m}
\mathcal{M}_{\xi}=&\frac{d^2}{dt^2}\bigg{\rvert}_{t=0}\int_{B_{\rho}(p)}|d\big{(}\exp(\xi(t))\cdot\Theta^*\big{)}|^2.
\end{split}
\end{eqnarray}
Next, we need to construct proper variational vector fields $\xi$. For starters, thanks to the irreducibility of $G$, $G$ admits an isometric embedding into an Euclidean space, whose coordinate functions are given by the first eigenspace $\{\phi_i\}_{i=1}^p$ of the Laplacian on $G$ with respect to the bi-invariant metric. This follows from the basics of representation theory. For details, we refer the readers to \cite{LB80}. Now consider such an embedding $\iota$. Let $\{\alpha_{i}\}_{i=1}^p$ denote the affine vector fields along the Euclidean coordinates, then let $\alpha_i^T$ denote the projection of $\alpha_i$ onto the tangent space of $\iota(G)$. Next, for any $\zeta\in C^{\infty}_c(B_{\rho}(p))$ define $v_i\eqqcolon \zeta\alpha_i^T$. For each $x\in B_{\rho}(y)$, we then have an integral curve $\phi_{i,x,t}$ starting from $x$ and generated by $v_i$ (which is defined on some $[0,\epsilon]$ and takes value in $\iota(G)$). Via the isometric identification $\iota$, we could find a smooth map $\xi_i:[0,\epsilon]\to C_c^{\infty}(B_{\rho}(y),\mathfrak{g})$ for each $v_i$ such that $\iota(\exp(\xi_i)\cdot\Theta^*)=\phi_{i,x,t}\circ \iota(\Theta^*)$. Using elementary computations one could check that for each $i=1,\cdots,p$ and some constant $C\equiv C(G)$ the following estimates hold:
\begin{eqnarray}
\begin{split}
\label{iden}
|\xi_i^{\prime}(0)|\le C\zeta,\ |d\xi_i^{\prime}(0)|\le C|d\zeta|,\ |\xi_i^{\prime\prime}(0)|\le C\zeta^2,\ |d\xi_i^{\prime\prime}(0)|\le C\zeta|d\zeta|.
\end{split}
\end{eqnarray}
Inserting the variational field $\xi_i$ into \eqref{m}, then using \eqref{weaklap} and integration by parts, and finally summing $\mathcal{M}_{\xi_i}$ over $i$ from $1$ to $p$, we have
\begin{eqnarray}
\begin{split}
\label{summ}
\sum_{i=1}^p\mathcal{M}_{\xi_i}=&m\int_{B_{\rho}(p)}|d\zeta|^2-\frac{2\tau-m\lambda_1}{m}\int_{B_{\rho}(p)}|d\Theta^*|^2\zeta^2+\int_{B_{\rho}(p)}Q(\Theta^*,d\Theta^*,A^*)\zeta^2\\
=&U_1+U_2+U_3,
\end{split}
\end{eqnarray}
where $m=\text{dim}(G)$, $\lambda_1$ denotes the first eigenvalue of $G$, and $\tau$ stands for the scalar curvature of $G$; moreover, the term $Q(\Theta^*,d\Theta^*,A^*)$ is multilinear in $\Theta^*$, $d\Theta^*$, and $A^*$, satisfying 
\begin{eqnarray}
\begin{split}
\label{mulli}
|Q(\Theta^*,d\Theta^*,A^*)|\le C|d\Theta^*||A^*|.
\end{split}
\end{eqnarray}
In deriving \eqref{summ} we closely followed \cite{HW86}. However, the computation is lengthy and irrelevant to the main goal of this paper, therefore we choose to omit it, and refer the reader to \cite{HW86} for its proof. On the one hand, by \eqref{mulli} as well as Young's Inequality with $\epsilon$ (which will be determined later), we could estimate $U_3$ as follows:
\begin{eqnarray}
\begin{split}
\label{u3}
|U_3|\le \epsilon \int_{B_{\rho}(p)}|d\Theta^*|^2\zeta^2+C(\epsilon)  \int_{B_{\rho}(p)}|A^*|^2\zeta^2.
\end{split}
\end{eqnarray}
According to \cite{N82}, if $G$ is simply-connected irreducible and compact, then 
\begin{eqnarray}
\begin{split}
\label{simple}
2\tau-m\lambda_1>0. 
\end{split}
\end{eqnarray}
See also \cite{HW86} for the proof of this fact. Lastly, to estimate $\sum_{i=1}^p\mathcal{E}_{\xi_i}$ we use \eqref{iden} and again Young's Inequality with $\epsilon$ to obtain
\begin{eqnarray}
\begin{split}
\label{e}
|\sum_{i=1}^p\mathcal{E}_{\xi_i}|\le  CC(\epsilon)\int_{B_{\rho}(p)}|A^*|^2\zeta^2+C\int_{B_{\rho}(p)}|d\zeta|^2+C\epsilon\int_{B_{\rho}(p)}|d\Theta^*|^2\zeta^2.
\end{split}
\end{eqnarray}
Now combine \eqref{nonneg}, \eqref{long}, \eqref{summ}, \eqref{simple}, \eqref{u3}, \eqref{e}, and then choose $\epsilon$ small, we obtain:
\begin{eqnarray}
\begin{split}
\label{statheta}
\int_{B_{\rho}(p)}|d\Theta^*|^2\zeta^2\le C_1(G)\big{(}\int_{B_{\rho}(p)}|A^*|^2\zeta^2+\int_{B_{\rho}(p)}|d \zeta|^2\big{)}.
\end{split}
\end{eqnarray}
From this and the fact that $|A_0|\le |A^*|+|d\Theta^*|$, we immediately obtain \eqref{sta1}.
\end{proof}
\begin{proof}[Proof of Theorem \ref{cpt}]
Consider a cover $B_{5\rho}(p)$ given by $\{B_{r_a}(y_a)\}_{a=1}^{S}$ (where $S$ is a finite number) such that $\int_{B_{10r_a}(y_a)}|F_{A_{\infty}}|^2<\epsilon_0/2$, where $\epsilon_0$ is so chosen that fulfills the curvature smallness hypothesis in Theorem \ref{eps} and Theorem \ref{stathm}. By \eqref{meconn}, we see that $\int_{B_{10r_a}(y_a)}|F_{A_{i}}|^2<2\epsilon_0/3$ for all $a$, when $i$ is large enough. Let $\{A^*_{a,i}\}_{a,i}$ be the connection forms of $\{A_i\}_i$ under the Uhlenbeck gauges $\{\sigma_{a,i}\}_i$ defined on $\{B_{10r_a}(y_a)\}_a$ that satisfies \eqref{uhgauge}. More importantly, thanks to the strong $W^{1,2}$ convergence of $A_i$ as in \eqref{meconn}, one might assume (up to passing to a subsequence) that
\begin{equation}
\label{uhcon}
A_{i,a}^*\to A^*_{\infty,a}\ \text{in $L^2(B_{10r_a}(y_a))$}
\end{equation}
for all $a$, where $A^*_{a,\infty}$ is the connection form gauge of $A_\infty$ under some Uhlenbeck gauge $\sigma_{a,\infty}$. For the proof of \eqref{uhcon}, see Chapter 6 of \cite{W04}. See also Section \ref{asc}. Next, let us note that $\int_{B_{5\rho}(p,g_i)}|A_{0,i}|^2$ is uniformly bounded in $i$; indeed, use the covering property of $\{B_{r_a}(y_a)\}_a$ and apply Theorem \ref{stathm} by choosing a proper cutoff function $\zeta$, we have:
\begin{equation}
\label{kbd}
\sup_{i}\int_{B_{5\rho}(p,g_i)}|A_{0,i}|^2\le C \sum_{a=1}^S r_a^2\coloneqq K. 
\end{equation}
Consider the set:
\begin{equation}
\label{singsigma}
\Sigma\eqqcolon \bigcup_{1\le a\le S}\bigcap_{0<r<r_a}\bigg{\{}x\in B_{5\rho}(p)\cap B_{r_a}(y_a): \liminf_{i\to \infty}r^{-2}\int_{B_{r}(x,g_i)}|A_{0,i}|^2 >\eta/2\bigg{\}}
\end{equation}
where $\eta$ is chosen as in Theorem \ref{eps}. For any point $x\notin\Sigma$, we could find some ball $B_{r}(x)\subseteq B_{2r_a}(y_a)$ for some $a$ such that: 
\begin{equation}
\label{small1}
r^{-2}\int_{B_{r}(x,g_i)}|A_{0,i}|^2 \le \eta/2
\end{equation}
for all sufficiently large $i$. We could then cover $B_{5\rho}(p)\backslash \Sigma$ by balls $\{B_{s_b}(z_b)\}_b$ such that for each $b$, \eqref{small1} is satisfied with $B_{r}(x,g_i)$ replaced by $B_{20s_b}(y_b,g_i)$, for all large $i$. Now apply Theorem \ref{eps} to each of them for all large $i$, we see that:
\begin{equation}
\label{140}
\int_{B_{s_b}(y_b)}|\nabla A_{0,i}|^2\le C\bigg{(}{s_b}^{-2}\int_{B_{20s_b}(y_b)}|A_{0,i}|^2+\int_{B_{20s_b}(y_b)}|F_{A_{i}}|^2\bigg{)}.
\end{equation}
Thus, by Sobolev embedding, upon passing to a diagonal sequence there exists $A_{0,\infty}$ and a subsequence $A_{0,i_l}$ such that $A_{0, i_l}\to A_{0,\infty}$ strongly in $L^{2}_{\text{loc}}(B_{5\rho}(p)\backslash \Sigma)$. In order to prove $A_{0, i_l}\to A_{0,\infty}$ is in fact strong in $L^{2}(B_{3\rho}(p))$, we simply need to prove it for each $B_{r_a}(y_a)$. For this purpose, let us fix arbitrary $a$. We will find for each $\tau$ a cutoff function $\zeta_{\tau}\in C^{\infty}_c(B_{3r_a}(y_a))$ such that $\zeta_{\tau}\equiv 1$ in a neighborhood of $\Sigma\cap B_{2r_a}(y_a)$, and moreover, the following holds
\begin{equation}
\label{tau}
\sup_{l}\int_{B_{3r_a}(y_a)}|A_{0,i_l}|^2\zeta_{\tau}^2<\tau.
\end{equation}
Let us admit \eqref{tau} for now. Then by the arbitrariness of $\tau$ together with the fact that $A_{0, i_l}\to A_{0,\infty}$ strongly in $L^{2}_{\text{loc}}(B_{5\rho}(p)\backslash \Sigma)$, we see that $A_{0, i_l}\to A_{0,\infty}$ strongly in $L^2(B_{r_a}(y_a))$ for every $a$. Now using with a diagonal argument, we proved the strong $L^2$-subconvergence \eqref{connconv}. Hence let us focus on proving \eqref{tau}. Fix arbitrary $\tau>0$. Using the uniform upper bound \eqref{kbd}, it follows from a Vitali covering argument that the Hausdorff measure $H^2(\Sigma\cap B_{3r_a}(y_a))<\infty$ for each $a$. Therefore, the $2$-capacity $\text{Cap}_2(\Sigma\cap B_{3r_a}(y_a))=0$, and hence for an arbitrarily small neighborhood of $\Sigma\cap B_{3r_a}(y_a)$ denoted by $\mathcal{N}$, one could find a function $\zeta_{\tau}$ such that $\zeta_{\tau} \in C^{\infty}_c(B_{3r_a}(y_a)\cap \mathcal{N})$ such that $\zeta_{\tau}\equiv 1$ in a neighborhood of $\Sigma\cap B_{2r_a}(y_a)$; moreover, the following holds:
\begin{equation}
\int_{B_{3r_a}(y_a)}|d\zeta_{\tau}|^2<\tau/2C(G).
\end{equation}
for the same $C(G)$ as in Theorem \ref{stathm}. Now apply \eqref{sta1}, we have:
\begin{equation}
\label{tau1}
\int_{B_{3r_a}(y_a)}|A_{0,i}|^2\zeta_{\tau}^2<C(G)\big{(}\int_{B_{3r_a}(y_a)}|A^*_{a,i}|^2\zeta_{\tau}^2+\frac{\tau}{2C(G)}\big{)}.
\end{equation}
Due to \eqref{uhcon}, the first term uniformly go to zero as the support $\mathcal{N}$ shrinks. Hence, by further shrinking $\mathcal{N}$ beforehand (if necessary) we obtain
\begin{equation}
\label{tauover}
\int_{B_{3r_a}(y_a)}|A^*_{a,i}|^2\zeta_{\tau}^2\le \frac{\tau}{2C(G)}.
\end{equation}
Then we see that \eqref{tau} follows from \eqref{tau1} and \eqref{tauover}.
\medskip

Let $\Theta^*_i$ be the trivialization of $\Theta_i$ under $\sigma_{a,i}$. Note an almost identical argument as showing \eqref{connconv} proves $\Theta^*_{a,i}\to \Theta^*_{a,\infty}$ strongly in $W^{1,2}$ and weakly in $W^{2,2}_{\text{loc}}(B_{5\rho}(p)\backslash \Sigma)$ up to a subsequence; the only difference is that in the places where we used \eqref{sta1} and \eqref{140}, this time one should use \eqref{statheta} and 
\begin{equation}
\label{140}
\int_{B_{s_b}(y_b)}|\nabla^2 \Theta^*_i|^2\le C\bigg{(}{s_b}^{-2}\int_{B_{20s_b}(y_b)}|A_{0,i}|^2+\int_{B_{20s_b}(y_b)}|F_{A_{i}}|^2\bigg{)}.
\end{equation}
respectively. Note \eqref{140} is a easy corollary of Theorem \ref{eps}. To find $\Theta_\infty$, it suffices to find its associated $P$-section $s_\infty$ (see Proposition \ref{ggauge}). Let $s_i$ be the $P$-sections such that $s_i^*A_i=A_{0,i}$. Choose a (finite) collection of smooth local $P$-section $\{\mathcal{S}_l\}_l$ such that $\mathcal{S}_l$ is defined over a ball $B_l$ and that $\bigcup_l B_l=M$. For each $l$ we estimate as follows:
\begin{eqnarray}
\begin{split}
\label{patchupnow}
\|ds_i\|_{L^2(B_l)}&=\|s_i^{-1}ds_i\|_{L^2(B_l)}\le \|s_i^*A_i\|_{L^2(B_l)}+\|u_i^{-1}A_iu_i\|_{L^2(B_l)}\\
&=\|A_{0,i}\|_{L^2(B_l)}+\|A_i\|_{L^2(B_l)}\le C_0,
\end{split}
\end{eqnarray}
where $\|ds_i\|_{L^2(B_l)}$ and $\|A_i\|_{L^2(B_l)}$ are taken with respect to the local trivialization $\mathcal{S}_l$, and $C_0$ is a constant independent of $i$. In the last inequality above we have used \eqref{connconv} and the assumption \eqref{meconn}. Hence, by the arbitrariness of $l$ and up to passing to a subsequence, $s_i$ has a weak limit $s_\infty\in W^{1,2}$. The strong convergence $s_i\to s_\infty$ also holds in $L^2$ topology due to the compact Sobolev embedding $W^{1,2}\hookrightarrow L^2$. This implies 
$$s_i^*A_i=s_i^{-1}A_is_i+s_i^{-1}ds_i \to s_\infty^{-1}A_{\infty}s_\infty+s_{\infty}^{-1} ds_\infty=s_\infty^*A_\infty\ \text{a.e.}.$$
On the other hand, $s_i^*A_i=A_{0,i}\to A_{0,\infty}$ in $L^2$, from which we conclude that $s_{\infty}^*A_{\infty}=A_{0,\infty}$. Finally, set $\Theta_\infty$ as the frame associated to $s_\infty$ (see Proposition \ref{ggauge}). In other words, $A_{0,\infty}$ is the connection form of $A_\infty$ under the frame $\Theta_\infty$, and one easily verifies that $\Theta^*_{a,\infty}$ is the trivialization of $\Theta_\infty$ with respect to $\sigma_{a,\infty}$.
Now applying a standard partition of unity argument combining with H\"older inequalities, one can easily show by using the $L^2$-convergence of $A_{0,i}$, $A_{a,i}^*$ together with the $W^{1,2}$-convergence of $\Theta^*_{a,i}$ that the weakly Coulomb equation $d^*A_{0,i}=0$ is preserved in the limit as we send $i$ to infinity. Since the proof is elementary we omit the details. Therefore, $d^*A_{0,\infty}=0$ weakly. This completes the proof of Theorem \ref{cpt}.
\end{proof}
\begin{proof}[Proof of Corollary \ref{keycoro}]
Upon choosing an Uhlenbeck gauge of $A_i$ for all sufficiently large $i$, the proof of \eqref{keco} is the same as the argument used in the above proof to show that $\Theta^*_{a,i}\to \Theta^*_{a,\infty}$ strongly in $W^{1,2}$ for each $a$. By \eqref{uhgauge} and \eqref{keco}, the limit bundle is a trivial bundle $B_\rho(p,g_{\text{Euc}})$ with flat connection $A_{\text{flat}}$. To see that $\Theta_\infty^*$ satisfies the stability \eqref{stab1} with $\nabla_A$ replaced by $d$, notice that \eqref{long} holds upon replacing $\Theta^*$ and $A^*$ by $\Theta_i$ and $A_i^*$. In \eqref{nonneg} and \eqref{long} replace $\Theta^*$ with $\Theta^*_i$, and then send $i$ to infinity in both inequalities. Now use the facts that $A_{i}^*\to A^*_{\infty,i}$ in $L^2(B_\rho(p))$ and $\Theta^*_{i}\to \Theta^*_{\infty}$ strongly in $W^{1,2}(B_\rho(p))$, we obtain
\begin{eqnarray}
\begin{split}
\label{nonneg1}
\frac{d^2}{dt^2}\bigg{\rvert}_{t=0}\int_{B_{\rho}(p)}|d\big{(}\exp(\xi(t))\cdot\Theta_\infty^*\big{)}|^2\ge 0.
\end{split}
\end{eqnarray}
Similarly, one sees that the $L^2$-convergence of $A_{i}^*$ together with the $W^{1,2}$-convergence of $\Theta^*_{i}$ shows that the stationarity equation \eqref{divfree} applied to $\Theta_i$ is preserved in the limit as we send $i$ to infinity, an becomes the stationarity equation for the map $\Theta_\infty^*$. This proves Corollary \ref{keycoro}.
\end{proof}
\bigskip

\section{Tangent cones and local singular structure}\label{tconestru}
From this section on, we concentrate on $\text{SU}(2)$-fibered bundles. In addition, the smoothness assumption on $A$ will be of crucial importance in studying the tangent cone structure of the Coulomb minimizer $\Theta_0$. In this section, we show under the smoothness assumption on $A$ that admits at most isolated singularities with unique tangent cones up to orthogonal transformations. See Theorem \ref{singtangent}. Then we give an upper bound on the number of singularities on a ball where the curvature in small in $L^2$. See Theorem \ref{numbersing}. These theorems significantly employs the theory regarding the stable-stationary harmonic maps into $\mathbb{S}^3$ achieved in \cite{LW06} and \cite{naka06}. We begin by the introducing the notion of singularity:
\begin{definition}
The singular set of $\Theta_0$ is defined to be the following set
\begin{equation}
\label{sing}
S(\Theta_0)\eqqcolon \{x\in M: \liminf_{r\to0}r^{-2}\int_{B_r(x)}|\nabla_A \Theta_0|^2\ge \eta/2 \}.
\end{equation}
where $\eta=\eta(\text{SU}(2))$ as in Theorem \ref{eps}. Furthermore, a point in $S(\Theta_0)$ is called a singularity.
\end{definition}
Now let us present the first main result in this section:
\begin{theorem}
\label{singtangent}
Let $P\to M$ be a smooth $\text{SU}(2)$-principal bundle over a compact $4$-manifold and $A$ be a smooth connection on $P$. Let $\Theta_0$ be a Coulomb minimizer (see Definition \ref{coumindef0}), with associated connection form $A_0$. Then, $S(\Theta_0)$ consists of isolated singularities. Further, at each singularity, a tangent cone is given by $U:\mathbb{R}^4\backslash\{0\}\to \mathbb{S}^3$, $x\mapsto T(\frac{x}{|x|})$ for some $T\in O(3)$.
\end{theorem}
\begin{proof}
Choose any point $x_0\in S(\Theta_0)$, and any sequence of $r_i\to 0$. Up to rescaling we could assume $\int_{B_{10}(x_0)}|F_A|^2<\epsilon_0$. Now let us consider the blow up of $\Theta_0$ at the sequence $\{r_i\}_{i}$ denoted by $\Theta_{0,x_0,i}$ with connection forms $\{A_{0,x_0,i}\}$. Fix arbitrary $R>0$. In Corollary \ref{keycoro}, take $\rho=R$, $g_i=r_i^{-2}g_M$ and $A_i(x)=r_iA(p+r_i(x-p))$. Due to the smoothness of $A$, we have $F_{A_i}\to 0$ in $L^2_{\text{loc}}(\mathbb{R}^4)$. Now we follow almost verbatim the paragraph preceding Corollary \ref{keycoro} to choose the Uhlenbeck gauges for each $A_i$ defined on $B_{10r_i^{-1/2}}(x_0,r_i^{-2}g_M)$ (in order to exhaust $\mathbb{R}^4$) with connection form $A_i^*$. In addition, due to classical $C^{1,\alpha}$-compactness result we see that $g_i\to g_{\text{Euc}}$ in $C_{\text{loc}}^{1,\alpha}(\mathbb{R}^4)$. Then we could apply Corollary \ref{keycoro} upon replacing $B_\rho(p)$ by any $B_R(x_0,g_{\text{Euc}})$. Using a diagonal argument we find a mapping $\Theta_{\infty}^*:\mathbb{R}^4\to \text{SU}(2)$ such that for every $R>0$ the following holds:
\begin{equation}
\lim_{i\to \infty}\|\Theta_{i}^*-\Theta^*_{\infty}\|_{W^{1,2}(B_{R}(x_0,g_{\text{Euc}}))}=0.
\end{equation}
Moreover, $\Theta^*_\infty$ is indeed a stable-stationary harmonic map from $\mathbb{R}^4$ into $\text{SU}(2)\equiv\mathbb{S}^3$. Let us now prove the following claim:
\begin{claim}
\label{uni}
$\lim_{r\to 0}r^{-2}\int_{B_{r}(x_0)}|\nabla_{A}\Theta_0|^2$ exists.
\end{claim}
\begin{remark}
Proof of this claim uses smoothness of $A$. 
\end{remark}
\begin{proof}[proof of Claim \ref{uni}]
Firstly, by \eqref{sta1} (where we choose a proper cutoff function $\zeta$) there exists $C$ such that
\begin{equation}
\label{scalebd}
\sup_{r<10}r^{-2}\int_{B_{r}(x_0)}|\nabla_{A}\Theta_0|^2\le C.
\end{equation}
Apply H\"older's inequality to the last term of \eqref{mono} and then use \eqref{scalebd}, the following is trivially achieved 
\begin{equation}
\label{roughbd}
\rho^{-2}\int_{B_{\rho}(x_0)}|\nabla_A \Theta_0|^2-\sigma^{-2}\int_{B_{\sigma}(x_0)}|\nabla_A \Theta_0|^2\ge -C\|F_A\|_{L^{\infty}(B_{\rho}(x_0))}\rho^2.
\end{equation}
for all $\rho>\sigma$. Suppose $R_1$ and $S_1$ are two possible limits achieved by
$$\{r_i^{-2}\int_{B_{r_i}(x_0)}|\nabla_{A}\Theta_0|^2\}_i\ \text{and}\ \{s_i^{-2}\int_{B_{s_i}(x_0)}|\nabla_{A}\Theta_0|^2\}_i$$ 
respectively. Thanks to smoothness of $A$, for any $\epsilon$ there exists $\rho(\epsilon)$ small enough such that
\begin{equation}
\sup_{\rho\le \rho(\epsilon)}\rho^2\|F_A\|_{L^{\infty}(B_{\rho}(x_0))}<\epsilon.
\end{equation}
Inserting this into \eqref{roughbd}, we have for all $\sigma<\rho<\rho(\epsilon)$ that
\begin{equation}
\label{diff}
\rho^{-2}\int_{B_{\rho}(x_0)}|\nabla_A \Theta_0|^2-\sigma^{-2}\int_{B_{\sigma}(x_0)}|\nabla_A \Theta_0|^2\ge -C\epsilon.
\end{equation}
Choose any $\delta$, we shall prove $|R_1-S_1|<3\delta$; this would yield Claim \ref{uni}. Consider sufficiently large $i_0$ in both sequences such that for $i\ge i_0$ the following hold
\begin{equation}
r_{i},s_{i}<\rho(\epsilon)\ \text{and}\ |R_1-r_i^{-2}\int_{B_{r_i}(x_0)}|\nabla_{A}\Theta_0|^2|<\delta,\ |S_1-s_i^{-2}\int_{B_{s_i}(x_0)}|\nabla_{A}\Theta_0|^2|<\delta.
\end{equation}
First choose some $r_i$ (with $i\ge i_0$), then choose $s_j$ with $j>i_0$ and $s_j<r_i$, and thirdly choose $r_{i+K}$ such that $r_{i+K}<s_j$. Denote $\rho^{-2}\int_{B_{\rho}(x_0)}|\nabla_A \Theta_0|^2$ by $\zeta_{x_0}(\rho)$. Now, using \eqref{diff} we have
\begin{eqnarray}
\begin{split}
&-\delta-C\epsilon\rho-\delta\le (S_1-\zeta_{x_0}(s_j))+(\zeta_{x_0}(s_j)-\zeta_{x_0}(r_{i+K})+(\zeta_{x_0}(r_{i+K})-R_1)=S_1-R_1\\
&=(S_1-\zeta_{x_0}(s_i))+(\zeta_{x_0}(s_i)-\zeta_{x_0}(r_{i})+(\zeta_{x_0}(r_{i})-R_1)\le \delta+C\epsilon\rho+\delta.
\end{split}
\end{eqnarray}
By choosing $\epsilon$ sufficiently small at the beginning, the above implies $|S_1-R_1|<3\delta$. Hence Claim \ref{uni} is proved.
\end{proof}
Now we continue to finish the proof by showing that $\Theta_{\infty}^*:\mathbb{R}^4\backslash\{0\}\to \mathbb{S}^3$ such that $x\mapsto T(\frac{x}{|x|})$ for some $T\in O(3)$. Let us first observe that $\Theta_{\infty}^*$ is a cone map; more precisely,
\begin{claim}
\label{rigid}
$R^{-2}\int_{B_{R}(x_0)}|d\Theta^*_{\infty}|^2$ is a positive constant in $R$.
\end{claim}
\begin{proof}
By Claim \ref{uni}, we could set $Q=\lim_{r\to 0}r^{-2}\int_{B_{r}(x_0)}|\nabla_{A}\Theta_0|^2$. Then for any $\epsilon$ there is $\rho(\epsilon)$ such that $|Q-r^{-2}\int_{B_{r}(x_0)}|\nabla_{A}\Theta_0|^2|<\epsilon$ for all $r\le \rho(\epsilon)$. For any $R_1$ and $R_2$, we have
\begin{eqnarray}
\begin{split}
&|R_1^{-2}\int_{B_{R_1}(x_0)}|d\Theta^*_{\infty}|^2-R_2^{-2}\int_{B_{R_2}(x_0)}|d\Theta^*_{\infty}|^2|\\
\le& |R_1^{-2}\int_{B_{R_1}(x_0)}|d\Theta^*_{\infty}|^2-R_1^{-2}\int_{B_{R_1}(x_0,g_i)}|\nabla_{A_{i}}\Theta_{i}^*|^2|\\
+& |(r_iR_1)^{-2}\int_{B_{r_iR_1}(x_0)}|\nabla_{A}\Theta_0|^2-(r_iR_2)^{-2}\int_{B_{r_iR_2}(x_0)}|\nabla_{A}\Theta_0|^2|\\
+&|R_2^{-2}\int_{B_{R_2}(x_0)}|d\Theta^*_{\infty}|^2-R_2^{-2}\int_{B_{R_2}(x_0,g_i)}|\nabla_{A_{i}}\Theta_{i}^*|^2|\\
\eqqcolon& K_1+K_2+K_3.
\end{split}
\end{eqnarray}
By the strong convergence in Corollary \ref{keycoro}, $K_1$ and $K_3$ can be made less than $\epsilon$ if $i$ is chosen large; in addition, $K_2$ can be made less than $2\epsilon$, if $i$ is further increased such that $r_iR_1,r_iR_2<\rho(\epsilon)$. Since $\epsilon$ is arbitrary, we conclude that $R^{-2}\int_{B_{R}(x_0)}|d\Theta^*_{\infty}|^2$ is constant and equals to $Q$. Moreover, from this and the fact that $x_0\in S(\Theta_0)$, we see that $Q>0$.
\end{proof}
The rigidity of stable-stationary cone map from $\mathbb{R}^4\to \mathbb{S}^3$ (see \cite{LW06} and \cite{naka06}) then says that the only such map is given by $T(\frac{x}{|x|})$ for some $T\in O(3)$. Combining this with Claim \ref{rigid} we have proved the tangent cone part of Theorem \ref{singtangent}.
\bigskip

The second statement to be proved is that the singularities are isolated. However, this is an easy consequence of Theorem \ref{eps}. Indeed, on the one hand, by the strong $L^2$-convergence of $A_{0,x_0,i}$ for every ball $B_{r}(p)\subseteq \mathbb{R}^4$ we have
\begin{eqnarray}
\begin{split}
\label{2con}
\int_{B_{r}(p)}|\nabla_{A_{i}}\Theta^*_{i}|^2\to \int_{B_{r}(p)}|d\Theta^*_{\infty}|^2.
\end{split}
\end{eqnarray}
On the other hand, since $x_0$ is only singularity of $\Theta_{\infty}$, for every $x\in B_{1}(x_0)$ different from $x_0$ we could find a ball $B_{r_x}(x)$ such that 
\begin{eqnarray}
\begin{split}
r_x^{-2}\int_{B_{r_x}(x)}|d\Theta^*_{\infty}|^2<\eta/3.
\end{split}
\end{eqnarray}
By \eqref{2con}, for $i$ large enough, the following holds
\begin{eqnarray}
\begin{split}
r_x^{-2}\int_{B_{r_x}(x,g_i)}|\nabla_{A_i}\Theta_{i}^*|^2<\eta/2.
\end{split}
\end{eqnarray}
After rescaling, this becomes
\begin{eqnarray}
\begin{split}
(r_ir_x)^{-2}\int_{B_{r_ir_x}(x,g_M)}|\nabla_{A}\Theta_{0}|^2<\eta/2.
\end{split}
\end{eqnarray}
By Definition \eqref{sing}, $x\notin S(\Theta_0)$. Thus $S(\Theta_0)$ consists of isolated singularities only.
\end{proof}
In the next theorem we manage to control the number of isolated singularities on ball with small curvature:
\begin{theorem}
\label{numbersing}
Let $P\to M$ be a smooth $\text{SU}(2)$-principal bundle over a compact $4$-manifold $M$ and $A$ be a smooth connection on $P$. Let $\Theta_0$ be a Coulomb minimizer (see Definition \ref{coumindef0}). Then there exists $\epsilon_0$, $r_M$ and $N_0\ge 10$, if $10\rho_0<r_M$ and $\int_{B_{2\rho_0}(y_0)}|F_A|^2<\epsilon_0$, then $\#\{S(\Theta_0)\cap B_{3\rho_0/2}(y_0)\}\le N_0$.
\end{theorem}
\begin{proof}
For all small $\delta$ define $r_{M,\delta}(x)=\sup\{r:r^2\sup_{y\in B_r(x)}\big{(}|\text{Sec}|+\text{inj}^{-2}\big{)}<\delta\}$. From Theorem \ref{singtangent}, there are finitely many singularities inside $B_{19\rho_0/10}(y_0)$ when both $\epsilon_0$ and $r_M$ are chosen small enough. 
\begin{claim}
\label{key}
There exists a small universal constant $\delta_*<10^{-6}$, upon further decreasing $\epsilon_0$ and $r_M$, if $10\rho_0<r_M$ and $\int_{B_{2\rho_0}(y_0)}|F_A|^2<\epsilon_0$, then for any two singularities $x_*, y_*\in B_{3\rho_0/2}(y_0)$ we have $\text{dist}(x_*,y_*)>\delta_*\rho_0$.
\end{claim}
\begin{proof}
Let $\delta_*<10^{-6}$ be some small number to be specified. The proof of Claim \ref{key} goes by contradiction. Were the claim not true for $\delta_*$, there would exist sequences $\epsilon_i\to 0$, $\delta_i\to 0$, $\{w_i\}_i\subseteq M$, $\{A_i\}_i$, the associated Coulomb minimizers $\{\Theta_{0,i}\}_i$, and a sequence of pairs of singularities of $\{\Theta_{0,i}\}_i$ denoted by $\{x_{*,i},y_{*,i}\}$, such that $\int_{B_{2\rho_i}(w_i)}|F_{A_i}|^2<\epsilon_i$, $10\rho_i<r_{M,\delta_i}(w_i)$, $x_{*,i},y_{*,i}\in B_{\rho_i}(w_i)$ while $\text{dist}(x_{*,i},y_{*,i})\le \delta_*\rho_i$. Due to the finiteness of singularities inside $B_{19\rho_i/10}(w_i)$, we might find for each $i$ a pair of singularities $\{x_i,y_i\}\subseteq B_{17\rho_i/10}(w_i)$ such that:
\begin{equation}
\label{noncollision}
\text{dist}(x_i,y_i)<100\inf_{x,y\in S(\Theta_{0,i})\cap B_{\sqrt{\text{dist}(x_i,y_i)}}(x_i)}\text{dist}(x,y).
\end{equation}
To find such a pair we applied a discrete version of the same method being used in the proof of Theorem \ref{apr}. Hence we choose to omit here, and present its proof in details in Chapter \ref{gle}. However, it needs to be pointed out that $\delta_*<10^{-6}$ is necessary for this fact. Next, let us rescale each $B_{2\rho_i}(w_i)$ such that $\text{dist}(\hat{x}_{i},\hat{y}_{i})=100\delta_*$. For convenience, we could identify all $x_i$ as $\hat{x}$; also, upon passing to a subsequence we might assume that $y_i\to \hat{y}$. Clearly,
\begin{equation}
\label{collision}
\text{dist}(\hat{x},\hat{y})=100\delta_*(<1/2).
\end{equation}
Under this identification, $\{\Theta_{0,i}\}_i$ become $\{\hat{\Theta}_{0,i}\}_i$ and $\{A_i\}_i$ become $\{\hat{A}_i\}_i$. Let $\sigma_i$ be an Uhlenbeck gauge of $\hat{A}_i$ defined on $B_{2}(\hat{x},g_i)$ with connection form $\hat{A}_i^*$ that satisfies \eqref{uhgauge}. Then denote by $\hat{\Theta}_{0,i}^*$ the trivialization of $\{\hat{\Theta}_{0,i}\}_i$ under $\sigma_i$. Clearly, $g_i$ converges to $g_{\text{Euc}}$ in $C^{1,\alpha}$ and $\hat{A}_i\to A_{\text{flat}}$ in $W^{1,2}(B_{3/2}(\hat{x},g_{\text{Euc}}))$. Moreover, by Corollary \ref{keycoro}, we have $\hat{\Theta}_{0,i}^*$ converges strongly in $W^{1,2}(B_1(\hat{x},g_{\text{Euc}}))$ to a stable-stationary harmonic map $\Theta_{\infty}: B_1(\hat{x},g_{\text{Euc}})\to\mathbb{S}^3$. Now according to the theory of the stable-stationary harmonic maps from $B^{(4)}_1(\hat{x},g_{\text{Euc}})$ into $\mathbb{S}^3$ (see, for instance, Theorem 3.4.12 of \cite{LW08}), there exists universal $\delta_0$, such that if $\Theta_{\infty}$ is a stable-stationary harmonic map on $B_{1}(\hat{x},g_{\text{Euc}})$, then the distance between two singularities in $B_{1/2}(\hat{x})$ are bounded below by $\delta_0$. On the other hand, due to \eqref{noncollision} and the definition of $S(\Theta_{\infty})$, both $\hat{x}$ and $\hat{y}$ are singularities of $\Theta_{\infty}$ inside $B_{1/2}(\hat{x})$. Hence $d(\hat{x},\hat{y})\ge \delta_0$. However, by choosing $\delta_*<\delta_0/200$ from the beginning of the proof, then \eqref{collision} gives a contradiction. This suffices to conclude Claim \ref{key} by simply taking $r_M=\inf\{r_{M,\delta_0}(x):x\in M\}$.
\end{proof}
Evidently, Claim \ref{key} will be enough to conclude Theorem \ref{numbersing}. Indeed, by Claim \ref{key} there exist $\epsilon_0$ and $\delta_0$, such that under the condition $10\rho_0<r_{M,\delta_0}(y_0)$ together with $\int_{B_{2\rho_0}(y_0)}|F_A|^2<\epsilon_0$, one has $\#\{S(\Theta_0)\cap B_{3\rho_0/2}(y_0)\}\le N_0\equiv 1000^4\delta_0^{-4}$. This is the end of the proof of Theorem \ref{numbersing}.
\end{proof}
\bigskip

\section{$\mathcal{L}^{4,\infty}$-estimate under small curvature assumption}
\label{gle}
In this section, we achieve the $\mathcal{L}^{4,\infty}$ estimate of $A_0$ under the curvature smallness hypothesis. More precisely, the main result in this section is the following theorem:
\begin{theorem}
\label{4localeps}
Under the same assumptions as in Theorem \ref{main}, there exists $\epsilon_0$, $r(M)$, and $C_0$ such that for all $B_{4r}(x_0)\subseteq M$ with $10r<r(M)$ and $\int_{B_{4r}(x_0)}|F_A|^2<\epsilon_0$, it holds $\|A_0\|_{\mathcal{L}^{4,\infty}(B_{r}(x_0))}\le C_0$.
\end{theorem} 
We begin by introducing the following definition:
\begin{definition}
\label{locnewnorm}
Let us define $s_{A_0,\epsilon_0}(x)\eqqcolon\sup\{0\le r\le r_{F_A,\epsilon_0}(x):\sup_{y\in B_{r}(x)}\int_{B_r(y)}|A_0|^4\le 1\}.$ 
\end{definition}
Next, we define the notions of ``curvature-scale'' $r_{F_A,\epsilon_0}(x)$ and ``regularity scale" $r_{\Theta_0,\epsilon_0}(x)$ as follows:
\begin{definition}
\label{rads}
\begin{eqnarray}
\begin{split}
&r_{F_A,\epsilon_0}(x)=\sup\{r>0:\sup_{y\in B_{r}(x)}\int_{B_r(y)}|F_A|^2<\epsilon_0\}.\\
&r_{\Theta_0,\epsilon_0}(x)=\sup\{0\le r\le r_{F_A,\epsilon_0}(x):\sup_{y\in B_{r}(x)}r^{-2}\int_{B_r(y)}|A_0|^2\le \theta_1\eta\}.
\end{split}
\end{eqnarray}
where $\eta$ is the same as in \eqref{sing}, and $\theta_1$ is a universal small number such that $s_{A_0,\epsilon_0}\ge 100^{-1}r_{\Theta_0,\epsilon_0}$. 
\end{definition}
\begin{proposition}
\label{posrad}
If $S(\Theta_0)\cap B_r(x)=\emptyset$, then $\inf_{w\in B_{r/2}(x)}r_{\Theta_0,\epsilon_0}(w)>0$. Here $S(\Theta_0)$ is the singular set of $\Theta_0$ defined in \eqref{sing}. 
\end{proposition}
\begin{proof}
By the fact $S(\Theta_0)\cap B_r(x)=\emptyset$ and Theorem \ref{eps}, we have
\begin{equation}
\int_{B_{2r/3}(x)}|A_0|^4<\infty,
\end{equation}
Hence, there exists some $r_0>0$ such that $\int_{B_s(w)}|A_0|^4<\theta_1\eta$ for all $w\in B_{r/2}(x)$. Further, by H\"older's inequality we have
\begin{equation}
s^{-2}\int_{B_s(w)}|A_0|^2\le \int_{B_s(w)}|A_0|^4,\ \forall s\le 1/2.
\end{equation}
From this we immediately conclude that $\inf_{w\in B_{r/2}(x)}r_{\Theta_0,\epsilon_0}(w)\ge r_0/2>0$.
\end{proof}
The next theorem improves the above proposition by giving an effective (positive) lower bound on the regularity scale $r_{\Theta_0,\epsilon_0}(x)$: 
\begin{theorem}
\label{apr}
Let $P\to M$ be a smooth $\text{SU}(2)$-principal bundle over a compact $4$-manifold and $A$ be a smooth connection on $P$. Let $\Theta_0$ be a Coulomb minimizer obtained in Subsection \ref{prep}. Then there exist positive constants $\epsilon_0$ and $c_0\le10^{-4}$. Given $\int_{B_{2\rho_0}(y_0)}|F_A|^2<\epsilon_0$ and $S(\Theta_0)\cap B_{2\rho_0}(y_0)=\emptyset$. Then for each $x\in B_{\rho_0}(y_0)$ it holds $r_{\Theta_0,\epsilon_0}(x)>c_0\rho_0$.
\end{theorem}
\begin{remark}
Roughly speaking, Theorem \ref{apr} allows us to obtain an effective $L^4$-regularity estimate of $A_0$ on a ball by simply knowing its $L^4$ norm is finite on a twice larger ball. The local $\mathcal{L}^{4,\infty}$-estimate will then follow as a consequence of Theorem \ref{numbersing} and Theorem \ref{apr}.
\end{remark}
\begin{proof}
Up to scaling and translation, we assume $\rho_0=1$, $y_0=0$, and $r_M\ge 100$ (i.e. $|\text{Sec}|<10^{-4}$ and $\text{inj}>100$) to hold on all the balls involved in the proof later. Were Theorem \ref{apr} not true, then there would exist sequences $\epsilon_i\to 0$, $10^{-4}>c_i\to 0$, $\{A_i\}_i$, the associated Coulomb minimizers $\{\Theta_{0,i}\}$, and $\{y_i\}\subseteq B_{1}(y_0)$ such that (upon further scaling) $\int_{B_{2}(y_0)}|F_{A_i}|^2<\epsilon_i$, $r_{\Theta_{0,i},\epsilon_i}(y_i)<c_i$; furthermore,
\begin{equation}
\label{finiteness}
S(\Theta_{0,i})\cap B_2(y_0)=\emptyset.
\end{equation}
Next, we proceed as follows to find $x_i\in B(y_0,\frac{3}{2})$ for each $i$ such that
\begin{equation}
\label{ptmin}
0\leftarrow r_{\Theta_{0,i},\epsilon_i}(x_i)<100\inf_{y\in B(x_i,\sqrt{ r_{\Theta_{0,i},\epsilon_i}(x_i)})}r_{\Theta_{0,i},\epsilon_i}(y).
\end{equation}
For brevity, let us denote $r_{\Theta_{0,i},\epsilon_i}(z)$ by $r_i(z)$ from now on. Starting with $y_i$, we are facing two possible alternatives. Either
\begin{equation}
\label{minpt}
r_{i}(y_i)<100\inf_{y\in B(y_i,\sqrt{r_i(y_i)})}r_{i}(y);
\end{equation}
or, there exists $y_i^{(1)}\in B(y_i,\sqrt{r_i(y_i)})$ such that
\begin{equation}
\label{nminpt}
r_{i}(y_i)\ge100 r_i(y_i^{(1)}).
\end{equation} 
In the first case scenario, we immediately obtain \eqref{ptmin} simply by setting $y_i\coloneqq x_i$. In the second case, we need to start over with $y_i^{(1)}$ and repeat the dichotomy strategy. The process terminates in finite steps. Indeed, suppose the process has been iterated for $K$ times, for some large $K$ to be determined. Due to the relations that $r_i(y_i^{(\alpha+1)})<1/100r_{i}(y_i^{\alpha})$ and $r_i(y_i^{(0)})<c_i<10^{-4}$, we have
$$\text{dist}(w_i^{K},0)\le \sum_{\alpha=0}^K\sqrt{r_i(y_i^{(\alpha)})}+1<\frac{3}{2}.$$
In other words, $y_i^{K}$ lies in $B_\frac{3}{2}(0)$. Here we use the convenient notation that $y_i^{(0)}=y_i$. Meanwhile, $r_i(w_i^{(K)})$ decays to $0$ as $K$ goes to infinity. On the other hand, let us apply Proposition \ref{posrad} to $\Theta_i$ and $B_2(0)$ to see that $r_i(\cdot)$ has a positive lower bound in $B_\frac{3}{2}(0)$. Hence for sufficiently large $K$ we arrive at a contradiction. Therefore, one arrives at a desired point $x_i$ which satisfies \eqref{ptmin} by repeating above process finite steps. Again, for the sake of brevity let us denote $r_i(x_i)$ by $r_i$. 
\bigskip

Next, for convenience let us rescale and then translate $B_{r_i}(x_i)$ to $B_{1}(0)$ for all $i$, where $0$ denotes the orgin of $\mathbb{R}^4$. Under this change, $\{\Theta_{0,i}\}_i$ become $\{\hat{\Theta}_{0,i}\}_i$ and $\{A_i\}_i$ become $\{\hat{A}_i\}_i$. Clearly, $\hat{A}_i\to A_{\text{flat}}$ strongly in $W^{1,2}_{\text{loc}}(\mathbb{R}^4)$ and the underlying metric $r_i^{-2}g_i$ converges to $g_{\text{Euc}}$ in $C^{1,\alpha}$. Note the latter convergence is a consequence of our assumption at the very beginning that $r_{(M,g_i)}\ge 100$. Let $\sigma_i$ be an Uhlenbeck-gauge of $\hat{A}_i$ on $B_{r_i^{-1}}(x_i)$ with connection form $\hat{A}_i^*$ that satisfies \eqref{uhgauge}. Under the Uhlenbeck gauge $\sigma_i$, the frame $\hat{\Theta}_{0,i}$ admits a trivialization denoted by $\hat{\Theta}^*_{0,i}$. By Corollary \ref{keycoro}, $\hat{\Theta}^*_{0,i}$ converges strongly in $W^{1,2}_{\text{loc}}(\mathbb{R}^4)$ to a stable-stationary harmonic map $\Theta_{\infty}$ from $\mathbb{R}^4\to \text{SU}(2)$. Due to the strong $W^{1,2}_{\text{loc}}(\mathbb{R}^4)$ convergence, we have $r_{\Theta_{\infty}}\ge 1$ on $\mathbb{R}^4$. Upon further decreasing $\eta$ if necessary, we apply the standard $\epsilon$-regularity theorem of stationary harmonic map to conclude that $\Theta_{\infty}$ is smooth.  
\medskip

Next, we show that $\Theta_{\infty}$ is nontrivial (i.e. a non-constant map). Indeed, according to the definition of $r_{\Theta_{0,i}}$, the following holds:
\begin{eqnarray}
\begin{split}
\theta_1\eta=r_{\Theta_{0,i}}(x_i)^{-2}\int_{B_{2r_{\Theta_{0,i}}(x_i)}(x_i)}|\nabla_{A_i}\Theta_{0,i}|^2=\int_{B_2(0)}|\nabla_{\hat{A}_i}\hat{\Theta}^*_{0,i}|^2.
\end{split}
\end{eqnarray}
Again by the strong $W^{1,2}_{\text{loc}}(\mathbb{R}^4)$ convergence, we have
\begin{eqnarray}
\begin{split}
\int_{B_2(0)}|d \Theta_{\infty}|^2\ge \theta_1\eta.
\end{split}
\end{eqnarray}
Therefore $\Theta_{\infty}$ is nontrivial. Further, we need the following claim:
\begin{claim}\label{blowdown}
\begin{equation}
\label{blowdown}
\sup_{R>0}R^{-2}\int_{B_{R}(0)}|d\Theta_{\infty}|^2<\infty.
\end{equation}
\end{claim} 
\begin{proof}
Choose any $R$. Then for sufficiently large $i$ we have
\begin{eqnarray}
\begin{split}
\label{blowdown0}
&R^{-2}\int_{B_{R}(0)}|d\Theta_{\infty}|^2=\lim_{i\to \infty}R^{-2}\int_{B_{R}(0,g_i)}|\nabla_{\hat{A}_i}\hat{\Theta}_{0,i}|^2=(r_iR)^{-2}\int_{B_{r_iR}(x_i)}|\nabla_{A_i}\Theta_{0,i}|^2\\
\le&C(r_iR)^{-2}\int_{B_{2r_iR}(x_i)}|A^*_i|^2+C\le C_1
\end{split}
\end{eqnarray}
for some finite constant $C_1$ independent of $R$. Note in the first inequality we have used Theorem \ref{stathm}, while in the second inequality we used H\"older inequality and \eqref{uhgauge}.
\end{proof}
Consider now the blowdown of $\Theta_{\infty}$ at $0$, that is, $\Theta_\infty(R_i^{-1}x)$ for a sequence of $R_i\to \infty$. By \eqref{blowdown} and the theory of stable-stationary harmonic map from $\mathbb{R}^4$ into $\mathbb{S}^3$ (see Section 3.4 of \cite{LW08}), we know that the blowdown sequence converges strongly in $W_{\text{loc}}^{1,2}(\mathbb{R}^4)$ to some cone-stable-stationary map $\Theta_*$. Moreover, due to the non-triviality of $\Theta_{\infty}$, we have $\Theta_*$ is non-trivial. By the rigidity theorem (see Theorem 3.4.10 of \cite{LW08}), $\Theta_*$ is a degree $\pm1$ map. We conclude that the blowdown sequence converge smoothly outside $0$. Indeed, by the strong $W_{\text{loc}}^{1,2}(\mathbb{R}^4)$ convergence of $\Theta_\infty(R_i^{-1}x)$ as well as the smoothness of $\Theta_*$ outside $0$, for each $\epsilon$ and $x\neq 0$ there exists a ball $B_{r_x}(x)$ such that $r_x^{-2}\int_{B_{r_x}(x)}|d\Theta_\infty(R_i^{-1}x)|^2<\epsilon$ for all large $i$. Now by the $\epsilon$-regularity theorem of the stationary harmonic maps, we conclude that the convergence $\Theta_\infty(R_i^{-1}x)\to \Theta_*$ is indeed smooth in $B_{r_x/2}(x)$. Due to the arbitrariness of $x$ we see that the blowdown sequence converge smoothly outside $0$. Hence the degree of $\Theta_\infty(R_i^{-1}x)$ on the unit sphere must be preserved in the limit as $0$. This contradicts the fact that $\Theta_*$ is of degree $\pm1$. Therefore, the conclusion of Theorem \ref{apr} holds for some $0<c_0\le 10^{-4}$. This finishes the proof of Theorem \ref{apr}.
\end{proof}
\begin{proof}[Proof of Theorem \ref{4localeps}]
By the Sobolev embedding $W^{1,2}\hookrightarrow L^4$ in dimension $4$, Definition \ref{4rad}, and Definition \ref{locnewnorm}, we have: 
\begin{eqnarray}
\begin{split}
\label{soborad}
s_{A_0}(x)\ge s_{A_0,\epsilon_0}(x)\ge c_*r_{\Theta_0,\epsilon_0}(x) 
\end{split}
\end{eqnarray}
for some $c_*>0$ depending on the Sobolev constant. In view of \eqref{soborad}, Theorem \ref{4localeps} is an immediate consequence of Theorem \ref{numbersing} and Theorem \ref{apr}. This completes the proof of Theorem \ref{4localeps}.
\end{proof}
\bigskip

\section{Annular-bubble decomposition}
\label{annbub}
The last key ingredient we need for the global $\mathcal{L}^{4,\infty}$-estimate is the annular-bubble decomposition for a smooth connection. This idea of using the annular-bubble region also appeared in \cite{JN16}, \cite{NV16}, and \cite{CJN18}. Firstly let us introduce the following notations:
\begin{eqnarray}
\begin{split}
\zeta_x(r)\eqqcolon\int_{B_{r}(x)}|F_A|^2,\ \overline{\zeta}_x(r)\eqqcolon\sup_{z\in B_{r}(x)}\int_{B_r(z)}|F_A|^2.
\end{split}
\end{eqnarray}
Let $\gamma$ be a small universal constant to be determined later, and $K_0$ be any large integer for now, which will also be fixed by the end of the next section. Now we introduce the following notions of weakly flat regions, annular regions, and bubble regions. We note the following notion is scaling invariant, and hence could be efined
\begin{definition}
\label{weaklyflat}
We say $\mathcal{W}$ is a weakly flat region of scale $r$ if $\mathcal{W}=B_{2r}(p)$, and it satisfies $|\overline{\zeta}_{2r}(p)-\overline{\zeta}_{\gamma^{K_0} r}(p)|<\gamma^{K_0}$. 
\end{definition}
\begin{definition}
We say $\mathcal{A}$ is an annular region $\mathcal{A}$ of scale $r$ if $\mathcal{A}=A_{s,r}(p)$ for some $s\le 5\gamma^{K_0} r$, such that the following hold:\\
(1) $|\overline{\zeta}_{2r}(p)-\overline{\zeta}_{s}(p)|<\gamma^{K_0}$;\\
(2) $\text{while }|\overline{\zeta}_{2r}(x)-\overline{\zeta}_{\gamma^{-1}s}(x)|\ge\gamma^{K_0}\ \text{for all }x\in B_{\gamma^{K_0} r}(p)$.
\end{definition}
\begin{definition}
\label{bublow}
We say $\mathcal{B}$ is a bubble region $\mathcal{B}$ of scale $r$ if $\mathcal{B}=B_r(p)\backslash \bigcup_{i=1}^N B_{r_i}(x_i)$, such that for some constants $N(\Lambda), c(\Lambda,K_0)$ (both will be specified later) the following hold:\\
(1) $N\le N(\Lambda)$;\\
(2) $r_i\ge c(\Lambda, K_0)r$;\\
(3) $\overline\zeta_{x_i}(r_i)<\overline\zeta_p(r)-\gamma^{K_0}$;\\
(4) For all $x\in \mathcal{B}$, there is $r_x\ge c(\Lambda,K_0) r$ such that $\int_{B_{r_x}(x)}|F_A|^2<\epsilon_0/2$. Here $\epsilon_0$ is the small constant determined in Theorem \ref{eps}, Theorem \ref{numbersing}, and Theorem \ref{apr}.
\end{definition}
Given a weakly flat region, the following proposition allows us to decompose it into controllable many annular and bubble regions:
\begin{proposition}
\label{k0unfixed}
Fix $K_0$. There exists $N(\Lambda,K_0)$. Given a weakly flat region $B_2(p)$, there exists a collection of annular and bubble regions $\{\mathcal{A}_i\}_{i=1}^{N_a}\cup \{\mathcal{B}_j\}_{j=1}^{N_b}$  such that $B_1(p)\subseteq \bigcup_{i=1}^{N_a}\mathcal{A}_i\cup\bigcup_{j=1}^{N_b}\mathcal{B}_j$ and $N_a+N_b\le N(\Lambda,K_0)$.
\end{proposition}
\begin{proof}
Let us begin by the following claim:
\begin{claim}
\label{4ab}
For each positive integer $l\le \Lambda\gamma^{-K_0}+1$, there exists a constant $C_l(\Lambda,K_0)$ and a cover of $B_1(p)$ given by
\begin{eqnarray}
\begin{split}
\{\mathcal{A}_{l,i}\}_i^{N_{l,a}}\cup \{\mathcal{B}_{l,j}\}_j^{N_{l,b}}\cup\{\mathcal{W}_{l,k}\}_k^{N_{l,w}};
\end{split}
\end{eqnarray}
where $\mathcal{A}_{l,i},\ \mathcal{B}_{l,j}$, and $\mathcal{W}_{l,k}\equiv B_{s_{l,k}}(z_{l,k})$ are annular, bubble, and weakly flat regions respectively for all $i,l$. Moreover, the following holds:
\begin{eqnarray}
\begin{split}
\overline{\zeta}_{z_{l,k}}(s_{l,k})<\Lambda-l\cdot\gamma^{K_0},\ N_{l,a}+N_{l,b}+N_{l,w}\le C_l.
\end{split}
\end{eqnarray}
\end{claim}
\begin{proof}
We prove the claim by induction. To begin with, consider the beginning stage $l=1$. Firstly, for each $x\in B_{1}(p)$, let us define
\begin{eqnarray}
\begin{split}
K_x=\sup\{K:|\overline{\zeta}_{1}(x)-\overline{\zeta}_{\gamma^{K}}(x)|<\gamma^{K_0}\}.
\end{split}
\end{eqnarray}
It is obvious from Definition \ref{weaklyflat} that there exists $x_0\in B_{10\gamma^{K_0}}(p)$ such that
\begin{eqnarray}
\begin{split}
x_0=\text{argmax}_{x\in B_{1}(p)}K_x.
\end{split}
\end{eqnarray}
Clearly, $\mathcal{A}\eqqcolon A_{5\gamma^{K_{x_0}},1}(x_0)$ is an annular region attached to the weakly flat region $B_2(p)$. Next, we need to construct a bubble region on $B_2(p)$. Consider $B_{10\gamma^{K_{x_0}}}(x_0)\subseteq B_2(p)$. For convenience, let us rescale $B_{10\gamma^{K_{x_0}}}(x_0)$ to $B_1(x_0)$. 
By the pigeonhole principle, for each $y\in B_{1/2}(x_0)$, there exists integer $j_y\in [1,\Lambda\gamma^{-K_0-1}]$ such that
\begin{eqnarray}
\begin{split}
|\overline{\zeta}_{(2\gamma^{K_0})^{11+j}}(y)-\overline{\zeta}_{(2\gamma^{K_0})^{10+j}}(y)|<\gamma^{K_0}.
\end{split}
\end{eqnarray}
Define
\begin{eqnarray}
\begin{split}
c(\Lambda,K_0)=(2\gamma^{K_0})^{10+\Lambda\gamma^{-K_0-1}}.
\end{split}
\end{eqnarray}
Construct a Vitali cover of $B_{1}(x_0)$ given by $\{B_{r_s/2}(y_s)\}_s$ where $r_s=(2\gamma^{K_0})^{(10+j_{y_s})}$ and $\{B_{r_s/6}(y_s)\}_s$ are disjoint. Clearly, $r_s\ge c(\Lambda, K_0)$. Apparently there exists a sub-collection of balls $\{B_{r_{s_i}}(y_{s_i})\}_{i=1}^{N}\subseteq \{B_{r_{s}}(y_s)\}_s$ with $N\le10\Lambda/\epsilon_0\eqqcolon N(\Lambda)$ such that
\begin{eqnarray}
\begin{split}
\label{dropk0}
&\int_{B_{r_t}(y_{t})}|F_A|^2<\epsilon_0/2,\ \forall B_{r_t}(y_t)\in  \{B_{r_{s}}(y_s)\}_s\backslash \{B_{r_{s_i}}(y_{s_i})\}_{i=1}^{N},\\
&\overline{\zeta}_{y_{s_i}}(r_{s_i})<\overline{\zeta}_{1}(p)-\gamma^{K_0},\ \forall B_{r_{s_i}}(y_{s_i}).
\end{split}
\end{eqnarray}
The second inequality is a consequence of the facts that $x_0=\text{argmax}_{x\in B_{1}(p)}K_x$, $r_s<\gamma^{K_0}$, and further decreasing $\gamma$ if necessary. Now scaling $B_{1}(x_0)$ back to $B_{10\gamma^{K_{x_0}}}(x_0)$, we have that 
$$\mathcal{B}\eqqcolon B_{10\gamma^{K_{x_0}}}(x_0)\backslash \bigcup_{i}B_{5r_{s_i}\gamma^{K_{x_0}}}(y_{s_i})$$ 
is a bubble region.  
Let us observe that $B_{20r_{s_i}\gamma^{K_{x_0}}}(y_{s_i})$ is once again a weakly flat region (see Definition \ref{weaklyflat}) for each $i$. Thus upon setting $\mathcal{A}_{1,1}\eqqcolon\mathcal{A}$, $\mathcal{B}_{1,1}\eqqcolon\mathcal{B}$, $\mathcal{W}_{1,i}\eqqcolon B_{20r_{s_i}\gamma^{K_{x_0}}}(y_{s_i})$, $C_1=N(\Lambda)+3$, we finished the first stage. Now assume for the purpose of induction that the claim is true for some $l_0\ge 1$. Then by repeating the previous argument to each $B_{s_{l_0,k}}(z_{l_0,k})=\mathcal{W}_{l_0,k}$ we obtain
\begin{eqnarray}
\begin{split}
B_{s_{l_0,k}/2}(z_{l_0,k})\subseteq \mathcal{A}_{l_0,k}\cup \mathcal{B}_{l_0,k}\cup \bigcup_{k^{\prime}}\mathcal{W}_{k,k^{\prime}},
\end{split}
\end{eqnarray}
where $\mathcal{W}_{k,k^{\prime}}=B_{s_{k,k^{\prime}}}(z_{k,k^{\prime}})$; hence by \eqref{dropk0} and the inductive assumption we have
\begin{eqnarray}
\begin{split}
\overline{\zeta}_{z_{k,k^{\prime}}}(s_{k,k^{\prime}})<\zeta(z_{l_0,k})(s_{l_0,k})-\gamma^{K_0}<\Lambda-(l_0+1)\cdot\gamma^{K_0}.
\end{split}
\end{eqnarray}
By setting
\begin{eqnarray}
\begin{split}
&\{\mathcal{A}_{l_0+1,i}\}_i^{N_{l_0+1,a}}=\{\mathcal{A}_{l_0,i}\}_i\cup\{\mathcal{A}_{l_0,k}\}_k,\ \{\mathcal{B}_{l_0+1,j}\}_j^{N_{l_0+1,b}}=\{\mathcal{B}_{l_0,j}\}_j\cup\{\mathcal{B}_{l_0,k}\}_k,\\
&\{\mathcal{W}_{l_0+1,k}\}_k^{N_{l_0+1,w}}=\bigcup_{k}\{B_{s_{k,k^{\prime}}}(z_{k,k^{\prime}})\}_{k^{\prime}},\ C_{l_0+1}=10 N(\Lambda)C_{l_0},
\end{split}
\end{eqnarray}
we complete the proof of Claim \ref{4ab}.
\end{proof}
Taking $l=\Lambda\gamma^{-K_0}+1$ in Claim \ref{4ab}, we immediately prove Proposition \ref{k0unfixed}.
\end{proof}
Using Proposition \ref{k0unfixed}, we shall prove an annular-bubble decomposition result stated as follows:
\begin{theorem}
\label{abde}
Fix $K_0$. Given $\int_{B_{10r}(p)}|F_A|^2<\Lambda$ with $10r\le r(M)$. There exists a large number $P(\Lambda,K_0)$ and a collection of annular and bubble regions $\{\mathcal{A}_{i}\}_{i=1}^{P_a}\cup\{\mathcal{B}_j\}_{j=1}^{P_b}$ that form a cover of $B_r(p)$, such that $P_a+P_b\le P(\Lambda,K_0)$.
\end{theorem}
\begin{remark}
\label{gammaannbub}
We shall refer to $B_1(p)\subseteq \bigcup_i \mathcal{A}_i\cup\bigcup_j \mathcal{B}_j$ as a $\gamma^{K_0}$-annular-bubble decomposition of $B_1(p)$. The integer $K_0$ will be determined in the next Section (i.e. Section \ref{glsec}).
\end{remark}
\begin{corollary}
\label{decompm}
Fix $K_0$. There exists a large number $P(\Lambda,K_0,M)$ and a collection of annular and bubble regions $\{\mathcal{A}_{i}\}_{i=1}^{P_a}\cup\{\mathcal{B}_j\}_{j=1}^{P_b}$ that form a cover of $M$, such that $P_a+P_b\le P(\Lambda,K_0,M)$.
\end{corollary}
To prove Theorem \ref{abde}, we need the following claim.
\begin{claim}
\label{trivial}
Fix $K_0$. Given $\int_{B_{10}(p)}|F_A|^2<\Lambda$. One of the following holds:\\
(1) There exists $x\in B_2(p)$ such that $|\overline{\zeta}_{\gamma^{K_0}}(x)-\overline{\zeta}_{4}(x)|<\gamma^{K_0}$.\\
(2) There exists a Vitali cover of $B_{2}(p)$ given by $\{B_{\gamma^{K_0}}(x_i)\}_{i=1}^{N_1}$ such that $\overline{\zeta}_{\gamma^{K_0}}(x_i)<\overline{\zeta}_{4}(x_i)-\gamma^{K_0}$, where $N_1\le N_1(K_0)$.
\end{claim}
\begin{proof}
This is straightforward, hence we omit the details.
\end{proof}
The key to the proof of Theorem \ref{abde} lies in the following inductive lemma.
\begin{lemma}
\label{trivial1}
Fix $K_0$. For each $l\le 10\Lambda\gamma^{-K_0}+1$, there exists $M_l(\Lambda,K_0)$, and a cover of $B_{r}(p)$ by $\{\mathcal{A}_{l,i}\}_{i=1}^{P_{l,a}}\cup\{\mathcal{B}_{l,j}\}_{j}^{P_{l,b}}\cup \{B_{r_{l,k}}(z_{l,k})\}_k^{P_{l,d}}$ such that\\
(1) $P_{l,a}+P_{l,b}+P_{l,d}\le P_l(\Lambda,K_0)$.\\
(2) $\overline{\zeta}_{z_{l,k}}(r_{l,k})<\Lambda-l\cdot10^{-1} \gamma^{K_0}$.
\end{lemma}
\begin{proof}
Upon rescaling we assume that $r=1$. Let us prove the claim by induction. For $l=1$, either (1) or (2) of Claim \ref{trivial} occurs. For convenience, we refer to them as case (1) and case (2) respectively. In case (1), let us apply Proposition \ref{k0unfixed} to $B_{8}(x)$ to obtain a cover of $B_4(x)(\supseteq B_2(p)\supseteq B_1(p))$ given by a collection of annular and bubble regions. By taking $P_1(\Lambda,K_0)=N(\Lambda,K_0)$ as in Proposition \ref{k0unfixed}, we finish the discussion of this case. In case (2) we simply take $\{B_{r_{1,k}}(z_{1,k})\}_k^{P_{1,d}}$ to be $\{B_{\gamma^{K_0}}(x_i)\}_i$, and $P_1(\Lambda,K_0)$ to be $N_1(K_0)$ as in Claim \ref{trivial}. This finishes the proof of the beginning stage $l=1$. Now assume for the purpose of induction that for some $l_0\ge 1$ the conclusion of the claim holds. Then, upon replacing $B_1(p)$ by $B_{r_{l_0,k}}(z_{l_0,k})$ we repeat the same process as in stage $l=1$. Let $I$ denote the index set of $k$ such that (1) of Claim \ref{trivial} is satisfied upon replacing $B_2(p)$ by $B_{2r_{l_0,k}}(z_{l_0,k})$, and similarly denote by $II$ the index set where (2) of Claim \ref{trivial} is satisfied. By Claim \ref{trivial}, $I\cup II=\{1,\cdots,M_{l_0,d}\}$. For each $k\in I$, we have $B_{2r_{l_0,k}}(z_{l_0,k})\subseteq \bigcup_{i=1}^{N_{l_0,k,a}}\mathcal{A}_{l_0,k,i}\cup\bigcup_{j=1}^{N_{l_0,k,b}}\mathcal{B}_{l_0,k,j}$, where $N_{l_0,k,a}+N_{l_0,k,b}\le N(\Lambda,K_0)$, for the same $N(\Lambda,K_0)$ as in Proposition \ref{k0unfixed}. For each $k\in II$, $B_{2r_{l_0,k}}(z_{l_0,k})\subseteq \bigcup_{k^{\prime}=1}^{N_{l_0,k,d}}B_{r_{l_0,k,k^{\prime}}}(z_{l_0,k,k^{\prime}})$, where $N_{l_0,k,d}\le N_1(K_0)$. Therefore we find the following cover of $B_1(p)$:
\begin{eqnarray}
\begin{split}
B_1(p)\subseteq \bigcup_{i}\mathcal{A}_{l_0,i}\cup\bigcup_{j}\mathcal{B}_{l_0,j}\cup\bigcup_{k\in I} \bigg{(}\bigcup_{i=1}^{N_{l_0,k,a}}\mathcal{A}_{l_0,k,i}\cup\bigcup_{j=1}^{N_{l_0,k,b}}\mathcal{B}_{l_0,k,j}\bigg{)}\cup\bigcup_{k\in II}\bigcup_{k^{\prime}=1}^{N_{l_0,k,d}}B_{r_{l_0,k,k^{\prime}}}(z_{l_0,k,k^{\prime}}).
\end{split}
\end{eqnarray}
Note that for some $B_{r_{l_0-1,K_*}}(z_{l_0-1,k_*})$ which gives birth to $B_{r_{l_0,k}}(z_{l_0,k})$, the following holds by our inductive assumption on stage $l=l_0$:
\begin{eqnarray}
\begin{split}
&\overline{\theta}_{z_{l_0,k,k^{\prime}}}(r_{l_0,k,k^{\prime}})<\overline{\theta}_{z_{l_0,k,k^{\prime}}}(4^{-1}\gamma^{-K_0}r_{l_0,k,k^{\prime}})-\gamma^{K_0}<\overline{\theta}_{z_{l_0-1,k_*}}(r_{l_0-1,k_*})-\gamma^{K_0}\\
<& \Lambda-(l_0-1)10^{-1} \gamma^{K_0}-\gamma^{K_0}<\Lambda-(l_0+1)10^{-1} \gamma^{K_0}.
\end{split}
\end{eqnarray}
Finally, set $P_{l_0+1}(\Lambda,K_0)=2P_{l_0}(\Lambda,K_0)(1+N(\Lambda,K_0)+N_1(K_0))$. This will suffice to finish the induction, and hence proves Lemma \ref{trivial1}.
\end{proof}
\begin{proof}[Proof of Theorem \ref{abde}]
Upon taking $L=10\Lambda\gamma^{-K_0}+1$ in Lemma \ref{trivial1} and then setting $P(\Lambda,K_0)=P_{L}(\Lambda,K_0)$, we proved Theorem \ref{abde}.
\end{proof}
\begin{remark}
\label{refine}
As one could check by tracing the above proof, the annular and bubble cover obtained in Theorem \ref{abde} might be shrunk (at our convenience) while still forming a cover. More precisely, upon rescaling, each annular region that appears in the above cover looks like $A_{\gamma^{K_a},1}(x_a)$, while each bubble region looks like $B_{1}(x_b)\backslash \bigcup_{i}B_{r_{i,b}}(x_{i,b})$. Then after scaling back, $\{A_{2\gamma^{K_a},3/4}(x_a)\}_a\cup\{B_{3/4}(x_b)\backslash \bigcup_{i}B_{2r_{i,b}}(x_{i,b})\}_b$ once again forms a cover. This observation will be needed in obtaining the global $\mathcal{L}^{4,\infty}$-estimate in the next section. To distinguish this refined cover from the original cover, let us denote the refined one by $\{\mathcal{A}^{\prime}_a\}_a\cup\{\mathcal{B}_b^{\prime}\}_b$.
\end{remark}
\bigskip

\section{Global $\mathcal{L}^{4,\infty}$-estimate}
\label{glsec}
In this section, we shall complete the proof of Theorem \ref{main} for smooth $A$ via obtaining the $\mathcal{L}^{4,\infty}$-estimates on annular and bubble regions, as stated in following theorem:  
\begin{theorem}
\label{annbubsing}
Under the same setting as Theorem \ref{main}, let $\Theta_0$ be the Coulomb minimizer with associated connection form $A_0$. Let $B_{r_0}(x_0)\subseteq M$ be such that $10r<r(M)$. Fix a $\gamma^{K_0}$-annular-bubble decomposition (see Remark \ref{gammaannbub}) on $B_{r_0}(x_0)$ given by $\{\mathcal{A}_a\}_a\cup\{\mathcal{B}_b\}_b$ (obtained in Theorem \ref{abde}). Then there exists $C(K_0,\Lambda)$ such that $\|A_0\|_{\mathcal{L}^{4,\infty}(\mathcal{B}_b^{\prime})}\le C(K_0,\Lambda)$. Moreover, there exists an universal small constant $\gamma$, constants $K^*\equiv K^*(\Lambda)$ and $C_1(\Lambda)$, such that for all $K_0\ge K^*$ we have $\|A_0\|_{\mathcal{L}^{4,\infty}(\mathcal{A}_a^{\prime})}\le C_1(\Lambda)$, where $\mathcal{A}$ is a $\gamma^{K^*}$-annular region. Here $\mathcal{A}^{\prime}$ and $\mathcal{B}$ are the refined cover as in Remark \ref{refine}.
\end{theorem}
Admitting Theorem \ref{annbubsing} for now, we could complete the proof of Theorem \ref{main} under the smoothness hypothesis of $A$. 
\begin{proof}[Proof of Theorem \ref{main} under the smoothness assumption on $A$]
Consider a cover of $M$ by $N(M)$-many balls $\{B_{r_i}(x_i)\}_i$ such that $20r_i<r(M)$. Choose $K_0$ to be the universal integer $K^*$ as in Theorem \ref{annbubsing}. By applying Theorem \ref{decompm}, we arrive at a $\gamma^{K^*}$-annular-bubble decomposition of $M$ given by $\{\mathcal{A}_a\}_{a=1}^{N_A}\cup \{\mathcal{B}_b\}_{b=1}^{N_B}$, with $N_A+N_B\le N(\Lambda,K_0,M)<\infty$. As mentioned in Remark \ref{refine}, the union of sets $\{\mathcal{A}^{\prime}_a\}_{a=1}^{N_A}\cup\{\mathcal{B}_b^{\prime}\}_{b=1}^{N_B}$ has once again formed a refined cover. Let us now apply Theorem \ref{annbubsing} to each $\mathcal{A}_a^{\prime}$ and $\mathcal{B}_b^{\prime}$ to estimate as follows:
\begin{eqnarray}
\begin{split}
\label{noeps}
\|A_0\|_{\mathcal{L}^{4,\infty}(M)}&\le \sum_a\|A_0\|_{\mathcal{L}^{4,\infty}(\mathcal{A}_a^{\prime})}^4+\sum_b\|A_0\|_{\mathcal{L}^{4,\infty}(\mathcal{B}_b^{\prime})}^4\le N(M,\Lambda,K_0)(C_1(\Lambda)+C(K_0,\Lambda))\\
&\coloneqq C_0(M,\Lambda).
\end{split}
\end{eqnarray}
Thus we complete the proof of Theorem \ref{main} when $A$ is smooth.
\end{proof}
The rest of this section is devoted into proving Theorem \ref{annbubsing}.
\begin{proof}[Proof of Theorem \ref{annbubsing}]
The first claim that
\begin{equation}
\label{bubl4}
\|A_0\|_{\mathcal{L}^{4,\infty}(\mathcal{B}_b^{\prime})}\le C(K,\Lambda)
\end{equation}
is an immediate consequence of Theorem \ref{4localeps} and item (2) of Definition \ref{bublow}. Hence let us focus on proving the second claim; namely, to prove an effective $\mathcal{L}^{4,\infty}$-estimate on a $\gamma^K$-annular region provided $K$ large enough. Firstly, we begin with achieving a weaker result, namely the effective $L^{4-\epsilon}$-estimate. This almost trivially follows from Theorem \ref{4localeps}:
\begin{lemma}
\label{4epsilon0}
There exists $K_0$ large and $C(\epsilon)$ for each $\epsilon$, such that for each rescaled $\gamma^K$-annular region $\mathcal{A}=A_{\gamma^{K_x},1}(x)$, we have
\begin{eqnarray}
\begin{split}
\|A_0\|_{L^{4-\epsilon}(\mathcal{A}_a^{\prime})}<C(\epsilon).
\end{split}
\end{eqnarray}
Here $\mathcal{A}_a^{\prime}$ denotes the refined annular region as in Remark \ref{refine}.
\end{lemma}
\begin{proof}
It is easy to build a cover of $\mathcal{A}^{\prime}$ given by $\mathscr{B}=\{B_{r_i}(x_i)\}_i$, such that $\int_{B_{2r_i}(x_i)}|F_A|^2<\epsilon_0$ (provided $K$ being large), and moreover that:
\begin{eqnarray}
\begin{split}
\label{4epsilon}
\#\{B\in \mathscr{B}:B\cap A_{2^{-l-1},2^{-l}}(x_0)\neq\emptyset\}\le C,\ \text{for all } l\ge 0. 
\end{split}
\end{eqnarray}
From \eqref{4epsilon}, Theorem \ref{4localeps}, Lemma \ref{interpo}, and then H\"older inequality, we see that $\int_{A_{2^{-l-1},2^{-l}}(x_0)}|A_0|^{4-\epsilon}< C2^{-\theta\epsilon l}$ for each $l\ge 0$ and some $\theta>0$. Therefore, $\int_{\mathcal{A}^{\prime}}|A_0|^{4-\epsilon}<C^{\prime}(\epsilon)$. Thus finishes Lemma \ref{4epsilon0}.
\end{proof}
\begin{corollary}
\label{4epsilon1}
For any $\epsilon$, there exists $C(\epsilon,\Lambda)$, such that $\|A_0\|_{L^{4-\epsilon}(M)}<C(\epsilon,\Lambda)$.
\end{corollary}
\begin{proof}
This follows immediately from \eqref{bubl4}, Lemma \ref{4epsilon0}, and the same argument used in obtaining \eqref{noeps}. We omit the details.
\end{proof}
As an immediate consequence we have the following scaling invariant estimate:
\begin{corollary}
\label{scalinginvariant}
There exists $C(\Lambda)$ such that $\sup_{x\in M,10r<r(M)}r^{-2}\int_{B_{r}(x)}|\nabla_A\Theta_0|^2<C(\Lambda)$.
\end{corollary}
Next, we need the following key lemma which gives a control on the number of singularities of $\Theta_0$ on a annular region. Given a annular region $\mathcal{A}=A_{\gamma^{K_{x_0}},1}(x_0)$, it will be useful to introduce the notation $\text{Ann}_i(x_0)\coloneqq B_{\gamma^{i}}(x_0)\backslash B_{\gamma^{i+1}}(x_0)$. In the presence of no ambiguity we shall omit the notation for the center $x_0$ and simply write $\text{Ann}_i$. Using these notations, we present the following lemma:
\begin{lemma}
\label{singannbdd}
There exists an universal small constant $\gamma$, $K_0(\Lambda)$, and $N_0(\Lambda)$. Let $\mathcal{A}=A_{\gamma^{K_{x_0}},1}(x_0)$ be an arbitrary rescaled $\gamma^{K_0}$-annular region associated to connection $A$, and let $\Theta_0$ be a Coulomb minimizer with associated connection form $A_0$. Then $\#\{i: \text{Ann}_i\cap \text{Sing}(\Theta_0)\neq \emptyset, 0\le i\le K_{x_0}\}\le N_0$.
\end{lemma}
Since the proof of this Lemma is somewhat lengthy, a sketch of the ideas behind the argument will be helpful. Recall that Claim \ref{key} says that the singularities are $c_0 2^{-i}$ sparse on each annulus $\text{Ann}_i(x_0)\subseteq \mathcal{A}$. So our task becomes to bound the number of annuli $\text{Ann}_i$ that contains at least one singularity of $\Theta_0$. The proof goes by contradiction. Suppose this number could not be bounded no matter how large one chooses $K_0$ to be, then one gets a sequence of tuple $(A_l,\Theta_l,\mathcal{A}_l)$, so that $\mathcal{A}_l$ is a $\gamma^{K_l}$-annular region centered at $x_l$ with $K_l\to \infty$, and the number
\begin{eqnarray}
\begin{split}
\label{nlinfty00}
\#\{\text{Ann}_i\subseteq \mathcal{A}_l: \text{Ann}_i\cap S(\Theta_l)\neq \emptyset\}\to \infty.
\end{split}
\end{eqnarray}
Certainly, one should not expect this sequence itself to give a contradiction since all the fuzzy behavior of the singularities might only occur in a region that shrinks towards the centers in $l$ and eventually disappear in the limit, in which case we will not obtain a contradiction. Instead, we find a properly rescaled sequence of $\Theta_i$ that eventually leads to a contradiction. To be a little more precise, we show that upon shifting the center $x_l$ slightly to some $y_l$ nearby one can find a sphere $\partial B_{r_l}(y_l)$ for each $l$ satisfying these properties: (1), $\partial B_{r_l}(y_l)$ is at least $c_0 r_l$ away from $S(\Theta_l)$ (i.e. not too close to a singularity); (2), There is a singularity of $\Theta_l$ lying out side and locating less than $\gamma^{-1} r_l$ away from $\partial B_{r_l}(y_l)$ (i.e. nor too far from a singularity outside); (3), The ``degree'' of $\Theta_l$ restricted to $\partial B_{r_l}(y_l)$ is nonzero (i.e. admits a singularity inside).
\medskip

Notice that the first two conditions are not hard to satisfy due to our assumption \eqref{nlinfty00}. In the third condition, we will introduce the notion of a degree of a frame, which will be defined towards the end of this section. Roughly speaking, the degree is defined to be the trivialization of $\Theta_l$ under a special gauge on the $3$-sphere $\partial B_{r_l}(y_l)$, a gauge that is obtained by patching up the Coulomb gauges defined on two hemisphere of $\partial B_{r_l}(y_l)$; note the curvature $L^2$ smallness on the annular region and Uhlenbeck's Coulomb gauge Theorem (see \cite{U82a}, Theorem 1.3) guarantees the existence of such a gauge. Furthermore, the degree defined above jumps by $\pm1$ as the sphere passes through a singularity of $\Theta_l$, therefore the third condition described above is now easily achieved due to the assumption \eqref{nlinfty00}.
\medskip

Next, we rescale and recenter $\partial B_{r_l}(y_l)$ to $\partial B_{\gamma^2}(y_*)$. Using a somewhat tedious gauge transform argument and, most importantly, the preceding estimates on $\Theta_l$ including the $\epsilon$-regularity theorem \ref{eps}, Theorem \ref{apr} etc., we successfully find a limit $\Theta_\infty$, which heuristically speaking arises as a strong $W^{1,2}$ limit of $\Theta_l$ (this heuristic is obviously non-rigorous because $\Theta_l$ are all sections living on different bundles! It will be made rigorous in the proof given later on) and that $\Theta_\infty$ is equivalent to a stable stationary harmonic map from a $4$ dimensional domain $B_{c\gamma^{-1}}(y_*)$ into $\text{SU}(2)\cong \mathbb{S}^3$. 
\medskip

By the choice of the sphere (before scaling) $\partial B_{r_l}(y_l)$, such sphere is away from singularities, and hence (by Theorem \ref{eps} and \ref{apr}) $\Theta_l$ converges weakly in $W^{2,2}(\partial B_{r_l}(y_l))$ and strongly in $W^{1,3}(\partial B_{r_l}(y_l))$ by the $3$ dimensional compact Sobolev embedding $W^{2,2}\hookrightarrow W^{1,3}$. We then show from this fact that the degree of $\Theta_l$ will be preserved on $\partial B_{\gamma^2}(y_*)$ upon rescaling, and its limit is exactly $\text{deg}(\Theta_\infty)\rvert_{\partial B_{\gamma^2}(y_*)}$. We therefore conclude that there must be a singularity inside $\partial B_{\gamma^2}(y_*)$. On the other hand, the second condition of the sphere $\partial B_{r_l}(y_l)$ implies there is a different singularity of $\Theta_{\infty}$ locating outside while being $C\gamma$ close to $\partial B_{\gamma^2}(y_*)$. That is to say, we find two singularities of $\Theta_\infty$, which is a stable stationary harmonic map into $\mathbb{S}^3$, that locate $C\gamma$ close to each other. From the singularity theory in \cite{LW06}, this cannot happen at all if we further decrease $\gamma$ if necessary, which gives a contradiction. This finishes the outline of the proof of Lemma \ref{singannbdd}. Now let us start the proof:

\begin{proof}[Proof of Lemma \ref{singannbdd}]
Choose $\gamma$ small enough to be specified later. The proof goes by contradiction. Suppose there does not exist such constants $K_0$ and $N_0$. In other words, there exists $K_l\to \infty$, $N_l\to\infty$, a sequence of smooth connection $\{A_l\}_l$, the associated $\gamma^{K_l}$-annular region sequence $\{\mathcal{A}_l\}_l\eqqcolon \{A_{\gamma^{K_{x_l}},1}(x_l)\}_l$, and the associated Coulomb Minimizer sequence $\{\Theta_l\}_l$; moreover,
\begin{eqnarray}
\begin{split}
\label{nlinfty}
\#\{i: \text{Ann}_{i}(x_l)\cap \text{Sing}(\Theta_l)\neq \emptyset, 0\le i\le K_{x_l}\}\ge N_l.
\end{split}
\end{eqnarray}
First notice we have the following trivial relation: $N_l\le K_l\le K_{x_l}$. Now, using \eqref{nlinfty} and the pigeonhole principle, for all sufficiently large $l$ one can find $J_l\ge 10$ satisfying $K_{x_l}-J_l\to \infty$, such that there exists $y_l\in B_{\gamma^{J_l}/10}(x_l)$ and some $r_l\in [\gamma^{J_l}/2,2\gamma^{J_l}/3]$, with the following properties hold:
\begin{eqnarray}
\begin{split}
\label{trickysphere}
\exists z_l\in \text{Sing}(\Theta_l)\cap B_{\gamma^{J_l-1}}(x_l)\backslash B_{r_l}(y_l),\ \text{dist}(\partial B_{r_l}(y_l),\text{Sing}(\Theta_l))>\gamma r_l,\ \text{deg}(\Theta_l)\rvert_{\partial B_{r_l}(y_l)}\neq 0.
\end{split}
\end{eqnarray}
A detailed definition of the degree of the section $\text{deg}(\Theta_l)$ will be postponed to the end of this section, since to introduce it requires the techniques to be used later. Also the notion of the degree will not be needed until we introduce it in Definition \ref{trsp}. Most importantly, the existence of the sphere described in \eqref{trickysphere} will be clear once we have Definition \ref{trsp}. Let us note that the first condition in \eqref{trickysphere} is easily satisfied due to our assumption \eqref{nlinfty}; the second inequality in \eqref{trickysphere} can be achieved in view of Claim \ref{key} applied to $B_{r_l/5}(y)$ for all $y\notin B_{r_l/2}(x_l)$, provided $\gamma$ chosen to be sufficiently small. Now let us rescale $B_{r_l}(y_l)$ to $B_{\gamma^2}(y_l)$. For simplicity we shall still denote the connections and the Coulomb minimizers by $\{A_l\}_l$ and $\{\Theta_l\}_l$ respectively. As a matter of fact, we have the following upon rescaling:
\begin{eqnarray}
\begin{split}
&B_{\gamma^{J_l-1}}(x_l)\to B_{t_l}(x_l)\ \text{where }t_l\in [3\gamma/2,2\gamma],\ B_{\gamma^{-3}r_l}(y_l)\to B_{\gamma^{-1}}(y_l),\\
&B_{\gamma^{K_{x_l}}}(x_l)\to B_{s_l}(x_l)\ \text{where }s_l\eqqcolon\gamma^{K_{x_l}+2}/r_l.
\end{split}
\end{eqnarray}
Observe the following straightforward implications: $K_{x_l}-J_l\to \infty$ implies $s_l\to 0$; $y_l\in B_{\gamma^{J_l}/10}(x_l)$ implies $x_l\in B_{\gamma^2/5}(y_0)$; $\text{dist}(\partial B_{r_l}(y_l),\text{Sing}(\Theta_l))>\gamma r_l$ implies $z_l\in B_{3\gamma}(y_l)\backslash B_{\gamma^2+\gamma^3/2}(y_l)$. Let us now consider $\Theta_l\rvert_{B_1(y_l)\backslash B_{s_l}(x_l)}$. Firstly, let us identify all $\{y_l\}_l$ to be $y_*$. Then up to passing to subsequence we could assume
\begin{eqnarray}
\begin{split}
\label{zstar}
x_l\to x_*\in B_{\gamma^2/5}(y_*),\ z_l\to z_*\in B_{3\gamma}(y_*)\backslash B_{\gamma^2+\gamma^3/2}(y_*).
\end{split}
\end{eqnarray}
Moreover, from \eqref{trickysphere} we see that
\begin{eqnarray}
\begin{split}
\label{trickysphere2}
\text{dist}(\partial B_{\gamma^2}(y_*),\text{Sing}(\Theta_l))\ge\gamma^3/2,\ \text{deg}(\Theta_l)\rvert_{\partial B_{\gamma^2}(y_*)}\neq0.
\end{split}
\end{eqnarray}
In addition, due to $J_l\ge 10$, we have $B_{\gamma^{-3}r_l}(y_l)\subseteq B_{1/2}(x_l)$; thus, upon rescaling and using above identification, $\{\Theta_l\}_l$ are defined on $B_{\gamma^{-1}}(y_*)$. By the construction of a $\gamma^{K_{l}}$-annular region, we have
\begin{eqnarray}
\begin{split}
\label{curvkl}
\int_{B_{\gamma^{-1}}(y_*)\backslash B_{2s_l/3}(x_l)}|F_{A_l}|^2<\gamma^{K_l}.
\end{split}
\end{eqnarray}
Now let us fix an arbitrary sequence of positive constants $\eta(\alpha)\to 0$ as $\alpha\to \infty$. For each fixed $\alpha$ large enough, let us consider $\Theta_l\rvert_{B_{\gamma^{-1}}(y_*)\backslash B_{\eta(\alpha)}(x_*)}$. One subtlety that lies in the current situation, and makes it different from the previous compactness argument such as in the proof of Theorem \ref{numbersing} or Theorem \ref{apr}, is that the existence of an Uhlenbeck gauge is not guaranteed on the annulus $B_{\gamma^{-1}}(y_*)\backslash B_{\eta(\alpha)}(x_*)$ in spite of the curvature smallness \eqref{curvkl}. To get round this issue, we directly study the limit section $\Theta_\infty$ of the limit flat bundle, by using a partition of unity argument to prove its Coulomb property, stationarity, as well as the stability. Lastly, we show that the trivialization of $\Theta_\infty$ under the canonical flat section of the limit flat bundle is a stationary stable harmonic map into $\text{SU}(2)\cong \mathbb{S}^3$. To begin, we need to obtain for each $\alpha$ a frame $\Theta_{\alpha,\infty}$ on the limit bundle which is isomorphic to $E_{\infty,\alpha}=\bigg{(}B_{\gamma^{-1}}(y_*)\backslash B_{\eta(\alpha)}(x_*)\bigg{)}\times \text{Mat}_{\mathbb{R}}(k\times k).$ Fix $\alpha$, and denote $\eta\equiv \eta(\alpha)$ for simplicity. Cover the annulus $B_{7\gamma^{-1}/8}(y_*)\backslash B_\eta(x_*)$ by a finite collection of balls $\{B_{\eta,a}\}_{a=1}^{N_\eta}$, such that $\{B_{\eta,a}/2\}_{a=1}^{N_\eta}$ still forms a cover, and that on every ball $B_{\eta,a}$ and all sufficiently large $l$, there exists a Coulomb gauge $\sigma_{\eta,a,l}$ of $A_l$ defined on $B_{\eta,a}$ with connection form $A^*_{\eta,a,l}$ that satisfies \eqref{uhgauge}. Denote by $\Theta_{\eta,a,l}^*$ the trivialization of $\Theta_l$ under the Uhlenbeck gauge $\sigma_{\eta,a,l}$. By the property of the Uhlenbeck gauge \eqref{uhgauge}
\begin{eqnarray}
\begin{split}
\label{lastcon00}
A^*_{\eta,a,l}\to A^*_{\eta,a,\infty}\equiv 0\ \text{in $W^{1,2}(B_{\eta,a})$}.
\end{split}
\end{eqnarray}
Applying Corollary \ref{keycoro} we obtain $\Theta^*_{\eta,a,\infty}$ for each $a$ such that
\begin{eqnarray}
\begin{split}
\label{lastcon11}
\Theta^*_{\eta,a,l}\to \Theta^*_{\eta,a,\infty}\ \text{in $W^{1,2}(B_{\eta,a})$.}
\end{split}
\end{eqnarray}
Let $g_{\eta,ab,l}$ be the transition function such that $g_{\eta,ab,l}^*A^*_{\eta,b,l}=A^*_{\eta,a,l}$. By the assumption \eqref{curvkl}, there exists an element $q_{\eta,ab}\in \text{SU}(2)$ such that
\begin{eqnarray}
\begin{split}
\label{lastcon1}
g_{\eta,ab,l}\to q_{\eta,ab}\ \text{in $W^{2,2}(B_{\eta,a,l}\cap B_{\eta,b,l})$}.
\end{split}
\end{eqnarray}
Therefore by \eqref{lastcon11} and \eqref{lastcon1} we have
\begin{eqnarray}
\begin{split}
\label{lastcon2}
g_{\eta,ab,l}\cdot \Theta^*_{\eta,b,l}\to q_{\eta,ab}\cdot & \Theta^*_{\eta,b,\infty}\ \text{in $W^{1,2}(B_{\eta,a,l}\cap B_{\eta,b,l})$},\\ 
q_{\eta,ab}\cdot q_{\eta,bc}&=q_{\eta,ac}.
\end{split}
\end{eqnarray}
On the other hand, by the formula $g_{\eta,ab,l}\cdot \Theta^*_{\eta,b,l}=\Theta^*_{\eta,a,l}$ and \eqref{lastcon} we have
\begin{eqnarray}
\begin{split}
g_{\eta,ab,l}\cdot \Theta^*_{\eta,b,l}\to \Theta^*_{\eta,a,\infty}\ \text{in $W^{1,2}(B_{\eta,a,l}\cap B_{\eta,b,l})$}.
\end{split}
\end{eqnarray}
Thus 
\begin{eqnarray}
\begin{split}
\label{secpat}
q_{\eta,ab}\cdot  \Theta^*_{\eta,b,\infty}=\Theta^*_{\eta,a,\infty}.
\end{split}
\end{eqnarray}
The second identity in \eqref{lastcon2} together with \eqref{lastcon00}, determines a flat bundle $E_{\infty,\alpha}$ in the limit defined over $B_{\gamma^{-1}}(y_*)\backslash B_{\eta(\alpha)}(x_*)$. Because $B_{\gamma^{-1}}(y_*)\backslash B_{\eta(\alpha)}(x_*)$ is homeomorphic to a $4$ dimensional annulus which is simply connected, $E_{\infty,\alpha}$ is indeed a trivial bundle:
$$E_{\infty,\alpha}=\bigg{(}B_{\gamma^{-1}}(y_*)\backslash B_{\eta(\alpha)}(x_*)\bigg{)}\times \text{Mat}_{\mathbb{R}}(k\times k).$$
Finally, \eqref{secpat} together with the second identity in \eqref{lastcon2} implies that $\{\Theta^*_{\eta,a,\infty}\}_a$ determine a frame $\Theta_{\alpha,\infty}$ on $E_\infty$. By the arbitrariness of $\alpha$ and a diagonal argument, we find a frame $\Theta_\infty$ on the flat bundle $E_{\infty}=B_{\gamma^{-1}}(y_*)\times \text{Mat}_{\mathbb{R}}(k\times k)$ with associated principal $\text{SU}(2)$-bundle $P_{\infty}$ and flat connection $A_\infty$. For convenience, let us identify $(P_\infty,A_{\infty})$ with $(B_{\gamma^{-1}}(y_*)\times\text{SU}(2),A_{\infty})$ such that for the canonical section
\begin{eqnarray}
\begin{split}
\label{cansec}
s_0: B_{\gamma^{-1}}(y_*)\to B_{\gamma^{-1}}(y_*)\times\text{SU}(2),\ x\mapsto (x,1),
\end{split}
\end{eqnarray}
the following holds:
\begin{eqnarray}
\begin{split}
\label{cansec1}
s_0^*A_{\infty}\equiv0.
\end{split}
\end{eqnarray}
Define $\Theta^*_\infty=\Theta_\infty\circ s_0$. We first show that $\Theta^*_\infty\in W^{1,2}(B_{5\gamma^{-1}/6}(y_*))$. Let $\eta$ be an arbitrary small but fixed positive number, and we will be using the same notations as before for the constructions such as the cover $\{B_{\eta,a}\}_a$, the Uhlenbeck gauges $A^*_{\eta,a,l}$, and the trivializations of $\Theta_l$, $A_l$ under these gauges. Recall we have the following convergences hold:
\begin{eqnarray}
\begin{split}
\label{lastcon}
A^*_{\eta,a,l}\to A^*_{\eta,a,\infty}\ \text{in $L^2(B_{\eta,a})$},\ \Theta^*_{\eta,a,l}\to \Theta^*_{\eta,a,\infty}\ \text{in $W^{1,2}(B_{\eta,a})$ and weakly in $W^{2,2}_{\text{loc}}(B_{\eta,a}\backslash \Sigma)$}.
\end{split}
\end{eqnarray}
where $\Sigma$ is defined in the same way as \eqref{singsigma}. Choose a partition of unity $\{\phi_{\eta,a}\}_a$ subordinate to the cover $\{B_{\eta,a}\}_a$, such that $\text{supp}\phi_{\eta,a}\subseteq B_{\eta,a}$ and $\sum_a\phi_{\eta,a}\equiv 1$ on $B_{5\gamma^{-1}/6}(y_*)\backslash B_{2\eta}(x_*)$. Using Corollary \ref{scalinginvariant}, \eqref{lastcon}, and the partition of unity, we have the following estimates:
\begin{eqnarray}
\begin{split}
\label{unity}
C(\gamma,\Lambda)&\ge  \limsup_{l\to\infty}\int_{B_{7\gamma^{-1}/8}(y_*)}|\nabla_{A_l}\Theta_l|^2\ge \limsup_{l\to\infty}\int_{B_{7\gamma^{-1}/8}(y_*)\backslash B_{\eta}(x_*)}|\nabla_{A_l}\Theta_l|^2\\
&\ge \sum_a \lim_{l\to\infty}\int_{B_{7\gamma^{-1}/8}(y_*)\backslash B_{\eta}(x_*)}|\nabla_{A_l}\Theta_l|^2\phi_{\eta,a}= \sum_a \lim_{l\to\infty}\int_{B_{\eta,a}}|\nabla_{A_l}\Theta_l|^2\phi_{\eta,a}\\
&= \sum_a \int_{B_{\eta,a}}|d\Theta^*_{\eta,a,\infty}|^2\phi_{\eta,a}.
\end{split}
\end{eqnarray}
Let $\sigma_{\eta,a,\infty}$ be the local $P_\infty$-section such that $\Theta_\infty\circ \sigma_{\eta,a,\infty}=\Theta^*_{\eta,a,\infty}$. Then due to the fact that both $s_0^*A_\infty$ and $\sigma_{\eta,a,\infty}^*A_\infty$ are $0$ (see \eqref{lastcon00}), the gauge transformation formula then implies $\sigma_{\eta,a,\infty}\cdot g_a=s_0$ for some $\text{SU}(2)$ element $g_a$. Now that $\Theta^*_{\eta,a,\infty}$ and $\Theta^*_\infty$ are trivialization under $\sigma_{\eta,a,\infty}$ and $s_0$ on the ball $B_{\eta,a}$ respectively, we have
\begin{eqnarray}
\begin{split}
\label{modulus0}
g_a\cdot \Theta^*_{\eta,a,\infty}=\Theta^*_\infty.
\end{split}
\end{eqnarray}
Since $g_a$ is constant, the following is obvious:
\begin{eqnarray}
\begin{split}
\label{modulus}
|d\Theta^*_{\eta,a,\infty}|=|d\Theta^*_\infty|\ \text{on $B_{\eta,a}$.}
\end{split}
\end{eqnarray}
Insert \eqref{modulus} into the right hand side of \eqref{unity}, we obtain
\begin{eqnarray}
\begin{split}
C(\gamma,\Lambda)&\ge \sum_a \int_{B_{\eta,a}}|d\Theta^*_{\infty}|^2\phi_{\eta,a}=\int_{B_{7\gamma^{-1}/8}(y_*)\backslash B_{\eta}(x_*)}|d\Theta^*_{\infty}|^2\sum_a\phi_{\eta,a}\\
&\ge \int_{B_{5\gamma^{-1}/6}(y_*)\backslash B_{2\eta}(x_*)}|d\Theta^*_{\infty}|^2.
\end{split}
\end{eqnarray}
By the arbitrariness of $\eta$, we proved that $\Theta^*_\infty\in W^{1,2}(B_{5\gamma^{-1}/6}(y_*))$. Indeed, without much more work one can produce refine the preceding argument to show that 
\begin{eqnarray}
\begin{split}
\label{refine0}
\lim_{l\to \infty} \int_{B_r(x)}|\nabla_{A_l}\Theta_l|^2=\int_{B_r(x)}|d\Theta_\infty|^2,\ \text{for all $B_r(x)\subseteq B_{5\gamma^{-1}/6}(y_*)$},
\end{split}
\end{eqnarray}
and thus by Corollary \ref{scalinginvariant} we have
\begin{eqnarray}
\begin{split}
\label{refine1}
r^{-2}\int_{B_r(x)}|d\Theta^*_\infty|^2\le C(\Lambda),\ \text{for all $B_r(x)\subseteq B_{5\gamma^{-1}/6}(y_*)$}.
\end{split}
\end{eqnarray}
\begin{proposition}
\label{shs}
$\Theta^*_\infty$ is a stable stationary harmonic map in $B_{5\gamma^{-1}/6}(y_*)$.
\end{proposition}
\begin{proof}
The weakly harmonic map equation of $\Theta^*_\infty$ follows easily by using a partition of unity and a limiting argument; since the same process will be repeatedly used later in showing the stationarity and the stability in a much more complicated fashion, we omit the details for now. Let us focus on showing the stationarity of $\Theta^*_\infty$. For this purpose choose a vector field $X$ compactly supported in $B_{3\gamma^{-1}/4}(y_*)$.  Again, let us arbitrarily fix a small positive number $\eta$ and consider the same cover of $B_{7\gamma^{-1}/8}(y_*)\backslash B_\eta(x_*)$ by $\{B_{\eta,a}\}_{a=1}^{N_\eta}$ as in the preceding proof of the $W^{1,2}$ boundedness. For convenience, we will use the same notations as before to denote the Uhlenbeck gauges and the partition of unity subordinate to the cover. In addition, choose $\chi_{\eta}$ to be a proper cutoff function supported on $B_{3\eta}(x_*)$ with $\chi_{\eta}\equiv 1$ on $B_{2\eta}(x_*)$ and $\sup_{x\in B_{2\eta}(x_*)}|d\chi_{\eta}(x)|<C\eta^{-1}$ for some universal constant $C$. Firstly, let us write
\begin{eqnarray}
\begin{split}
\label{holesta}
\int_{B_{3\gamma^{-1}/4}(y_*)} \bigg{(}|\nabla_{A_\infty} \Theta_{\infty}|^2g_{ik}-2\langle \nabla_{A_\infty}\Theta_{\infty}(\partial_i),\nabla_{A_\infty}\Theta_{\infty}(\partial_k)\rangle \bigg{)}g^{ij}{X^k}_{;j}=E_1+E_2,
\end{split}
\end{eqnarray}
where
\begin{eqnarray}
\begin{split}
E_1=&\int_{B_{3\gamma^{-1}/4}(y_*)}  \bigg{(}|\nabla_{A_\infty} \Theta_{\infty}|^2g_{ik}-2\langle \nabla_{A_\infty}\Theta_{\infty}(\partial_i),\nabla_{A_\infty}\Theta_{\infty}(\partial_k)\rangle \bigg{)}g^{ij}{[(1-\chi_{\eta} )X^k]}_{;j},\\
E_2=&\int_{B_{3\gamma^{-1}/4}(y_*)}  \bigg{(}|\nabla_{A_\infty} \Theta_{\infty}|^2g_{ik}-2\langle \nabla_{A_\infty}\Theta_{\infty}(\partial_i),\nabla_{A_\infty}\Theta_{\infty}(\partial_k)\rangle \bigg{)}{[\chi_{\eta} X^k]}_{;j}.
\end{split}
\end{eqnarray}
Recall that $A_\infty$ is the flat connection on the limit flat bundle $P_\infty$. Using the partition of unity and writing $\Theta_\infty$ with respect to the Uhlenbeck gauge on each $B_{\eta,a}$, and lastly use \eqref{modulus0} and \eqref{modulus}, one could express $E_1$ as follows:
\begin{eqnarray}
\begin{split}
E_1=&\sum_a\int_{B_{\eta,a}} \bigg{(}|d \Theta^*_{\infty}|^2g_{ik}-2\langle \partial_i\Theta^*_{\infty},\partial_k\Theta^*_{\infty}\rangle \bigg{)}g^{ij}{[\phi_{\eta,a}(1-\chi_{\eta})X^k]}_{;j}\\
=&\sum_a\int_{B_{\eta,a}} \bigg{(}|d \Theta^*_{\eta,a,\infty}|^2g_{ik}-2\langle \partial_i\Theta^*_{\eta,a,\infty},\partial_k\Theta^*_{\eta,a,\infty}\rangle \bigg{)}g^{ij}{[\phi_{\eta,a}(1-\chi_{\eta})X^k]}_{;j}.
\end{split}
\end{eqnarray} 
Combining the convergences $\Theta_{\eta,a,l}^*\to \Theta_{\eta,a,\infty}$ in $W^{1,2}(B_{\eta,a})$, $A^*_{\eta,a,l}\to A^*_{\eta,a,\infty}$ in $L^2(B_{\eta,a})$, as well as $F_{A_l}\to 0$ in $L^2(B_{\eta,a})$ we have
\begin{eqnarray}
\begin{split}
\label{stalim}
&\sum_a\lim_{l\to\infty}\int_{B_{\eta,a}} \bigg{(}|\nabla_{A_{l}} \Theta^*_{\eta,a,l}|^2g_{ik}-2\langle \nabla_{A_{l}}\Theta^*_{\eta,a,l}(\partial_i),\nabla_{A_{l}}\Theta^*_{\eta,a,l}(\partial_k)\rangle \bigg{)}g^{ij}{[\phi_{\eta,a}(1-\chi_{\eta})X^k]}_{;j}\\
&+\sum_a2\int_{}\langle {F_{A_{l},i}}^s(\Theta^*_{\eta,a,l}),\nabla_{A_{l},s}\Theta^*_{\eta,a,l}\rangle \phi_{\eta,a}(1-\chi_{\eta})X^i\\
=&\sum_a\int_{B_{\eta,a}} \bigg{(}|\nabla_{A_{\infty}} \Theta^*_{\eta,a,\infty}|^2g_{ik}-2\langle \nabla_{A_{\infty}}\Theta^*_{\eta,a,\infty}(\partial_i),\nabla_{A_{\infty}}\Theta^*_{\eta,a,\infty}(\partial_k)\rangle \bigg{)}g^{ij}{[\phi_{\eta,a}(1-\chi_{\eta})X^k]}_{;j}\\
=&E_1.
\end{split}
\end{eqnarray}
On the other hand, since the stationarity equation \eqref{mono} holds for $\Theta_{l}$, the left hand side of \eqref{stalim} is identically zero. Therefore $E_1=0$. Next we estimate $E_2$. Without loss of generality, we assume $2\eta=2^{-P_{0}}$. We we apply \eqref{refine1} together with the fact that $\sup|\nabla \phi_{\eta}|\le C\eta^{-1}$ to obtain
\begin{eqnarray}
\begin{split}
\label{e20}
|E_2|&\le C(X)\eta^{-1}\sum_{\alpha=P_0}^{\infty}\int_{B_{3\eta}(x_*)}|d \Theta^*_{\infty}|^2\le C(X)\eta^{-1}\cdot C\eta^2\le C^{\prime}\eta.
\end{split}
\end{eqnarray}
Due to the arbitrariness of $\eta$, we see that the left hand side of \eqref{holesta} vanishes. Therefore, $\Theta_{\infty}$ is stationary on $B_{3\gamma^{-1}/4}(y_*)$. 
\medskip

Next, we prove the stability of $\Theta_\infty$. Choose arbitrary $\xi(t): [0,\epsilon]\to C^\infty_c(B_{5\gamma^{-1}/6}(y_*),\mathfrak{g})$, by \eqref{long} we have
\begin{eqnarray}
\begin{split}
\label{flathess}
&\frac{d^2}{dt^2}\bigg{\rvert}_{t=0}\int_{B_{5\gamma^{-1}/6}(y_*)}|d\big{(}\exp(\xi(t))\Theta^*_\infty\big{)}|^2\\
=&2\int_{B_{5\gamma^{-1}/6}(y_*)}|d\xi^{\prime}(0)|^2+2\int_{B_{5\gamma^{-1}/6}(y_*)}\langle \big{(}[d\xi^{\prime}(0),\xi^{\prime}(0)]+d\xi^{\prime\prime}(0)\big{)}(\Theta^*_\infty),d\Theta^*_\infty \rangle.
\end{split}
\end{eqnarray}
For arbitrary $\eta>0$, consider a standard cutoff function $\phi_\eta$ such that $\phi_\eta\equiv 1$ in $B_{2\eta}(x_*)$ and vanishes outside $B_{3\eta}(x_*)$; moreover $\sup|\nabla\phi_\eta|\le C\eta^{-1}$. Now let us write the right hand side of \eqref{flathess} as
\begin{eqnarray}
\begin{split}
\label{part2}
&2\int_{B_{5\gamma^{-1}/6}(y_*)}|d\big{(}\phi_\eta\xi^{\prime}(0)\big{)}|^2+2\int_{B_{5\gamma^{-1}/6}(y_*)}\bigg{\langle} \big{(}[d\big{(}\phi_\eta\xi^{\prime}(0)\big{)},\phi_\eta\xi^{\prime}(0)]+d\big{(}\phi_\eta\xi^{\prime\prime}(0)\big{)}(\Theta^*_\infty),d\Theta^*_\infty \bigg{\rangle}\\
+&\text{Err}_{\eta,\xi},
\end{split}
\end{eqnarray}
where the error term $\text{Err}_{\eta,\xi}$ could be estimated as follows:
\begin{eqnarray}
\begin{split}
\label{erreta}
|\text{Err}_{\eta,\xi}|\le C(\xi)\int_{B_{5\gamma^{-1}/6}(y_*)}\bigg{(}|\nabla\phi_\eta|^2+|\nabla\phi_\eta|\big{(}1+|d\Theta^*_\infty|\big{)}\bigg{)}\le C(\xi,\Lambda)\eta^2.
\end{split}
\end{eqnarray}
Here we have used the fact that $\sup|\nabla \phi_\eta|\le C\eta^{-1}$ together with the estimate \eqref{refine1}. Now let us estimate the first term in \eqref{part2}. By defining
\begin{eqnarray}
\begin{split}
\label{xi0}
\xi_0\eqqcolon \phi_\eta\xi:[0,\epsilon]\to C^{\infty}_c(B_{5\gamma^{-1}/6}(y_*)\backslash B_{2\eta}(x_*),\mathfrak{g}),
\end{split}
\end{eqnarray}
the first term in \eqref{part2} becomes
\begin{eqnarray}
\begin{split}
\label{xi00}
&2\int_{B_{5\gamma^{-1}/6}(y_*)}|d\big{(}\xi_0^{\prime}(0)\big{)}|^2+2\int_{B_{5\gamma^{-1}/6}(y_*)}\bigg{\langle} \big{(}[d\big{(}\xi_0^{\prime}(0)\big{)},\xi_0^{\prime}(0)]+d\big{(}\xi_0^{\prime\prime}(0)\big{)}(\Theta^*_\infty),d\Theta^*_\infty \bigg{\rangle}\\
=&\frac{d^2}{dt^2}\bigg{\rvert}_{t=0}\int_{B_{5\gamma^{-1}/6}(y_*)}|d\big{(}\exp(\xi_0(t))\Theta^*_\infty\big{)}|^2.
\end{split}
\end{eqnarray}
Unlike verifying the stationarity where we simply use the same test vector field $X$ for all $\Theta_l$ in the limiting sequence, this time the situation becomes more complicated. Indeed, we have to find for all $\Theta_l$ the test functions $\xi_l\in W^{1,2}(B_{5\gamma^{-1}/6}(y_*)\backslash B_{2\eta}(x_*),\mathfrak{g}_{P_l})$ (which are sections all living on different bundles), such that when $l$ goes to infinity $\xi_l$ ``converges'' to $\xi_0$ in a certain sense to be made precise later. To rigidify the process, we need to prove the following proposition; in the statement, we follow the same notations of the cover $\{B_{\eta,a}\}_a$ of $B_{5\gamma^{-1}/6}(y_*)\backslash B_{2\eta}(x_*)$ and the Uhlenbeck gauges on each $B_{\eta,a}$ as in the preceding paragraphs.
\begin{proposition}
\label{xicon}
Fix $\eta>0$. For any smooth mapping $\xi(t): [0,\epsilon]\to C^{\infty}_c(B_{5\gamma^{-1}/6}(y_*)\backslash B_{2\eta}(x_*),\mathfrak{g})$ with $\xi(0)\equiv 0$, up to a subsequence for all $l$ there exists a smooth mapping 
$$\xi_l(t): [0,\epsilon]\to W^{1,2}(B_{5\gamma^{-1}/6}(y_*)\backslash B_{2\eta}(x_*),\mathfrak{g}_{P_l})$$ 
with $\xi_l(0)\equiv 0$ satisfying the following: for every $a$, there exists an element $g_a\in \text{SU}(2)$, such that all following convergences hold in $W^{1,2}(2B_{\eta,a}/3)$:
\begin{eqnarray}
\begin{split}
\label{3w12}
&\xi^*_{\eta,a,l}(t)\to \text{Ad}_{{g_a}^{-1}}\xi(t)\ \text{for all $t\in [0,\epsilon]$},\ {\xi^*_{\eta,a,l}}^{\prime}(0)\to \text{Ad}_{{g_a}^{-1}}\xi^{\prime}(0),\\
&{\xi^*_{\eta,a,l}}^{\prime\prime}(0)(\Theta^*_{\eta,a,l})\to g_a^{-1}\cdot\xi^{\prime\prime}(0)(\Theta^*_{\infty}),\\
&[d{\xi^*_{\eta,a,l}}^{\prime}(0),{\xi^*_{\eta,a,l}}^{\prime}(0)](\Theta^*_{\eta,a,l})\to g_a^{-1}\cdot[d{\xi^*_{\infty}}^{\prime}(0),{\xi^*_{\infty}}^{\prime}(0)](\Theta^*_{\infty}).
\end{split}
\end{eqnarray}
Here ``$\text{Ad}$" stands for the adjoint action of $\text{SU}(2)$ on $\mathfrak{g}\equiv\mathfrak{s}\mathfrak{u}(2)$, and $\xi^*_{\eta,a,l}$ denotes the trivialization of $\xi_l$ under the Uhlenbeck gauge $\sigma_{\eta,a,l}$.
\end{proposition}
\begin{proof}
Without ambiguity we drop the symbol ``$t$''. For any $\xi: [0,\epsilon]\to C^{\infty}_c(B_{5\gamma^{-1}/6}(y_*)\backslash B_{2\eta}(x_*),\mathfrak{g})$, define $\xi_0: [0,\epsilon] \to W_0^{1,2}(B_{5\gamma^{-1}/6}(y_*)\backslash B_{2\eta}(x_*),\mathfrak{g}_{P_\infty})$ by $\xi_0\circ s_0=\xi$, where $s_0$ is the same as in \eqref{cansec}. Consider $\xi^{\prime}\eqqcolon \xi_0\circ \Theta_\infty$. Apparently, $\xi^{\prime}=\text{Ad}_{{\Theta_\infty^*}^{-1}}\xi$. Define $\xi_l\in  W^{1,2}_0(B_{5\gamma^{-1}/6}(y_*)\backslash B_{2\eta}(x_*),\mathfrak{g}_{P_l})$ by $\xi_l\circ \sigma_l\eqqcolon\xi^{\prime}$, where $\Theta_l\circ\sigma_l\equiv 1$. We will prove that, upon subsequence, $\{\xi_l\}_l$ is as desired. By the following identities:
\begin{eqnarray}
\begin{split}
\sigma_{\eta,a,l}\cdot \Theta^*_{\eta,a,l}=\Theta_l,\ \xi_l\circ\sigma_l=\text{Ad}_{{\Theta_\infty^*}^{-1}}\xi,
\end{split}
\end{eqnarray}
it is straightforward to verify that
\begin{eqnarray}
\begin{split}
\label{hessl}
&\xi_{\eta,a,l}^*=\text{Ad}_{\Theta^*_{\eta,a,l}}\circ \text{Ad}_{{\Theta_\infty^*}^{-1}}\xi=\text{Ad}_{\Theta^*_{\eta,a,l}\cdot {\Theta_\infty^*}^{-1}}\xi,\\
&{\xi^*_{\eta,a,l}}^{\prime\prime}(0)(\Theta^*_{\eta,a,l})=\Theta^*_{\eta,a,l}\cdot {\Theta_\infty^*}^{-1}\cdot\xi^{\prime\prime}(0)\cdot\Theta^*_\infty,\\
&[d{\xi^*_{\eta,a,l}}^{\prime}(0),{\xi^*_{\eta,a,l}}^{\prime}(0)](\Theta^*_{\eta,a,l})=\Theta^*_{\eta,a,l}\cdot {\Theta_\infty^*}^{-1}\cdot [d\xi^{\prime}(0),\xi^{\prime}(0)]\cdot\Theta^*_\infty.
\end{split}
\end{eqnarray}
To prove the convergences in \eqref{3w12}, the key ingredient is the weak convergence of $\Theta^*_{\eta,a,l}$ in $W^{2,2}_{\text{loc}}(B_{\eta,a}\backslash \Sigma)$, again a consequence of the $\epsilon$-regularity theorem \ref{eps}. Combining this with a capacity argument similar to that in the proof of Theorem \ref{cpt} will suffice to conclude Proposition \ref{xicon}. Since the proof of all the three convergences in \eqref{3w12} are almost identical, we only prove the first one in details.
\medskip

Using the first expression in \eqref{hessl}, we have
\begin{eqnarray}
\begin{split}
\label{hesss}
|\nabla^2\xi_{\eta,a,l}|\le& |\nabla^2 \Theta^*_{\eta,a,l}||\xi|+|\nabla^2 \Theta^*_{\infty}||\xi|+|\nabla^2\xi|+2|\nabla \xi||d\Theta_{\eta,a,l}^*|+2|\nabla \xi||d\Theta_{\infty}^*|\\
&+2|d\Theta_{\eta,a,l}^*||d\Theta_{\infty}^*||\xi|=W_1+W_2+W_3+W_4+W_5+W_6;
\end{split}
\end{eqnarray}
\begin{eqnarray}
\begin{split}
\label{hessss}
|d\xi_{\eta,a,l}|\le& |\nabla \xi|+|d\Theta_{\infty}^*||\xi|+|d\Theta_{\eta,a,l}^*||d\Theta_{\infty}^*||\xi|=V_1+V_2+V_3.
\end{split}
\end{eqnarray}
By \eqref{lastcon}, \eqref{modulus}, and \eqref{hesss}, there exists $\xi_a\in W^{2,2}_{\text{loc}}(B_{\eta,a}\backslash \Sigma)$ such that $\xi_{\eta,a,l}\to \xi_{a}$ weakly in $W^{2,2}_{\text{loc}}(B_{\eta,a}\backslash \Sigma)$ and strongly in $W^{1,2}_{\text{loc}}(B_{\eta,a}\backslash\Sigma)$. Next we show that for every $\tau>0$ one could find a cutoff function $\zeta_\tau\in C^\infty_c(B_{\eta,a})$ such that $\zeta_\tau\equiv 1$ in a neighborhood $\mathcal{N}$ of $\Sigma\cap 2B_{\eta,a}/3$, and that $\sup_l \int_{B_{\tau,a}}|d\xi^*_{\eta,a,l}|^2\zeta_\tau^2\le C(\xi)\tau$. Indeed, by the convergence $\Theta_{\eta,a,l}\to \Theta_{\eta,a,\infty}$ in $W^{1,2}(B_{\eta,a})$, $A^*_{\eta,a,l}\to A^*_{\eta,a,\infty}$ in $L^2(B_{\eta,a})$, as well as \eqref{sta1} applied to $\Theta^*_{\eta,a,l}$, it is not hard to see that for all $\zeta\in C^\infty_c(B_{\eta,a})$ the following holds:
\begin{eqnarray}
\begin{split}
\label{taua}
\int_{B_{\tau,a}}|d\Theta^*_{\eta,a,\infty}|^2\zeta^2\le C\bigg{(}\int_{B_{\tau,a}}|A^*_{\eta,a,\infty}|^2\zeta^2+|\nabla\zeta|^2\bigg{)}.
\end{split}
\end{eqnarray}
The fact that $H^2(\Sigma\cap B_{\eta,a})<\infty$ implies that the $2$-capacity $\text{Cap}_2(\Sigma\cap B_{\eta,a})=0$. Therefore, one could find a small neighborhood $\mathcal{N}$ of $\Sigma\cap 2B_{\eta,a}/3$ (which could be made arbitrarily small) and a cutoff function $\zeta_\tau\in C^{\infty}_{c}(B_{\eta,a})$ such that $\zeta_\tau\equiv 1$ on $\mathcal{N}$ and
\begin{eqnarray}
\begin{split}
\label{taub}
\int_{B_{\tau,a}}|\nabla \zeta_{\tau}|^2<\tau.
\end{split}
\end{eqnarray}
In addition, upon further shrinking $\mathcal{N}$, one can achieve
\begin{eqnarray}
\begin{split}
\label{tauc}
\sup_l\int_{B_{\tau,a}}|A^*_{\eta,a,l}|^2\zeta_\tau^2<\tau,\ \int_{B_{\tau,a}}|A^*_{\eta,a,\infty}|^2\zeta_\tau^2<\tau.
\end{split}
\end{eqnarray}
Now, combining \eqref{sta1} applied to $\Theta_{\eta,a,l}$, \eqref{taua}, \eqref{taub}, \eqref{tauc}, and \eqref{hessss}, we have
\begin{eqnarray}
\begin{split}
\label{taud}
\int_{B_{\tau,a}}|d\xi^*_{\eta,a,l}|^2\zeta_\tau^2< C(\xi)\tau.
\end{split}
\end{eqnarray}
Now from \eqref{taud} together with the strong $W^{1,2}_{\text{loc}}(B_{\eta,a}\backslash\Sigma)$ convergence, we conclude that $\xi_{\eta,a,l}\to \xi_a$ in $W^{1,2}(2B_{\eta,a}/3)$ for some $\xi_a$.
\medskip

Lastly, by \eqref{modulus0}, the convergence $\Theta_{\eta,a,l}\to \Theta_{\eta,a,\infty}$ in $W^{1,2}(B_{\eta,a})$, $\xi_{\eta,a,l}\to \xi_a$ in $W^{1,2}(2B_{\eta,a}/3)$, and the expression \eqref{hessl}, we see that $\xi_{a}=\text{Ad}_{{g_a}^{-1}}\xi$ a.e.. This finishes the proof of the first convergence in \eqref{3w12}. For the other two convergences we use the other two expressions in \eqref{hessl} to derive similar estimates as \eqref{hesss}, and the rest of proof are verbatim. We thus complete the proof of Proposition \ref{xicon}.
\end{proof}
Now we continue to estimate the first term in \eqref{part2}. Let us apply Proposition \ref{xicon} to $\xi_0$ defined in \eqref{xi0}. Then upon passing to a subsequence we obtain for each $l$ a mapping $\xi_l$ described in Proposition \ref{xicon}. Choosing the same partition of unity $\{\phi_{\eta,a}\}$ that is subordinate to the cover $\{B_{\eta,a}\}_a$, we estimate as follows:
\begin{eqnarray}
\begin{split}
\label{lstab}
&\frac{d^2}{dt^2}\bigg{\rvert}_{t=0}\int_{B_{5\gamma^{-1}/6}(y_*)}|\nabla_{A_l}\big{(}\exp(\xi_l(t))\Theta_l\big{)}|^2=\sum_a \frac{d^2}{dt^2}\bigg{\rvert}_{t=0}\int_{B_{5\gamma^{-1}/6}(y_*)}|\nabla_{A_l}\big{(}\exp(\xi_l(t))\Theta_l\big{)}|^2\phi_a\\
=&\sum_a \frac{d^2}{dt^2}\bigg{\rvert}_{t=0}\int_{B_{\eta,a}}|\nabla_{A_l}\big{(}\exp(\xi_l(t))\Theta_l\big{)}|^2\phi_a=\sum_a \frac{d^2}{dt^2}\bigg{\rvert}_{t=0}\int_{B_{\eta,a}}|\nabla_{A_l}\big{(}\exp(\xi^*_{\eta,a,l}(t))\Theta_{\eta,a,l}\big{)}|^2\phi_a\\
=&\int_{B_{\eta,a}}e_{\xi^*_{\eta,a,l}}\phi_a+\int_{B_{\eta,a}}m_{\xi^*_{\eta,a,l}}\phi_a,
\end{split}
\end{eqnarray}
where $e_{\xi^*_{\eta,a,l}}$ and $m_{\xi^*_{\eta,a,l}}$ correspond to the two terms in \eqref{long} respectively. Trivially we have
\begin{eqnarray}
\begin{split}
e_{\xi^*_{\eta,a,l}}&\le C_{\xi_0}|A^*_{\eta,a,l}|\big{(}|d\Theta^*_{\eta,a,l}|+|d\Theta^*_{\infty}|\big{)},\\
m_{\xi^*_{\eta,a,l}}&=2|d{\xi^*_{\eta,a,l}}^{\prime}(0)|^2+2\bigg{\langle} \big{(}[d{\xi^*_{\eta,a,l}}^{\prime}(0),{\xi^*_{\eta,a,l}}^{\prime}(0)]+d{\xi^*_{\eta,a,l}}^{\prime\prime}(0)\big{)}(\Theta^*_{\eta,a,l}),d\Theta^*_{\eta,a,l} \bigg{\rangle}.
\end{split}
\end{eqnarray}
By the convergence $A^*_{\eta,a,l}\to A^*_{\eta,a,\infty}=0$ in $L^2(B_{\eta,a})$ and the uniform boundedness of $\Theta^*_{\eta,a,l}$ in $W^{1,2}(B_{\eta,a})$, we see that 
\begin{eqnarray}
\begin{split}
\label{elim}
\int_{B_{\eta,a}}e_{\xi^*_{\eta,a,l}}\phi_a\to 0
\end{split}
\end{eqnarray}
as $l\to \infty$. Next, we apply the convergences \eqref{3w12} in Proposition \ref{xicon}, together with the convergence $\Theta^*_{\eta,a,l}\to \Theta^*_{\eta,a,\infty}$ in $W^{1,2}(B_{\eta,a})$ to obtain the following convergence:
\begin{eqnarray}
\begin{split}
\label{mlim}
&\lim_{l\to\infty}\int m_{\xi^*_{\eta,a,l}}\phi_a\\
=&\int_{B_{\eta,a}}\bigg{(}2|\text{Ad}_{g_a^{-1}}d{\xi_{0}}^{\prime}(0)|^2+2\bigg{\langle} g_a^{-1}\cdot\big{(}[d{\xi_{0}}^{\prime}(0),{\xi_{0}}^{\prime}(0)]+d{\xi_{0}}^{\prime\prime}(0)\big{)}(\Theta_{\infty}^*), d\Theta_{\eta,l,\infty}^* \bigg{\rangle}\bigg{)}\phi_a\\
=&\int_{B_{\eta,a}}\bigg{(}2|d{\xi_{0}}^{\prime}(0)|^2+2\bigg{\langle} g_a^{-1}\cdot\big{(}[d{\xi_{0}}^{\prime}(0),{\xi_{0}}^{\prime}(0)]+d{\xi_{0}}^{\prime\prime}(0)\big{)}(\Theta_{\infty}^*), g_a^{-1}\cdot d\Theta_{\infty}^* \bigg{\rangle}\bigg{)}\phi_a\\
=&\int_{B_{\eta,a}}\bigg{(}2|d{\xi_{0}}^{\prime}(0)|^2+2\bigg{\langle} \big{(}[d{\xi_{0}}^{\prime}(0),{\xi_{0}}^{\prime}(0)]+d{\xi_{0}}^{\prime\prime}(0)\big{)}(\Theta_{\infty}^*),  d\Theta_{\infty}^* \bigg{\rangle}\bigg{)}\phi_a.
\end{split}
\end{eqnarray}
Using \eqref{lstab}, and summing over $a$ in both \eqref{mlim} and \eqref{elim}, we achieve
\begin{eqnarray}
\begin{split}
\label{staleft}
&\lim_{l\to \infty}\frac{d^2}{dt^2}\bigg{\rvert}_{t=0}\int_{B_{5\gamma^{-1}/6}(y_*)}|\nabla_{A_l}\big{(}\exp(\xi_l(t))\Theta_l\big{)}|^2\\
=&\sum_a \int_{B_{\eta,a}}\bigg{(}2|d{\xi_{0}}^{\prime}(0)|^2+2\bigg{\langle} \big{(}[d{\xi_{0}}^{\prime}(0),{\xi_{0}}^{\prime}(0)]+d{\xi_{0}}^{\prime\prime}(0)\big{)}(\Theta_{\infty}^*),  d\Theta_{\infty}^* \bigg{\rangle}\bigg{)}\phi_a\\
=&2\int_{B_{5\gamma^{-1}/6}(y_*)}|d\big{(}\xi_0^{\prime}(0)\big{)}|^2+2\int_{B_{5\gamma^{-1}/6}(y_*)}\bigg{\langle} \big{(}[d\big{(}\xi_0^{\prime}(0)\big{)},\xi_0^{\prime}(0)]+d\big{(}\xi_0^{\prime\prime}(0)\big{)}(\Theta^*_\infty),d\Theta^*_\infty \bigg{\rangle}.
\end{split}
\end{eqnarray}
Note the right hand side of the above equality is exactly the first term on the right hand side of \eqref{part2}. Now applying \eqref{staleft}, the stability of $\Theta_l$, \eqref{erreta} as well as the arbitrariness of $\eta$ to \eqref{part2}, we conclude that
\begin{eqnarray}
\begin{split}
\frac{d^2}{dt^2}\bigg{\rvert}_{t=0}\int_{B_{5\gamma^{-1}/6}(y_*)}|d\big{(}\exp(\xi(t))\cdot\Theta^*_\infty\big{)}|^2\ge 0.
\end{split}
\end{eqnarray}
Thus finishes the proof of the fact that $\Theta^*_\infty$ is stable in $B_{5\gamma^{-1}/6}(y_*)$. We hence complete the proof of Proposition \ref{shs}.
\end{proof}
Now we shall give a detailed definition of $\text{deg}(\Theta)\rvert_{\partial B_{r_l}(y_l)}$ that appears in \eqref{trickysphere} as promised. First rescale $B_{\gamma^{J_l-1}/2}(y_l)$ to $B_{10\gamma}(y_l)$ and identify all $y_l$ with $y_*$. Choose a Vitali cover $\mathscr{B}$ of $B_{5\gamma}(y_*)\backslash\{x_*\}$ such that for every ball $B\in \mathscr{B}$ if $B\cap \partial B_r(x_*)\neq \emptyset$ then $\text{rad}(B)\ge r/10$, for all $r\in [0,\gamma]$. Let us a sub-collection of $\mathscr{B}$ that covers $A_{\gamma^2,2\gamma^2}(y_*)$ denoted by $\{B_a\}_{a=1}^N$ where $N$ is a universal constant. In addition, by Claim \ref{key} we could assume without loss of generality that for every sphere $\partial B_{r}(y_*)\subseteq A_{\gamma^2,2\gamma^2}(y_*)$ satisfies 
\begin{eqnarray}
\begin{split}
\label{awaysing}
\text{dist}(\partial B_{r}(y_*),\text{Sing}(\Theta_l))\ge\gamma^3/2
\end{split}
\end{eqnarray}
upon further decreasing $\gamma$ if necessary. Next, choose for each $a$ an Uhlenbeck gauge $\sigma_{a,l}$ of $A_l$ with connection form $A^*_{a,l}$. By the same argument as before, we have 
\begin{eqnarray}
\begin{split}
A^*_{a,l}\to A^*_{a,\infty}\equiv0\ \text{in $W^{1,2}(B_a)$},
\end{split}
\end{eqnarray}
Denote by $\Theta^*_{a,l}$ the trivialization of $\Theta_l$ under $\sigma_{a,l}$. Then by the same argument for obtaining \ref{modulus0}, one see that there exists a $g_a\in \text{SU}(2)$ such that 
\begin{eqnarray}
\begin{split}
\label{modu3}
\Theta^*_{a,l}\to g_a\cdot \Theta^*_{\infty}\ \text{in $W^{1,2}(B_a)$}.
\end{split}
\end{eqnarray}
For future purpose define
\begin{eqnarray}
\begin{split}
\label{taug}
\tau_{a,l}\eqqcolon \sigma_{a,l}\cdot g_a^{-1}.
\end{split}
\end{eqnarray}
Now let $\{\phi_a\}_a$ be a partition of unity subordinate to $\{B_a\}_a$, such that $\sum_a\phi_a\equiv 1$ on $A_{11\gamma^2/10,19\gamma^2/10}(y_*)$. By pigeonhole principle, there exists some $r_0\in [13\gamma^2/10,17\gamma^2/10]$ such that the following holds for all $l$:
\begin{eqnarray}
\begin{split}
\label{1119}
&\int_{\partial B_{r_0}(y_*)}|F_{A_l}|^2\le Cr_0^{-1}\int_{A_{11\gamma^2/10,19\gamma^2/10}(y_*)}|F_{A_l}|^2< C(\gamma)\gamma^{K_l},\\
&\int_{\partial B_{r_0}(y_*)}|\nabla^2_{A_l}\Theta_l|^2\le \sum_a\int_{\partial B_{r_0}(y_*)\cap B_a}\big{(}|A^*_{a,l}|^4+|d\Theta^*_{a,l}|^4\big{)}\phi_a\\
&\le Cr_0^{-1} \sum_a\int_{A_{11\gamma^2/10,19\gamma^2/10}(y_*)}\big{(}|A^*_{a,l}|^4+|d\Theta^*_{a,l}|^4\big{)}\phi_a\le C(\gamma),
\end{split}
\end{eqnarray}
where we have used \eqref{awaysing} as well as Theorem \ref{apr} in the last inequality. The first inequality above allows us to apply the Uhlenbeck's Coulomb gauge theorem (see \cite{U82a}, Theorem 1.3) to the bundle $P_l$ restricted on two hemispheres of $\partial B_{r_0}(y_*)$ (with definite overlap): $(P_l,A_l)\rvert_{\partial^{\pm} B_{r_0}(y_*)}$, to obtain the Coulomb gauges $\{U_l^{\pm}\}_{l}$ such that
\begin{eqnarray}
\begin{split}
\|{U_l^{\pm}}^*A_l\|_{W^{1,2}(\partial^{\pm} B_{r_0}(y_*))}\le C_\gamma\|F_{A_l}\|_{L^2(\partial^{\pm} B_{r_0}(y_*))}<C\gamma^{K_l}\to 0\ \text{as $l\to \infty$}.
\end{split}
\end{eqnarray}
Define $u_l$ by $U_l^+=U^-_l\cdot u_l$. The above then implies that $\|u_l\|_{W^{2,2}(\partial^{+} B_{r_0}(y_*)\cap \partial^{-} B_{r_0}(y_*))}\to 0$. Now by the compact Sobolev embedding $C^0\hookrightarrow W^{2,2}$ in dimension $3$, we see that $u_l\to e_l\in \text{SU}(2)$ uniformly on $\partial^{+} B_{r_0}(y_*)\cap \partial^{-} B_{r_0}(y_*)$. Hence for $l$ sufficiently large, by a standard cutoff argument one could patch up $U_l^+$ with $U^-_l\cdot e_l$ continuously in $W^{2,2}$ to obtain a section $U_l$ defined globally on $\partial B_{r_0}(y_*)$, satisfying
\begin{eqnarray}
\begin{split}
\label{circga}
U_l^*A_l\eqqcolon A_{U_l}\to 0\ \text{in $W^{1,2}(\partial B_{r_0}(y_*))$}.
\end{split}
\end{eqnarray}
Define $\Theta_{U_l}$ to be the trivialization of $\Theta_l$ under $U_l$. By the second inequality in \eqref{1119} and the Sobolev embedding $C^0\hookrightarrow W^{2,2}$ in dimension $3$, $\Theta_{U_l}$ is a $W^{2,2}$ (and hence continuous) mapping from $\partial B_{r_0}(y_*)\cong\mathbb{S}^3$ into $\text{SU}(2)\cong \mathbb{S}^3$, and therefore has a well-defined degree.
\begin{definition}
\label{trsp}
For all sufficiently large $l$ we define $\text{deg}(\Theta_l)\rvert_{\partial B_{r_0}(y_*)}$ to be $\text{deg}(\Theta_{U_l})$. 
\end{definition}
\begin{remark}
It is an immediate consequence of \eqref{circga} that the definition of $\text{deg}(\Theta_l)\rvert_{\partial B_{r_0}(y_*)}$ is independent of the choice of the gauges $\{U_l\}$ for sufficiently large $l$. In other words, given two sets of gauge $\{U_l\}_l$ and $\{U^{\prime}_l\}_l$ such that both satisfies \eqref{circga}, then for $l$ large enough, $\text{deg}\Theta_{U_l}=\text{deg}\Theta_{U_l^{\prime}}$.
\end{remark}
By the assumption we make at the beginning of the proof of Lemma \ref{singannbdd}, it is now straightforward to verify that a sphere fulfills the requirements in \eqref{trickysphere} exists, by adjusting the parameter $J_l$. Finally we rescale $\partial B_{r_0}(y_*)$ back to $\partial B_{r_l}(y_*)$ to obtain the desired sphere described in \eqref{trickysphere}. 
\medskip

Next, we continue to finish the proof of Lemma \ref{singannbdd}. By \eqref{modu3}, the second inequality in \eqref{1119}, and the expression introduced in \eqref{taug}, as well as the compact Sobolev embedding $W^{1,3}\hookrightarrow W^{2,2}$ in dimension $3$, we obtain the following convergence upon a subsequence:
\begin{eqnarray}
\begin{split}
\label{w22c}
\tau_{a,l}^*\Theta_l\to \Theta^*_\infty\ \text{weakly in $W^{2,2}(B_a\cap \partial B_{r_0}(y_*))$ and strongly in $W^{1,3}(B_a\cap\partial B_{r_0}(y_*))$}.
\end{split}
\end{eqnarray}
Now let us define the gauge transform $V_{a,l}$ by $\tau_{a,l}\cdot V_{a,l}=U_l$. Due to \eqref{circga} and the following convergence:
\begin{eqnarray}
\begin{split}
\tau_{a,l}^*A_l\to 0\ \text{in $W^{1,2}(B_a\cap\partial B_{r_0}(y_*))$},
\end{split}
\end{eqnarray}
it is easy to see that
\begin{eqnarray}
\begin{split}
V_{a,l}^{-1}dV_{a,l}\to 0\ \text{in $W^{1,2}(B_a \cap\partial B_{r_0}(y_*)))$},
\end{split}
\end{eqnarray}
which implies that for each $a$ there exists some $h_a\in\text{SU}(2)$, such that $V_{a,l}\to h_a$ strongly in $W^{2,2}(B_a\cap\partial B_{r_0}(y_*))$. But on the other hand, \eqref{w22c} together with the expression $\tau_{a,l}\cdot V_{a,l}=U_l$ implies that
\begin{eqnarray}
\begin{split}
U_l^*\Theta_l\to h_a^{-1}\cdot \Theta_\infty^*\ \text{weakly in $W^{2,2}(B_a\cap \partial B_{r_0}(y_*))$ and strongly in $W^{1,3}(B_a\cap\partial B_{r_0}(y_*))$}.
\end{split}
\end{eqnarray}
Thanks to the connectedness of $\mathbb{S}^3$, there exists some $h\in \text{SU}(2)$ such that $h_a\equiv h$ for all $a$. Applying the convergence $\Theta_{U_l}=U_l^*\Theta_l\to h^{-1}\cdot \Theta_\infty^*$ in $W^{1,3}(\partial B_{r_0}(y_*)\cap \partial B_{r_0}(y_*))$, and scaling $r_0$ back to $\gamma^2$ as in \eqref{trickysphere2}, we obtain
\begin{eqnarray}
\begin{split}
\label{degcon}
\text{deg}(\Theta_l)\rvert_{\partial B_{\gamma^2}(y_*)}\to \text{deg}(h^{-1}\cdot \Theta^*_\infty)=\text{deg}(\Theta^*_\infty).
\end{split}
\end{eqnarray}
Combining \eqref{degcon} with \eqref{trickysphere2}, we see that $\text{deg}(\Theta^*_{\infty})\rvert_{\partial B_{\gamma^2}(y_*)}\neq 0$. Hence there exists $w_*\in \text{Sing}(\Theta_{\infty})\cap B_{\gamma^2}(y_*)$. In addition, \eqref{zstar} implies that $z_*\in \text{Sing}(\Theta_{\infty})$. Therefore, $w_*$ and $z_*$ both are singularities of $\Theta_{\infty}$, a stable stationary harmonic map defined on $B_{5\gamma^{-1}/6}(y_*)$ (see Proposition \ref{shs}), and that $\gamma^3/2< \text{dist}(w_*,z_*)< 3\gamma$ according to \eqref{zstar}. Nevertheless, this is impossible due to a result by \cite{LW06}, saying that on $B_{1/2}$ two singularities of a stationary-stable harmonic map defined on $B_1$ has distance bounded below by some universal number $\delta_*$, provided $\gamma<\delta_*/10$ (see also Theorem 3.4.12 of \cite{LW08}). Thus we arrive at a contradiction. Note this also helps us fix $\gamma$. Now the proof of Lemma \ref{singannbdd} is completed.
\end{proof}
Now one sees that Theorem \ref{annbubsing} is an immediate consequence of a combination of applications of Theorem \ref{apr}, Theorem \ref{4localeps}, and Lemma \ref{singannbdd}. 
\end{proof}
\bigskip

\section{Approximation by smooth connections}
\label{asc}
In this section, we shall prove Theorem \ref{main} for general $W^{1,2}$ connections, by using an approximation by smooth connections. Let us fix an arbitrary $A\in W^{1,2}$. Firstly, by a density result (see, for example, \cite{kessel08}), smooth connections are dense in $W^{1,2}$-connection space on compact $4$-manifold. Hence there exists a sequence of smooth connections $\{A_i\}_i$ such that
\begin{eqnarray}
\begin{split}
\label{w12con}
A_i\to A\ \text{in $W^{1,2}$.}
\end{split}
\end{eqnarray}
Due to the smoothness of the bundle $P$, we might find a global reference frame on $P$ denoted by $\Theta_1$, such that $\Theta_1$ is smooth away from $k_0$ points in $M$, where $k_0$ equals the Chern number $c_2$ of $P$. For each $i$, let $\Theta_i$ be a Coulomb minimizer associated to $A_i$, with connection form $\hat{A}_i$. Let us fix a ball $B_*\subseteq M$ with $10\text{rad}(B_*)<r(M)$, such that on $2B_*$ (the concentric ball with two times radius) the following holds upon passing to a subsequence of $\{i\}$:
\begin{eqnarray}
\begin{split}
\label{bstar}
\int_{2B_*}|F_{A_i}|^2<\epsilon_0/2,\ \int_{2B_*}|d\Theta_i|^4+\int_{2B_*}|\nabla^2\Theta_i|^2<\infty,\ \forall i.
\end{split}
\end{eqnarray}
In other words, $2B_*$ is a ball where not only the curvature is small, but contains no singularity of $\Theta_i$ as well. We might find an Uhlenbeck gauge $A_{i}^*$ over $B_*$ satisfying \eqref{uhgauge} for each $i$. By Uhlenbeck's weak compactness result in \cite{U82b} (see also Theorem 6.1 and Remark 6.2, \cite{W04}), we see that the following convergences holds up to passing to subsequence:
\begin{eqnarray}
\begin{split}
\label{sl2}
A_{i}^*\to A^*\ \text{weakly in $W^{1,2}(\frac{3}{2}B_*)$ and strongly in $L^2(\frac{3}{2}B_*)$}.
\end{split}
\end{eqnarray}
Let $\Theta_{i}^*$ be the trivialization of $\Theta_i$ under $A_{i}^*$. By \eqref{sl2} and Theorem \ref{cpt}, there exists some $\Theta^*$ such that the following holds:
\begin{eqnarray}
\begin{split}
\label{sl22}
\Theta^*_{i}\to \Theta^*\ \text{in $W^{1,2}(B_*)$}.
\end{split}
\end{eqnarray}
In addtion, we let $V_{i}\in W^{1,2}(B_*,\text{SU}(2))$ be the gauge transform such that
\begin{eqnarray}
\begin{split}
V_{i}^*A^*_{i}=\hat{A}_i.
\end{split}
\end{eqnarray}
From Theorem \ref{apr}, we have
\begin{eqnarray}
\begin{split}
s_{\Theta_i}(x)>c_0\text{rad}(B_*),\ \forall x\in B_*.
\end{split}
\end{eqnarray}
for some universal constant $c_0$. 
\begin{claim}
\label{uniwl}
There exists $C$ such that it holds:
\begin{eqnarray}
\begin{split}
&\sup_i\|V_{i}\|_{W^{2,2}(B_*)}\le C,\\
&\sup_i(\|\hat{A}_i\|_{L^4(B_*)}+\|\hat{A}_i\|_{W^{1,2}(B_*)})\le C.
\end{split}
\end{eqnarray}
\end{claim}
\begin{proof}
Note that the second uniform bound follows immediately from the assumption \eqref{bstar} and Theorem \ref{apr}. So let us focus on proving the first uniform bound. Using the second uniform bound, \eqref{sl2}, \eqref{uhgauge}, the gauge transformation formula, and the Sobolev embedding, we could estimate as follows:
\begin{eqnarray}
\begin{split}
&\|dV_{i}\|_{L^4(B_*)}=\|V_{i}^{-1}dV_{i}\|_{L^4(B_*)}\le C_0\|\hat{A}_{i}\|_{W^{1,2}(B_*)}\\
+&\|V_{i}^{-1}A_{i}^*V_{i}\|_{L^4(B_*)}\le C_0C+C_{Uh}\|F_A\|_{L^2(B_*)}\le C_1.
\end{split}
\end{eqnarray}
Using this, we further estimate
\begin{eqnarray}
\begin{split}
&\|\nabla^2V_{i}\|_{L^2(B_*)}=\|\nabla(V_{i}^{-1}dV_{i})\|_{L^2(B_*)}+\|\nabla V_{i}\otimes V_{i}^{-1}dV_{i}\|_{L^2(B_*)}\\
\le& \|\hat{A}_i\|_{W^{1,2}(B_*)}+\|V_{i}^{-1}A_{i}^*V_{i}\|_{W^{1,2}(B_*)}+\|dV_{i}\|^2_{L^4(B_*)}\\
\le &C+\|A_{i}^*\|_{L^2(B_*)}+\|\nabla A_{i}^*\|_{L^2(B_*)}+2\|A_{i}^*\|_{L^4(B_*)}\|dV_{i}\|_{L^4(B_*)}+C_1^2\\
\le& C+C_3C_{Uh}\|F_A\|_{L^2(B_*)}+2C_1C_{Uh}\|F_A\|_{L^2(B_*)}+C_1^{2}\le C_4.
\end{split}
\end{eqnarray}
\end{proof}
On the one hand, by the first uniform bound in Claim \ref{uniwl}, the convergence \eqref{sl2}, and up to passing to a subsequence, we can assume 
\begin{eqnarray}
\begin{split}
V_{i}\rightharpoonup V\ \text{in $W^{2,2}(B_*)$.}
\end{split}
\end{eqnarray}
Therefore we have
\begin{eqnarray}
\begin{split}
\label{vae}
V_{i}^*A_{i}^*\to V^*A^*\ \text{a.e.}
\end{split}
\end{eqnarray}
On the other hand, by the second uniform bound in Claim \ref{uniwl}, and up to passing to subsequence, we can assume
\begin{eqnarray}
\begin{split}
\label{w4con}
V_i^*A_i^*=\hat{A}_i\rightharpoonup \hat{A}\ \text{in both $L^4(B_*)$ and $W^{1,2}(B_*)$}.
\end{split}
\end{eqnarray}
By \eqref{vae} and \eqref{w4con}, the following must hold:
\begin{eqnarray}
\begin{split}
\label{dtheta4}
\hat{A}=V^*A^*\in W^{1,2}(B_*)\cap L^4(B_*),\ d\Theta_{i}^*\rightharpoonup d\Theta^*\ \text{in $L^4(B_*)$}.
\end{split}
\end{eqnarray}
Then we could estimate $\|\hat{A}\|_{\mathcal{L}^{4,\infty}}$. Firstly, choose any $\alpha>0$ (this will be the height that appears in the $L^{4,\infty}$ quasi-norm). Note that there exists a finite collection of balls given by $\{B_{*,\beta}\}_{\beta}$ such that each $B_{*,\beta}$ satisfies \eqref{bstar}, that forms a cover of the sub-level set $\{s_{\hat{A}_i}<\alpha\}$:
\begin{eqnarray}
\begin{split}
\label{beta1}
 \{s_{\hat{A}_i}<\alpha\}\subseteq \bigcup_{\beta}B_{*,\beta}.
\end{split}
\end{eqnarray}
Using the $L^4$-weak convergence in \eqref{dtheta4} on each $B_{*,\beta}$, as well as the norm-semicontinuity under weak-convergence, we obtain for all large $i$ that
\begin{eqnarray}
\begin{split}
\label{beta2}
\{x\in B_{*,\beta}:\ s_{\hat{A}}<\alpha\}\subseteq  \{x\in B_{*,\beta}:\ s_{\hat{A}_i}<\alpha\}.
\end{split}
\end{eqnarray}
Combining \eqref{beta2} with the $\mathcal{L}^{4,\infty}$-estimate on each $\hat{A}_i$, we arrive at the following:
\begin{eqnarray}
\begin{split}
\label{desi}
\|\hat{A}\|^4_{\mathcal{L}^{4,\infty}(M)}\le C(M,\Lambda).
\end{split}
\end{eqnarray}
Finally, using the convergences in \eqref{w4con}, we see that $\hat{A}$ is weakly-Coulomb. Define $s_i:M\to P$ by $\Theta_i\circ s_i=1$. Now one could verify in the same way as in \eqref{patchupnow} that, upon a subsequence, $s_i\to s_{0}$ weakly in $W^{1,2}$ for some section $s_0: M\to P$, and that $s_{0}^*A=\hat{A}$. To sum up, we find a global section $s_{0}$ for $A$ such that $d^*(s_0^*A)=0$ weakly. Moreover, away from finitely (and indeed controllably) many singularities $s_0^*A$ is $L^4$ integrable. Most significantly, $s_0^*A$ admits the desired $\mathcal{L}^{4,\infty}$ estimate. This completes the proof of Theorem \ref{main}. 
\bigskip

\section*{Appendix}
\label{appendix}
We prove the fact mentioned in Remark \ref{clanew}; namely, the new space $\mathcal{L}^{4,\infty}$ possesses the property of being weaker than $L^4$ while stronger than $L^{4-\epsilon}$ for all $\epsilon>0$. The following lemma is sufficient to conclude the inequalities \eqref{clanew1}:
\begin{lemma}
\label{interpo}
Let $u$ be a measurable function defined on a bounded Riemannian domain $(U,g)$ with the diameter $\text{diam}(U)\le 1$, the sectional curvature $\text{sec}_g\le1/100$, and the injectivity radius $\text{inj}_g\ge 10$. Then there exist $C_0>0$, and $C(\epsilon)>0$ for every small $\epsilon>0$, such that:
\begin{eqnarray}
\begin{split}
\label{eleinterpo}
&\int_{U}|u|^{4-\epsilon}\le C(\epsilon)\|u\|^4_{\mathcal{L}^{4,\infty}(U)},\\
&\|u\|_{\mathcal{L}^{4,\infty}(U)}\le C_0\|u\|_{L^{4}(U)}.
\end{split}
\end{eqnarray}
\end{lemma}
\begin{proof}
Let us assume the right hand side of the first inequality in \eqref{eleinterpo} to be finite. Fix $\epsilon$. Consider a Vitali cover of $U$ given by $\{B_{r_i}(x_i)\}_i$, where $r_i=s_{u}(x_i)$ and $\{B_{r_i/5}(x_i)\}_i$ are disjoint. We have
\begin{eqnarray}
\begin{split}
\label{kepsilon}
&\int_U |u|^{4-\epsilon}\le \sum_i \int_{B_{r_i}(x_i)}|u|^{4-\epsilon}\le \sum_i\big{(} \int_{B_{r_i}(x_i)}|u|^4\big{)}^{\frac{4-\epsilon}{4}}\big{(}\int_{B_{r_i}(x_i)}1\big{)}^{\frac{\epsilon}{4}}\\
&\le  C\sum_i r_i^{\epsilon}=C\sum_{k\ge 0}\sum_{r_i\in [2^{-k-1},2^{-k})}r_i^{\epsilon}\le C\sum_{k\ge 0} N_{k,\epsilon}2^{-k\epsilon},
\end{split}
\end{eqnarray}
where $N_{k,\epsilon}$ is the number of $r_i$s that lie in $ [2^{-k-1},2^{-k})$. From Definition \ref{rads} it is easy to verify that
\begin{eqnarray}
\begin{split}
\label{comp}
\frac{r_{x_j}}{4}\le r_{x_i}\le 4r_{x_j},
\end{split}
\end{eqnarray}
whenever $B_{r_i}(x_i)\cap B_{r_j}(x_j)$ is non-empty. Due to \eqref{comp}, for each $r_i\in [2^{-k-1},2^{-k})$ it holds
\begin{eqnarray}
\begin{split}
B_{\frac{r_i}{5}}(x_i)\subseteq \{x: s_{u}(x)<5r_i\}\subseteq \{x: s_{u}(x)<5\cdot 2^{-k}\}.
\end{split}
\end{eqnarray}
Thus by the disjointness of $\{B_{r_i/5}(x_i)\}_i$ we have
\begin{eqnarray}
\begin{split}
N_{k,\epsilon}2^{-4(k+1)}5^{-4}\le \sum_{r_i\in [2^{-k-1},2^{-k})}(\frac{r_i}{5})^4\le |\{x:s_u(x)<5\cdot 2^{-k}\}|.
\end{split}
\end{eqnarray}
By definition, we have
\begin{eqnarray}
\begin{split}
(5\cdot 2^{-k})^{-4}|\{x:s_u(x)<5\cdot 2^{-k}\}|\le \|u\|^4_{\mathcal{L}^{4,\infty}(U)}.
\end{split}
\end{eqnarray}
Therefore the following holds:
\begin{eqnarray}
\begin{split}
N_{k,\epsilon}\le 10^8 \|u\|^4_{\mathcal{L}^{4,\infty}(U)}.
\end{split}
\end{eqnarray}
Insert this back to \eqref{kepsilon} we obtain
\begin{eqnarray}
\begin{split}
\int_U |u|^{4-\epsilon}\le C \|u\|^4_{\mathcal{L}^{4,\infty}(U)}\sum_{k\ge 0} 2^{-k\epsilon}\le C(\epsilon)\|u\|^4_{\mathcal{L}^{4,\infty}(U)}.
\end{split}
\end{eqnarray}
This concludes the first inequality in \eqref{eleinterpo}. Now we start to prove the second, whose right hand side is also assumed to be finite, for otherwise nothing needs to be proved. For each $\alpha>0$ cover the set $\{x:s_u(x)\le \alpha/6\}$ by a (finite) Vitali cover $\{B_{\alpha}(x_i)\}_{i=1}^M$ with $\{B_{\alpha/3}(x_i)\}_{i=1}^M$ being disjoint. Using the disjointness and Definition \ref{4rad}, we have
\begin{eqnarray}
\begin{split}
\label{21}
\int_U|u|^4\ge \sum_{i}\int_{B_{\frac{\alpha}{3}}(x_i)}|u|^4\ge M.
\end{split}
\end{eqnarray}
On the other hand by the covering property we have
\begin{eqnarray}
\begin{split}
M\cdot{\alpha}^{4}\ge |\{x:s_u(x)\le \alpha/6\}|.
\end{split}
\end{eqnarray}
Therefore, by Definition \ref{4rad} we obtain
\begin{eqnarray}
\begin{split}
\label{22}
M\ge {\alpha}^{-4}\cdot |\{x:s_u(x)\le \alpha/6\}|\ge 6^{-4} \|u\|_{\mathcal{L}^{4,\infty}(U)}^4.
\end{split}
\end{eqnarray}
Combining \eqref{22} with \eqref{21}, we see that
\begin{eqnarray}
\begin{split}
\|u\|^4_{\mathcal{L}^{4,\infty}(U)}\le6^4\int_U|u|^4,
\end{split}
\end{eqnarray}
which proves the second inequality in \eqref{eleinterpo}.
\end{proof}
\bigskip

\bibliographystyle{plain}
\bibliography{4dim}

\end{document}